\documentclass{article}

\usepackage[preprint]{neurips_2023}

\usepackage[utf8]{inputenc} % allow utf-8 input
\usepackage[T1]{fontenc}    % use 8-bit T1 fonts
\usepackage{hyperref}       % hyperlinks
\usepackage{url}            % simple URL typesetting
\usepackage{booktabs}       % professional-quality tables
\usepackage{amsfonts}       % blackboard math symbols
\usepackage{nicefrac}       % compact symbols for 1/2, etc.
\usepackage{microtype}      % microtypography
\usepackage{xcolor}         % colors
\usepackage{amsmath,amsfonts,bm}

\def\eqref#1{equation~\ref{#1}}
\def\1{\bm{1}}

\DeclareMathAlphabet{\mathsfit}{\encodingdefault}{\sfdefault}{m}{sl}
\SetMathAlphabet{\mathsfit}{bold}{\encodingdefault}{\sfdefault}{bx}{n}

\usepackage{wrapfig}

\usepackage{amsmath, amsthm, amsfonts, amssymb,mathrsfs,cite,bm}
\usepackage{stackengine}
\usepackage[makeroom]{cancel}
\usepackage{graphicx}
\usepackage{enumitem}
\usepackage{subcaption}

\usepackage{hyperref}
\usepackage{url}
\newcommand{\differential}{{\rm{d}}}

\newcommand{\prox}{{\rm{prox}}}
\newcommand{\diag}{{\rm{diag}}}

\usepackage{stackengine} 
\DeclareRobustCommand*{\oast}{\ensuremath{\stackMath\mathbin{\stackinset{c}{0ex}{c}{0ex}{\ast}{\bigcirc}}}}
\renewcommand{\qedsymbol}{\hfill\ensuremath{\blacksquare}}

\newtheorem{lemma}{Lemma}
\newtheorem{proposition}{Proposition}

\newtheorem{theorem}{Theorem}
\newtheorem{remark}{Remark}

\title{Wasserstein Consensus ADMM}

\author{%
Iman Nodozi\thanks{Department of Electrical and Computer Engineering, University of California, Santa Cruz, \texttt{inodozi@ucsc.edu}}
\And
  \textbf{Abhishek Halder}\thanks{Department of Applied Mathematics, University of California, Santa Cruz, \texttt{ahalder@ucsc.edu}}
}

\begin{document}

\maketitle

\begin{abstract}
  We introduce Wasserstein consensus alternating direction method of multipliers (ADMM) and its entropic-regularized version: Sinkhorn consensus ADMM, to solve measure-valued optimization problems with convex additive objectives. Several problems of interest in stochastic prediction and learning can be cast in this form of measure-valued convex additive optimization. The proposed algorithm generalizes a variant of the standard Euclidean ADMM to the space of probability measures but departs significantly from its Euclidean counterpart. In particular, we derive a two layer ADMM algorithm wherein the outer layer is a variant of consensus ADMM on the space of probability measures while the inner layer is a variant of Euclidean ADMM. The resulting computational framework is particularly suitable for solving Wasserstein gradient flows via distributed computation. We demonstrate the proposed framework using illustrative numerical examples.
\end{abstract}

\section{Introduction}\label{sec:Intro}
Let $\mathcal{P}_{2}(\mathcal{X})$ denote the space of Borel probability measures over $\mathcal{X}\subseteq\mathbb{R}^{d}$ with finite second moments. Let $\mathcal{P}_{2,{\rm{ac}}}(\mathcal{X}) := \{\mu \in \mathcal{P}_{2}(\mathcal{X}) \mid \mu\:\text{is absolutely continuous w.r.t. the Lebesgue
measure}\}\subset\mathcal{P}_{2}(\mathcal{X})$. A probability measure $\mu\in\mathcal{P}_{2,{\rm{ac}}}(\mathcal{X})$ admits a joint probability density function (PDF) $\rho(\bm{x}) := \frac{\differential\mu}{\differential\bm{x}}$ such that $\rho \geq 0$ for all $\bm{x}\in\mathcal{X}$ and $\int_{\mathcal{X}}\rho\:\differential\bm{x} = 1$.

We consider measure-valued optimization problems of the form 
\begin{align}
\underset{\mu\in\mathcal{P}_{2}(\mathcal{X})}{\arg\inf} F(\mu)
\label{AdditiveOptimizationMeasure}   
\end{align}
where the objective $F$ is expressible as a sum: $F(\mu) = F_{1}(\mu) + F_{2}(\mu) + \hdots + F_{n}(\mu)$ for some finite $n\in\mathbb{N}, n>1$. We suppose that the summand functionals $F_{i}:\mathcal{P}_{2}(\mathcal{X})\mapsto (-\infty,+\infty]$ are proper lower semi-continuous (lsc), and convex along the generalized geodesics w.r.t. the 2-Wasserstein distance \citep[Ch. 9]{ambrosio2008gradient} for all $i\in[n]$. We will review the relevant technical preliminaries in Sec. \ref{sec:preliminaries}. The purpose of this work is to design distributed algorithms to solve such measure-valued optimization problems with additive objective.

%and the functionals $F_{i}:\mathcal{P}_{2}(\mathbb{R}^{d})\mapsto \mathbb{R}$ {\red{are convex}} for all $i\in[n]$. If the optimization in (\ref{AdditiveOptimizationMeasure}) is instead over $\mathcal{P}_{2,{\rm{ac}}}(\mathbb{R}^{d})$, then we can equivalently write it as 
% \begin{align}
% \underset{\rho}{\arg\inf} \: F_{1}(\rho) + F_{2}(\rho) + \hdots + F_{n}(\rho)
% \label{AdditiveOptimizationPDF}    
% \end{align}
% where the decision variable $\rho$ is a joint PDF over $\mathbb{R}^{d}$ with finite second moment. 

% Problems of the form (\ref{AdditiveOptimizationMeasure}) arise in several contexts in statistics, machine learning, and control theory. 

Instances of (\ref{AdditiveOptimizationMeasure}) are often encountered in machine learning \citep{chizat2018global,mei2018mean,sirignano2020mean, zhang2018policy, domingo2020mean,bunne2022proximal} and control \citep{caluya2019gradient,caluya2021wasserstein}. Most existing algorithms \citep{peyre2015entropic,benamou2016augmented,carlier2017convergence,wibisono2018sampling,alvarez2021optimizing,mokrov2021large,kent2021modified,carrillo2022primal,fan2022variational,wang2022accelerated} for this class of problems require centralized computation; relatively few works \citep{dvurechenskii2018decentralize,arque2022approximate} are available on solving specific instances of (\ref{AdditiveOptimizationMeasure}) via distributed computation. The main contribution of this work is to deduce a distributed algorithm for solving (\ref{AdditiveOptimizationMeasure}) by generalizing the Euclidean consensus ADMM to Wasserstein spaces. Our proposed algorithm realizes measure-valued operator splitting \citep{bowles2015weak,bernton2018langevin,gallouet2017jko} but allows explicit distributed updates.

\noindent\textbf{Motivation and Contributions.} While problem (\ref{AdditiveOptimizationMeasure}) appears across many disciplines, one particular motivation behind our work is to numerically solve the \emph{transient} solutions for measure-valued PDE initial value problems (IVPs). These PDEs are often nonlinear and nonlocal (see e.g., the second case study in Sec. \ref{sec:Experiments}), and difficult to solve scalably via traditional scientific computing methods such as finite difference. However, it is known that the flow induced by such PDE IVPs can often be seen \citep[Ch. 11]{ambrosio2008gradient}, \citep{santambrogio2017euclidean} as gradient descent of a suitable free energy Lyapunov functional $F(\mu)$ w.r.t. the 2-Wasserstein metric over the space of measures. Then, high-level idea is to leverage this variational reformulation to compute the transient solutions for such IVPs by numerically performing Wasserstein gradient descent on (\ref{AdditiveOptimizationMeasure}). 

The specific idea in this work is to further recognize that the 
 functional $F$ in practice has an additive structure $F(\cdot)=F_1(\cdot)+\hdots+F_{n}(\cdot)$, which comes from different spatial operators (e.g., advection, interaction, diffusion) appearing in the PDE. One of our contribution here is to show that it is possible to leverage this additive structure in 
 $F$ to generalize the Euclidean ADMM to the Wasserstein space. The proposed algorithm can then be seen as a nonlinear superposition principle where different computers solve different (simpler) PDE IVPs by performing proximal update on a modified version of $F_i$, and then combine the resulting updates in a nonlinear manner. Historically, this point of view is very close to the origin of operator splitting \citep{glowinski1989augmented,glowinski2016some} in the PDE community that motivated the development of ADMM \citep{gabay1976dual}, albeit in the finite-dimensional setting.

 We clarify here that while augmented Lagrangian methods for infinite dimensional problems have been investigated before, they appeared in the Hilbert spaces \citep{ito1990augmented} or reflexive Banach spaces \citep{butnariu2000totally,kanzow2018augmented}. In contrast, the definition (\ref{WassAugLagrangian}) for the Wasserstein augmented Lagrangian is novel. Our development is also different from the (standard) augmented Lagrangian for Wasserstein gradient flow as in \citep[equation 2.12]{benamou2016augmented}, and directly works on the Wasserstein space.

%%%%%%%%%%%%%%%%%%%%%%%%%%%%%%%%%%%%%%%%%%%%%%%%%%%%%%%%%%%%%%%%%%%%%%%%%%%%%%%%%%

\section{Preliminaries}\label{sec:preliminaries}
\textbf{Wasserstein space and Wasserstein gradient flow.} 
Let $\mathcal{B}(\mathcal{X})$ denote the Borel $\sigma$-field over $\mathcal{X}\subseteq\mathbb{R}^{d}$. For $\mu\in\mathcal{P}_{2}(\mathcal{X})$, and for any measurable map $T$ defined on $\left(\mathcal{X},\mathcal{B}(\mathcal{X})\right)$, let $T_{\#}\mu$ denote the pushforward a.k.a. transport of the probability measure $\mu$ via $T$.

For $\mathcal{X},\mathcal{Y}\subseteq\mathbb{R}^{d}$, the \emph{squared 2-Wasserstein distance} between a pair of probability measures $\mu_{x}\in\mathcal{P}_{2}\left(\mathcal{X}\right), \mu_{y}\in\mathcal{P}_{2}\left(\mathcal{Y}\right)$, is defined as
\begin{align}
W^{2}\left(\mu_x,\mu_y\right) := \underset{\pi\in\Pi\left(\mu_x,\mu_y\right)}{\inf}\:\displaystyle\int_{\mathcal{X}\times\mathcal{Y}}c\left(\bm{x},\bm{y}\right)\:\differential\pi(\bm{x},\bm{y}),
\label{DefWassContinuous}    
\end{align}
where $\Pi\left(\mu_x,\mu_y\right)$ is the set of joint probability measures or couplings over the product space $\mathcal{X}\times\mathcal{Y}\subseteq\mathbb{R}^{2d}$, having $\bm{x}$ marginal $\mu_x$, and $\bm{y}$ marginal $\mu_{y}$. Throughout, we use the ground cost $c\left(\bm{x},\bm{y}\right):=\|\bm{x}-\bm{y}\|_{2}^{2}$ (the squared Euclidean distance) for $\bm{x}\in\mathcal{X}, \bm{y}\in\mathcal{Y}$. To lighten nomenclature, we henceforth refer to (\ref{DefWassContinuous}) as the ``squared Wasserstein distance'' dropping the prefix 2.

It is well-known \citep[Ch. 7]{villani2003topics} that the Wasserstein distance $W$ defines a metric on $\mathcal{P}_{2}\left(\mathcal{X}\right)$. The minimizer of the linear program (\ref{DefWassContinuous}), denoted as $\pi^{\text{opt}}$, is referred to as the \emph{optimal transportation plan}. If $\mu\in\mathcal{P}_{2,{\rm{ac}}}(\mathcal{X})$, then  $\pi^{\text{opt}}$ is supported on the graph of the \emph{optimal transport map} $T^{\text{opt}}$ pushing $\mu_x$ to $\mu_y$. We can rewrite (\ref{DefWassContinuous}) as
\begin{align}
W^{2}\left(\mu_x,\mu_y\right) = \underset{\text{Measurable}\, T:T_{\#}\mu_x = \mu_y}{\inf} \displaystyle\int_{\mathcal{X}}c\left(\bm{x},T(\bm{x})\right)\differential\mu_{x},    
\label{DefWassOMT}    
\end{align}
and for the ground cost $c\left(\bm{x},\bm{y}\right):=\|\bm{x}-\bm{y}\|_{2}^{2}$, the $\arg\inf$ for (\ref{DefWassOMT}) is precisely $T^{\text{opt}}$ that is unique a.e. \citep{brenier1991polar}. We refer to $\left(\mathcal{P}_2\left(\mathcal{X}\right),W\right)$ as the \emph{Wasserstein space} since it allows to define a Riemannian-like geometry. In particular, letting $L^{2}(\mu)$ denote the space of functions from $\left(\mathcal{X},\mathcal{B}(\mathcal{X})\right)$ to $\left(\mathcal{Y},\mathcal{B}(\mathcal{Y})\right)$, which are square integrable w.r.t. $\mu\in\mathcal{P}_2(\mathcal{X})$, we define the tangent space of $\left(\mathcal{P}_2\left(\mathcal{X}\right),W\right)$ at $\mu\in\mathcal{P}_2(\mathcal{X})$ as
$$\mathcal{T}_{\mu}\mathcal{P}_{2}\left(\mathcal{X}\right) := \overline{\{\nabla\phi \mid \phi\in C_{c}^{\infty}(\mathcal{X})\}},$$
where the overline denotes closure w.r.t. $L^{2}(\mu)$; see e.g., \citep[Ch. 13]{villani2009optimal}. 

A proper lsc functional $\Phi:\mathcal{P}_{2}(\mathcal{X})\mapsto (-\infty,+\infty]$ is said to be \emph{convex along generalized geodesics
defined by the 2-Wasserstein distance} \citep[Ch. 9]{ambrosio2008gradient}, if for any $t\in[0,1]$ and any $\mu_1,\mu_2\in\mathcal{P}_{2}(\mathcal{X})$, $\mu_3\in\mathcal{P}_{2,{\rm{ac}}}(\mathcal{X})$, we have
$$\Phi\left(\left(tT^{\rm{opt}}_{3\rightarrow 1} + (1-t)T^{\rm{opt}}_{3\rightarrow 2}\right)_{\#}\mu_3\right) \leq t\Phi(\mu_1) + (1-t)\Phi(\mu_2),$$
where $T^{\rm{opt}}_{3\rightarrow 1}$ and $T^{\rm{opt}}_{3\rightarrow 2}$ are the optimal transport maps pushing $\mu_3$ forward to $\mu_1$, and $\mu_3$ forward to $\mu_2$, respectively. The measure-valued curve $t \mapsto \left(tT^{\rm{opt}}_{3\rightarrow 1} + (1-t)T^{\rm{opt}}_{3\rightarrow 2}\right)_{\#}\mu_3$ interpolates between $\mu_2 (t=0)$ and $\mu_1 (t=1)$. 

Given proper lsc $\Phi:\mathcal{P}_{2}(\mathcal{X})\mapsto (-\infty,+\infty]$, its strong Fr\'{e}chet subdifferential $\mu\mapsto\partial\Phi(\mu)$ allows defining the \emph{Wasserstein gradient flow} (WGF) of the functional $\Phi$, see e.g., \citep[Ch. 11]{ambrosio2008gradient}, \citep[Ch. 23]{villani2009optimal}, \citep{santambrogio2017euclidean}. Additionally, when $\Phi$ is convex along generalized geodesics mentioned before, then the WGF can be characterized as the continuity equation
\begin{align}
\dfrac{\partial\mu}{\partial t} + \nabla\cdot(\mu\bm{v}(\mu)) = 0, \; \bm{v}(\mu)\in\partial \Phi(\mu) \cap \mathcal{T}_{\mu}\mathcal{P}_2(\mathcal{X}) \;\Leftrightarrow\; \bm{v}(\mu) = \nabla\dfrac{\delta\Phi}{\delta \mu}, 
\label{WassGradFlow1}    
\end{align}
where $\nabla$ is the $d$ dimensional Euclidean gradient operator, and $\frac{\delta}{\delta\mu}$ denotes the functional derivative w.r.t. $\mu$. More generally, for non-smooth $\Phi$, one can define WGF via Evolution Variational Inequality (EVI) \citep[Thm. 11.1.4]{ambrosio2008gradient}, \citep{salim2020wasserstein}.

Following (\ref{WassGradFlow1}), we can formally define the Wasserstein gradient \citep[Ch. 9.1]{villani2003topics}, \citep[Ch. 8]{ambrosio2008gradient} as
\begin{align}
\nabla^{W}\Phi(\mu) := -\nabla\cdot\left(\mu \nabla\dfrac{\delta\Phi}{\delta \mu}\right),   \label{DefWassGrad}    
\end{align}
and express the WGF in the form
\begin{align}
\dfrac{\partial\mu}{\partial t} = -\nabla^{W}\Phi(\mu).
\label{WassGradFlowStandardForm}    
\end{align}
In this work, we consider smooth $\Phi$ with singleton $\partial\Phi(\mu) = \{\nabla^{W}\Phi(\mu)\}$ \citep[Ch. 10.4]{ambrosio2008gradient}.

\textbf{Sinkhorn regularization.} For $\pi\in\Pi\left(\mu_x,\mu_y\right)$ and a reference probability measure $\pi_{0}$ supported over $\mathcal{X}\times\mathcal{Y}$, the notation $\pi \ll \pi_0$ means that $\pi$ is absolutely continuous w.r.t. $\pi_0$. Given a strictly convex regularizer $R(\cdot)$, define the \emph{regularized squared Wasserstein distance} 
\begin{align}
W_{\varepsilon}^{2}\left(\mu_x,\mu_y\right) := \underset{\stackrel{\pi\in\Pi\left(\mu_x,\mu_y\right)}{\pi\ll\pi_{0}}}{\inf}\:\displaystyle\int_{\mathcal{X}\times\mathcal{Y}}c\left(\bm{x},\bm{y}\right)\:\differential\pi(\bm{x},\bm{y}) + \varepsilon \displaystyle\int_{\mathcal{X}\times\mathcal{Y}}R\left(\dfrac{\differential\pi}{\differential\pi_{0}}\right)\differential\pi_{0}(\bm{x},\bm{y})
\label{DefRegWassContinuous}    
\end{align}
where $\varepsilon > 0$ is a regularization parameter, and $\dfrac{\differential\pi}{\differential\pi_{0}}$ denotes the Radon-Nikodym derivative. Examples of $\pi_{0}$ include the product measure $\mu_x(\bm{x}) \mu_y(\bm{y})$ \citep{genevay2016stochastic} and the uniform measure \citep{cuturi2013sinkhorn}. In this paper, we consider the entropic regularizer 
\begin{align}
R(x) := x\log x - x\quad\text{for}\; x\geq 0, \quad\text{with the convention}\;0\log 0 = 0.  \label{EntropicR}    
\end{align}
The work in \citep{cuturi2013sinkhorn} considered the discrete version of (\ref{DefRegWassContinuous}) with an entropic regularizer $R$ as above, and named it as the \emph{Sinkhorn divergence}. This entropy or Sinkhorn regularized squared Wasserstein distance has found widespread applications in the computation and analysis of variational problems involving the Wasserstein distance (see e.g., \citet{benamou2015iterative,carlier2017convergence,peyre2015entropic,cuturi2016smoothed}), and will be useful in our development too.

\textbf{Wasserstein barycenter.} Given the measures $\mu_1, \dots, \mu_{n} \in \mathcal{P}_{2}\left(\mathcal{X}\right)$ and positive weights $w_{1}, \hdots, w_{n}$, the \emph{Wasserstein barycenter} \citep{agueh2011barycenters} is given by
\begin{align}
\underset{\mu\in\mathcal{P}_{2}\left(\mathcal{X}\right)}{\arg\inf}\:\displaystyle\sum_{i=1}^{n}w_{i}W^{2}\left(\mu,\mu_{i}\right).
\label{DefWassBary}    
\end{align}
In (\ref{DefWassBary}), replacing $W^{2}$ by $W_{\varepsilon}^{2}$ defined in (\ref{DefRegWassContinuous}) with $R$ as in (\ref{EntropicR}), results in the \emph{Sinkhorn regularized Wasserstein barycenter}
\begin{align}
\underset{\mu\in\mathcal{P}_{2}\left(\mathcal{X}\right)}{\arg\inf}\:\displaystyle\sum_{i=1}^{n}w_{i}W_{\varepsilon}^{2}\left(\mu,\mu_{i}\right).
\label{DefSinkhornRegBary}    
\end{align}

\textbf{Wasserstein proximal operator.} We use the notation $\prox^{W}_{G(\cdot)}(\zeta)$ to denote the \emph{Wasserstein proximal operator} of proper lsc $G:\mathcal{P}_{2}(\mathcal{X})\mapsto (-\infty,+\infty]$, acting on $\zeta\in\mathcal{P}_{2}\left(\mathcal{X}\right)$, given by 
\begin{align}
\prox^{W}_{G(\cdot)}(\zeta) := \underset{\mu\in\mathcal{P}_{2}\left(\mathcal{X}\right)}{\arg\inf}\:\dfrac{1}{2}W^{2}\left(\mu,\zeta\right) + G(\mu).
\label{defWassProx}
\end{align}
The Wasserstein proximal operator (\ref{defWassProx}) can be seen as a generalization of the finite dimensional Euclidean proximal operator of proper lsc $g:\mathbb{R}^{d}\mapsto(-\infty,+\infty]$, given by
\begin{align}
\prox^{\|\cdot\|_{2}}_{g}(\bm{z}) := \underset{\bm{x}\in\mathbb{R}^{d}}{\arg\inf}\:\dfrac{1}{2}\|\bm{x}-\bm{z}\|_{2}^{2} + g(\bm{x}).   
\label{defEuclideanProx}    
\end{align}
Wasserstein proximal operators of the form (\ref{defWassProx}) go back to the seminal work of \citet{jordan1998variational}, and have been used in stochastic prediction \citep{caluya2019gradient}, control \citep{caluya2021wasserstein}, learning \citep{chu2019probability,frogner2020approximate,salim2020wasserstein,mokrov2021large}, and in modeling of population dynamics \citep{bunne2022proximal}.

\textbf{Legendre-Fenchel conjugate.} The Legendre-Fenchel conjugate of a real-valued function $f$ is 
$$f^{*}(\bm{y}) := \underset{\bm{x}\in\:\text{domain}(f)}{\sup}\left(\langle\bm{y},\bm{x}\rangle - f(\bm{x})\right),$$
where $\langle\cdot,\cdot\rangle$ denotes the standard inner product. The function $f^{*}$ is convex even if $f$ is not. When $f(\bm{x}) = \langle\bm{a},\bm{x}\rangle$, $\bm{a}\in\mathbb{R}^{d}\setminus\{\bm{0}\}$, then $f^{*}(\bm{y})$ is the indicator function of the singleton $\{\bm{a}\}$, i.e., 
\begin{align}
f^{*}(\bm{y}) = \begin{cases} 0 & \text{if}\quad\bm{y}=\bm{a},\\
+\infty & \text{otherwise}.
\end{cases}
\label{ConjugateOfLinear}    
\end{align}

\textbf{ADMM.} The constrained optimization problem $\underset{\bm{x}\in\mathbb{R}^{N}}{\min} f(\bm{x})$ subject to $\bm{x}\in\mathcal{C}\subset\mathbb{R}^{N}$, where the function $f$ and the set $\mathcal{C}$ are convex, can be re-written as $\underset{\bm{x},\bm{z}\in\mathbb{R}^{N}}{\min} f(\bm{x}) + \bm{1}_{\mathcal{C}}(\bm{z})$ subject to $ \bm{x}=\bm{z}$ where the indicator function $\bm{1}_{\mathcal{C}}(\bm{z}) := 0$ if $\bm{z}\in\mathcal{C}$, and $\bm{1}_{\mathcal{C}}(\bm{z}) := +\infty$ if $\bm{z}\notin\mathcal{C}$. Denote the dual variable associated with the constraint $\bm{x}=\bm{z}$ as $\bm{\nu}\in\mathbb{R}^{N}$, and let $\widetilde{\bm{\nu}}:= \bm{\nu}/\tau$ be the scaled dual variable for some parameter $\tau >0$. The augmented Lagrangian for this problem is $L_{\tau}\left(\bm{x},\bm{z},\widetilde{\bm{\nu}}\right) := f(\bm{x}) + \bm{1}_{\mathcal{C}}(\bm{z}) + \frac{\tau}{2}\|\bm{x} - \bm{z} + \widetilde{\bm{\nu}}\|_{2}^{2}$. Each iteration of the ADMM algorithm in the so-called ``scaled form'' \citep[Ch. 5]{boyd2011distributed}, comprises of the following three steps:
\begin{subequations}
\begin{align}
\bm{x}^{k+1} &= \underset{\bm{x}\in\mathbb{R}^{N}}{\arg\min}\:f(\bm{x}) + \frac{\tau}{2}\|\bm{x} - \bm{z}^{k} + \widetilde{\bm{\nu}}^{k}\|_{2}^{2} \stackrel{(\ref{defEuclideanProx})}{=} \prox^{\|\cdot\|_{2}}_{\frac{1}{\tau}f}\left(\bm{z}^{k} - \widetilde{\bm{\nu}}^{k}\right), \label{xUpdate}\\
\bm{z}^{k+1} &= {\rm{proj}}_{\mathcal{C}}\left(\bm{x}^{k+1} + \widetilde{\bm{\nu}}^{k}\right), \label{zUpdate}\\
 \widetilde{\bm{\nu}}^{k+1} &= \widetilde{\bm{\nu}}^{k} + \left(\bm{x}^{k+1} - \bm{z}^{k+1}\right), \label{dualUpdate}
\end{align}
\label{BasicADMM}
\end{subequations}
where the iteration index $k\in\mathbb{N}_{0}$ (the set of whole numbers $\{0,1,2,\hdots\}$), and ${\rm{proj}}_{\mathcal{C}}$ denotes the Euclidean projection onto $\mathcal{C}$. The steps (\ref{xUpdate})-(\ref{zUpdate}) involve alternating minimization of the augmented Lagrangian $L_{\tau}$, and the step (\ref{dualUpdate}) involves dual ascent. Notice that in the scaled form ADMM, the parameter $\tau$ \emph{does not} appear in (\ref{dualUpdate}) as the pre-factor of the term in parenthesis. For ADMM convergence results, see e.g.,  \citep{nishihara2015general}, \citep{wang2019global}.

For a separable objective $f(\bm{x}_{1},\hdots,\bm{x}_{n}) = \displaystyle\sum_{i=1}^{n}f_{i}(\bm{x}_{i})$, where $\bm{x}_{i}\in\mathbb{R}^{N}$ and $f_i$ convex for all $i\in[n]$, it is immediate from (\ref{BasicADMM}) that the updates (\ref{xUpdate}) and (\ref{dualUpdate}) can be parallelized across the index $i\in[n]$. The nature of computation in step (\ref{zUpdate}) depends on the constraint set $\mathcal{C}$, see e.g., \citep[Ch. 5]{parikh2014proximal}. For instance, if $\mathcal{C}$ is the consensus constraint $\bm{x}_{1} = \hdots = \bm{x}_{n}=\bm{z}$, then (\ref{zUpdate}) requires an averaging of the local updates, resulting in a ``broadcast and gather'' computation. In Sec. \ref{subsec:ZetaUpdate}, we will encounter an instance of (\ref{BasicADMM}) that will admit parallelization.

%%%%%%%%%%%%%%%%%%%%%%%%%%%%%%%%%%%%%%%%%%%%%%%%%%%%%%%%%%%%%%%%%%%%%%%%%%%%%%%%%%

\section{Main Idea}\label{subsec:MainIdea}
To leverage the additive structure of the objective in (\ref{AdditiveOptimizationMeasure}) for distributed computation, we start by rewriting it in the consensus form. Specifically, we relabel the argument of the functional $F_{i}$ as $\mu_{i}$ for all $i\in[n]$, and then impose the consensus constraint $\mu_{1} = \mu_{2} = \hdots = \mu_{n}$. Letting $\mathcal{P}_{2}^{n+1}(\mathcal{X}) := \underbrace{\mathcal{P}_{2}(\mathcal{X}) \times \hdots \times  \mathcal{P}_{2}(\mathcal{X})}_{n+1\;\text{times}}$, we thus transcribe (\ref{AdditiveOptimizationMeasure}) into
\begin{subequations}
\begin{align}
&\underset{(\mu_1, \hdots, \mu_n, \zeta)\in \mathcal{P}_{2}^{n+1}(\mathcal{X})}{\arg\inf} \: F_{1}(\mu_1) + F_{2}(\mu_2) + \hdots + F_{n}(\mu_n) \label{AdditiveOptimizationnplusoneObj}\\
&\;\qquad\text{subject to} \qquad\; \mu_{i} = \zeta \quad\text{for all}\;i\in[n].\label{AdditiveOptimizationnplusoneConstr}
\end{align}
\label{AdditiveOptimizationnplusone}
\end{subequations}
Denote an element of the base space as $\bm{\theta}\in\mathcal{X}\subseteq\mathbb{R}^{d}$. Akin to the standard (Euclidean) augmented Lagrangian, we define the \emph{Wasserstein augmented Lagrangian} 
\begin{align}
L_{\alpha}(\mu_1, \hdots, \mu_{n},\zeta,\nu_1,\hdots,\nu_{n}) := \displaystyle\sum_{i=1}^{n}\bigg\{F_{i}(\mu_i) + \dfrac{\alpha}{2}W^{2}\left(\mu_i,\zeta\right) + \int_{\mathcal{X}}\nu_{i}(\bm{\theta})\left(\differential\mu_{i} -\differential\zeta\right) \bigg\}
\label{WassAugLagrangian}    
\end{align}
where $\nu_{i}(\bm{\theta})$, $i\in[n]$, are the Lagrange multipliers for the constraints in (\ref{AdditiveOptimizationnplusoneConstr}), and $\alpha>0$ is a regularization constant. 

Motivated by the Euclidean ADMM, we then set up the recursions
\begin{subequations}
\begin{align}
\mu_{i}^{k+1} &= \underset{\mu_{i}\in\mathcal{P}_{2}(\mathcal{X})}{\arg\inf}\:L_{\alpha}\left(\mu_1, \hdots, \mu_{n},\zeta^{k},\nu_1^{k},\hdots,\nu_{n}^{k}\right) \label{ADMMrecursionsmuupdate}\\
\zeta^{k+1} &= \underset{\zeta\in\mathcal{P}_{2}(\mathcal{X})}{\arg\inf}\:L_{\alpha}\left(\mu_1^{k+1}, \hdots, \mu_{n}^{k+1},\zeta,\nu_1^{k},\hdots,\nu_{n}^{k}\right) \label{ADMMrecursionszetaupdate}\\
\nu_{i}^{k+1} &= \nu_i^{k} + \alpha\left(\mu_{i}^{k+1} - \zeta^{k+1}\right)\label{ADMMrecursionsnuupdate}
\end{align}
\label{ADMMrecursions}
\end{subequations}
where $i\in[n]$, and the recursion index $k\in\mathbb{N}_{0}$. It will be useful to introduce 
\begin{align}
\nu_{\text{sum}}^{k}(\bm{\theta}) := \displaystyle\sum_{i=1}^{n}\nu_{i}^{k}(\bm{\theta}), \quad k\in\mathbb{N}_{0}.     
\label{LagMultiplierAverage}    
\end{align} 
We view (\ref{ADMMrecursionsmuupdate})-(\ref{ADMMrecursionszetaupdate}) as primal updates, and (\ref{ADMMrecursionsnuupdate}) as dual ascent.

Substituting (\ref{WassAugLagrangian}) in (\ref{ADMMrecursions}), dropping the terms independent of the decision variable in the respective $\arg\inf$, re-scaling, and using (\ref{LagMultiplierAverage}), the recursions (\ref{ADMMrecursions}) simplify to
\begin{subequations}
\begin{align}
\mu_{i}^{k+1} &= \underset{\mu_{i}\in\mathcal{P}_{2}(\mathcal{X})}{\arg\inf}\: \dfrac{1}{2}W^{2}\left(\mu_{i},\zeta^{k}\right) + \dfrac{1}{\alpha}\bigg\{F_{i}(\mu_i) + \int_{\mathcal{X}} \nu_{i}^{k}(\bm{\theta})\differential\mu_{i}\bigg\}\nonumber\\
&= \prox^{W}_{\frac{1}{\alpha}\left(F_{i}(\cdot) + \int \nu_{i}^{k}\differential(\cdot)\right)}\left(\zeta^{k}\right), \label{ADMMrecursionsmuupdateSimplified}\\
\zeta^{k+1} &= \underset{\zeta\in\mathcal{P}_{2}(\mathcal{X})}{\arg\inf}\:\displaystyle\sum_{i=1}^{n}\bigg\{\dfrac{1}{2}W^{2}\left(\mu_{i}^{k+1},\zeta\right) - \dfrac{1}{\alpha}\int_{\mathcal{X}} \nu_{i}^{k}(\bm{\theta})\differential\zeta\bigg\}\nonumber\\
&= \underset{\zeta\in\mathcal{P}_{2}(\mathcal{X})}{\arg\inf}\bigg\{\left(\displaystyle\sum_{i=1}^{n}W^{2}\left(\mu_{i}^{k+1},\zeta\right)\right) - \dfrac{2}{\alpha}\int_{\mathcal{X}} \nu_{\text{sum}}^{k}(\bm{\theta})\differential\zeta\bigg\}, \label{ADMMrecursionszetaupdateSimplified}\\
\nu_{i}^{k+1} &= \nu_i^{k} + \alpha\left(\mu_{i}^{k+1} - \zeta^{k+1}\right).\label{ADMMrecursionsnuupdateSimplified}
\end{align}
\label{ADMMSimplified}
\end{subequations}
We refer to (\ref{ADMMSimplified}) as the \emph{Wasserstein consensus ADMM} generalizing its finite dimensional Euclidean counterpart in the sense (\ref{ADMMrecursionsmuupdateSimplified})-(\ref{ADMMrecursionszetaupdateSimplified}) are analogues of the so-called $x$ and $z$ updates, respectively \citep[Ch. 5.2.1]{parikh2014proximal}. However, important difference arises in (\ref{ADMMrecursionszetaupdateSimplified}) compared to its Euclidean counterpart due to the sum of squares of Wasserstein distances. In the Euclidean case, the corresponding $z$ update can be analytically performed in terms of the \emph{arithmetic mean} of the $x$ updates. While (\ref{ADMMrecursionszetaupdateSimplified}) involves a \emph{generalized mean} of the updates from (\ref{ADMMrecursionsmuupdateSimplified}), we now have \emph{Wasserstein barycentric proximal} of a linear functional in $\nu_{\text{sum}}^{k}$ w.r.t. $n$ measures $\{\mu_{1}^{k+1},\hdots,\mu_{n}^{k+1}\}$. 

The proximal updates (\ref{ADMMrecursionsmuupdateSimplified}) are closely related to the WGFs of the form (\ref{WassGradFlowStandardForm}) generated by the respective (scaled) free energy functionals \begin{align}
\Phi_{i}(\mu_i) := F_{i}(\mu_i) + \int_{\mathcal{X}} \nu_{i}^{k}\differential\mu_i, \quad \mu_{i}\in\mathcal{P}_{2}(\mathcal{X}), \quad i\in[n].    
\label{FiPhii}    
\end{align}
As per the assumptions on $F_i$, the functionals $\Phi_i$ are also proper lsc and convex along generalized geodesics defined by the 2-Wasserstein distance. 
%Following (\ref{DefWassGrad}), the Wasserstein gradient of $\Phi_{i}$ evaluated at $\widetilde{\mu}_{i}\in\mathcal{P}_{2}(\mathcal{X})$, is given by
% $$\nabla^{W}\Phi_{i}\left(\widetilde{\mu}_{i}\right) := -\nabla \cdot\left(\widetilde{\mu}_{i} \nabla \frac{\delta \Phi_{i}}{\delta \widetilde{\mu}_{i}}\right), \quad i\in[n].$$
As $1/\alpha \downarrow 0$, the sequence $\{\mu_{i}^{k}(\alpha)\}_{k\in\mathbb{N}_{0}}$ generated by the updates (\ref{ADMMrecursionsmuupdateSimplified}) converge to the measure-valued solution trajectory $\widetilde{\mu}_{i}(t,\cdot)$, $t\in[0,\infty)$ solving the initial value problem (IVP)
\begin{align}
\dfrac{\partial\widetilde{\mu}_{i}}{\partial t} = - \nabla^{W}\Phi_{i}\left(\widetilde{\mu}_{i}\right), \quad \widetilde{\mu}_{i}(t=0,\cdot) = \widetilde{\mu}_{i}^{0}(\cdot), \quad i\in[n].    \label{WassGradFlow}    
\end{align}
Thus, in a rather generic setting, performing the proximal updates (\ref{ADMMrecursionsmuupdateSimplified}) in parallel across the index $i\in[n]$, amounts to performing distributed time updates for the approximate transient solutions of the IVPs (\ref{WassGradFlow}). In Appendix \ref{AppExampleFi}, we provide important examples of (\ref{FiPhii})-(\ref{WassGradFlow}). An interesting observation for (\ref{ADMMrecursionsmuupdateSimplified}) is that for each $i\in[n]$, the dual variables $\nu_{i}^{k}$ contribute as time-varying advection potentials irrespective of whether $F_i$ already has an advection potential or not.

\begin{remark}\label{Remark:HandHoldingLangrangeMultiplier}
Notice that the Lagrange multiplier $\nu_i$ for the $i$\textsuperscript{th} measure consensus constraint (\ref{AdditiveOptimizationnplusoneConstr}) must be an element of the dual space of $\mathcal{P}_2$ comprising of bounded linear functionals of the elements of $\mathcal{P}_2$. Thus, when the primal updates for $\mu_i$ are identified with the corresponding WGFs, then the Lagrange multipliers $\nu_i$
 become "algorithmic" advection potentials. For the same reason, the integral involving the Lagrange multiplier ends up being simply an Euclidean inner product post-discretization; see (\ref{ADMMdiscrete}).    
\end{remark}

In the next Section, we propose a two-layer ADMM algorithm (see Fig. \ref{fig:GeneralBlockDiagram}) to solve (\ref{ADMMSimplified}).

\begin{figure}[!h]
\centering
\includegraphics[width=0.8\linewidth]{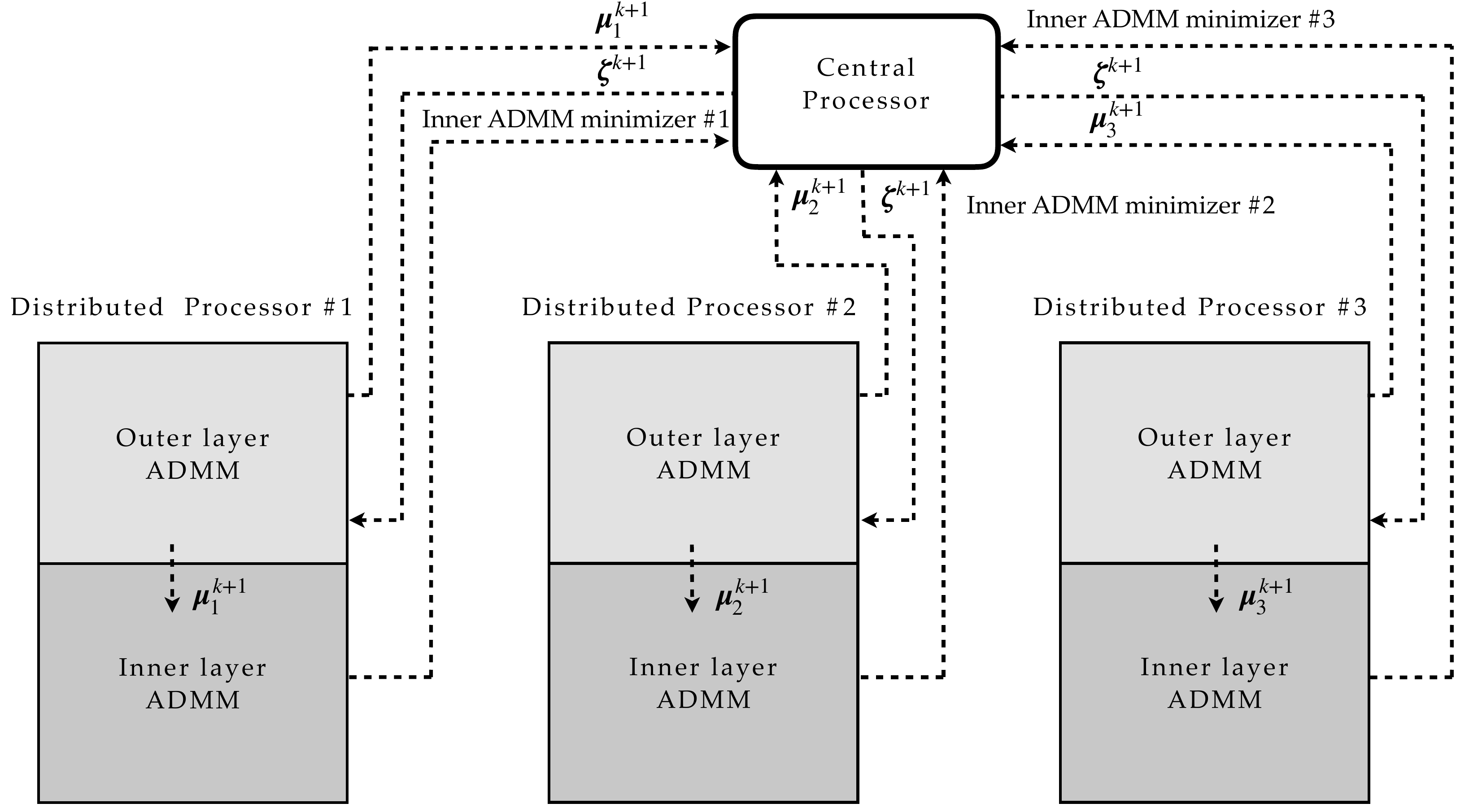}
\caption{High level schematic of the proposed two-layer ADMM algorithm illustrated with one central and $n=3$ distributed processors. The central processor updates $\bm{\zeta}^{k+1}$.
The ``upstairs" (lighter shade) of the distributed processors  update $\bm{\mu}_i^{k+1}$ via \emph{outer layer ADMM} (Sec. \ref{subsec:MuUpdate}). These distributed $\bm{\mu}_i^{k+1}$ updates and the centralized $\bm{\zeta}^{k+1}$ values are passed to the ``downstairs" (darker shade) of the distributed processors for updating $\bm{\zeta}^{k+1}$ via an \emph{inner layer ADMM} (Sec. \ref{subsec:ZetaUpdate}).}
\label{fig:GeneralBlockDiagram}
\end{figure}

%%%%%%%%%%%%%%%%%%%%%%%%%%%%%%%%%%%%%%%%%%%%%%
%%%%%%%%%%%%%%%%%%%%%%%%%%%%%%%%%%%%%%%%%%%%%%

\section{Results}\label{sec:Results}
To numerically realize the recursions (\ref{ADMMSimplified}), we consider a sequence of discrete probability distributions $\{\bm{\mu}_{1}^{k}, \hdots, \bm{\mu}_{n}^{k}, \bm{\zeta}^{k}\}_{k\in\mathbb{N}_{0}}$ where each distribution is a probability vector of length $N\times 1$ comprising respective probability values at $N$ samples. Thus, for each fixed $k\in\mathbb{N}_{0}$, the tuple $$\left(\bm{\mu}_{1}^{k}, \hdots, \bm{\mu}_{n}^{k}, \bm{\zeta}^{k}\right)\in\underbrace{\Delta^{N-1} \times \hdots \times\Delta^{N-1}}_{n+1\;\text{times}} =: \left(\Delta^{N-1}\right)^{n+1}\;\text{(the product simplex)}.$$ Likewise, for each $k\in\mathbb{N}_{0}$, the multipliers $\left(\bm{\nu}_{1}^{k},\hdots,\bm{\nu}_{n}^{k}\right)\in\mathbb{R}^{nN}$, and $\bm{\nu}^{k}_{\text{sum}}=\displaystyle\sum_{i=1}^{n}\bm{\nu}^{k}_{i}\in\mathbb{R}^{N}$. 

Given probability vectors $\bm{\xi}, \bm{\eta}\in\Delta^{N-1}$, let
$\Pi_{N}\left(\bm{\xi},\bm{\eta}\right) := \{\bm{M}\in\mathbb{R}^{N\times N} \mid \bm{M} \geq \bm{0}\:\text{(elementwise)}, \: \bm{M}\bm{1} = \bm{\xi},\: \bm{M}^{\top}\bm{1} = \bm{\eta}\}$. Also, let $\bm{C}\in\mathbb{R}^{N\times N}$ denote the squared Euclidean distance matrix for the sampled data $\{\bm{\theta}_{r}\}_{r\in[N]}$ in $\mathbb{R}^{d}$, i.e., the entries of the matrix $\bm{C}$ are $\bm{C}(i,j) := \|\bm{\theta}_{i}-\bm{\theta}_{j}\|_{2}^{2}$ for all $i,j\in[N]$.  

For each $i\in[n]$ and $k\in\mathbb{N}_{0}$, we write the discrete version of (\ref{ADMMSimplified}) as
\begin{subequations}
\begin{align}
\bm{\mu}_{i}^{k+1} &= \prox^{W}_{\frac{1}{\alpha}\left(F_{i}(\bm{\mu}_{i}) + \langle\bm{\nu}_{i}^{k},\bm{\mu}_{i}\rangle\right)}\left(\bm{\zeta}^{k}\right)   \nonumber\\
&= \underset{\bm{\mu}_{i}\in\Delta^{N-1}}{\arg\inf}\bigg\{\underset{\bm{M}\in\Pi_{N}\left(\bm{\mu}_{i},\bm{\zeta}^{k}\right)}{\min}\frac{1}{2}\langle\bm{C},\bm{M}\rangle + \frac{1}{\alpha}\left(F_{i}(\bm{\mu}_{i}) + \langle\bm{\nu}_{i}^{k},\bm{\mu}_{i}\rangle\right)\bigg\}, \label{MuUpdateDiscrete}\\
\bm{\zeta}^{k+1} &= \underset{\bm{\zeta}\in\Delta^{N-1}}{\arg\inf} \bigg\{ \left(\displaystyle\sum_{i=1}^{n}\underset{\bm{M}_{i}\in\Pi_{N}\left(\bm{\mu}_{i}^{k+1},\bm{\zeta}\right)}{\min}\frac{1}{2}\langle\bm{C},\bm{M}_{i}\rangle\right) - \frac{2}{\alpha}\langle\bm{\nu}^{k}_{\text{sum}},\bm{\zeta}\rangle \bigg\}, \label{ZetaUpdateDiscrete}\\
\bm{\nu}_{i}^{k+1} &= \bm{\nu}_{i}^{k} + \alpha\left(\bm{\mu}_{i}^{k+1} - \bm{\zeta}^{k+1}\right), \label{NuUpdateDiscrete}
\end{align}
\label{ADMMdiscrete}  
\end{subequations}
wherein (\ref{MuUpdateDiscrete})-(\ref{ZetaUpdateDiscrete}) used the discrete version of (\ref{DefWassContinuous}).

Replacing the squared Wasserstein distance (\ref{DefWassContinuous}) in (\ref{ADMMSimplified}) by its Sinkhorn regularized version (\ref{DefRegWassContinuous}), modify the recursions (\ref{ADMMdiscrete}) as
\begin{subequations}
\begin{align}
\bm{\mu}_{i}^{k+1} &= \prox^{W_{\varepsilon}}_{\frac{1}{\alpha}\left(F_{i}(\bm{\mu}_{i}) + \langle\bm{\nu}_{i}^{k},\bm{\mu}_{i}\rangle\right)}\left(\bm{\zeta}^{k}\right)   \nonumber\\
&= \underset{\bm{\mu}_{i}\in\Delta^{N-1}}{\arg\inf}\bigg\{\underset{\bm{M}\in\Pi_{N}\left(\bm{\mu}_{i},\bm{\zeta}^{k}\right)}{\min}\bigg\langle\frac{1}{2}\bm{C} + \varepsilon\log\bm{M},\bm{M}\bigg\rangle + \frac{1}{\alpha}\left(F_{i}(\bm{\mu}_{i}) + \langle\bm{\nu}_{i}^{k},\bm{\mu}_{i}\rangle\right)\bigg\}, \label{MuUpdateDiscreteRegularized}\\
\bm{\zeta}^{k+1} &= \underset{\bm{\zeta}\in\Delta^{N-1}}{\arg\inf} \bigg\{ \left(\displaystyle\sum_{i=1}^{n}\underset{\bm{M}_{i}\in\Pi_{N}\left(\bm{\mu}_{i}^{k+1},\bm{\zeta}\right)}{\min}\bigg\langle\frac{1}{2}\bm{C} + \varepsilon\log\bm{M}_{i},\bm{M}_{i}\bigg\rangle\right) - \frac{2}{\alpha}\langle\bm{\nu}^{k}_{\text{sum}},\bm{\zeta}\rangle \bigg\}, \label{ZetaUpdateDiscreteRegularized}\\
\bm{\nu}_{i}^{k+1} &= \bm{\nu}_{i}^{k} + \alpha\left(\bm{\mu}_{i}^{k+1} - \bm{\zeta}^{k+1}\right), \label{NuUpdateDiscreteRegularized}
\end{align}
\label{ADMMdiscreteRegularized}  
\end{subequations}
where $\varepsilon > 0$ is a regularization parameter.

\begin{remark}\label{RemarkWhyEntropicReg}
For $\varepsilon\downarrow 0$, the solution of the inner minimization (\ref{MuUpdateDiscreteRegularized}) is known \citep[Sec. 3]{peyre2015entropic} to be a consistent approximation of that in (\ref{MuUpdateDiscrete}). The Wasserstein proximal update (\ref{MuUpdateDiscrete}) can, in principle, be performed by the proximal gradient Jordan-Kinderlehrer-Otto (JKO) algorithm as in \citep{salim2020wasserstein} with more general regularization. Our motivation for choosing  Sinkhorn regularization is computational convenience. As we explain in Sec. \ref{subsec:MuUpdate}, when we dualize the inner minimization problem in (\ref{MuUpdateDiscreteRegularized}), not only we have strong duality, but we also can explicitly write the proximal update in terms of multipliers which can be obtained, in general, numerically via provably contractive block-coordinate ascent. Note that \citep[Remark 1]{salim2020wasserstein} mentions the computational convenience of performing the JKO update for the negative entropy regularization.
\end{remark}

\begin{remark}\label{RemarkWhyNotBarycenter}
While there exists prior work such as \citep{yang2021fast} for unregularized computation of the Wasserstein barycenter using multi-block ADMM, the nested minimization in (\ref{ZetaUpdateDiscrete}) is different from computing barycenter in that it involves computing the Wasserstein barycentric proximal.  
\end{remark}

We next provide novel results and algorithmic details to numerically perform the recursions (\ref{ADMMdiscreteRegularized}).

\subsection{The \texorpdfstring{$\bm{\mu}$}{\mu} Update}\label{subsec:MuUpdate}
The Sinkhorn regularized recursions (\ref{MuUpdateDiscreteRegularized}) comprise the outer layer ADMM in Fig. \ref{fig:GeneralBlockDiagram}. These recursions allow us to get semi-analytical handle on the nested minimization via strong duality. Specifically, consider the (proper lsc and convex w.r.t. generalized geodesic) functionals $F_i, G_i:\Delta^{N-1}\mapsto\mathbb{R}$ for all $i\in[n]$, where 
\begin{align}
G_{i}(\bm{\mu_i}) := F_{i}(\bm{\mu}_{i}) + \langle\bm{\nu}_{i}^{k},\bm{\mu}_{i}\rangle,
\label{DefGi}    
\end{align}
and denote the Legendre-Fenchel conjugate of $G_{i}$ as $G_{i}^{*}$. Following \citep[Lemma 3.5]{karlsson2017generalized}, \citep[Sec. III]{caluya2019gradient}, the Lagrange dual problem associated with (\ref{MuUpdateDiscreteRegularized}), for each $i\in[n]$, is
\begin{align}
\left(\bm{\lambda}_{0i}^{\text{opt}},\bm{\lambda}_{1i}^{\text{opt}}\right) = \underset{\bm{\lambda}_{0i},\bm{\lambda}_{1i}\in\mathbb{R}^{N}}{\arg\max} \bigg\{\!\!\langle\bm{\lambda}_{0i},\bm{\zeta}_{k}\rangle - G_{i}^{*}\left(-\bm{\lambda}_{1i}\right) - \alpha\varepsilon\left(\exp\left(\frac{\bm{\lambda}_{0i}^{\top}}{\alpha\varepsilon}\right)\exp\left(-\frac{\bm{C}}{2\varepsilon}\right)\exp\left(\frac{\bm{\lambda}_{1i}}{\alpha\varepsilon}\!\right)\!\right)\!\!\bigg\}.
\label{LagrangeDualProblem}    
\end{align}
Using (\ref{LagrangeDualProblem}), the proximal updates in (\ref{MuUpdateDiscreteRegularized}) can be recovered\footnote{See Appendix \ref{AppExampleMuUpdate} for examples.} via the following Proposition.
\begin{proposition}\citep[Lemma 3.5]{karlsson2017generalized},\citep[Theorem 1]{caluya2019gradient}\label{prop:ProximalUpdateGeneral}
Given $\alpha,\varepsilon > 0$, the squared Euclidean distance matrix $\bm{C}\in\mathbb{R}^{N\times N}$, and the probability vector $\bm{\zeta}^{k}\in\Delta^{N-1}$, $k\in\mathbb{N}_{0}$. Let $\bm{0}$ denote the $N\times 1$ vector of zeros. For $i\in[n]$, the vectors $\bm{\lambda}_{0i}^{\rm{opt}},\bm{\lambda}_{1i}^{\rm{opt}}\in\mathbb{R}^{N}$ in (\ref{LagrangeDualProblem}) solve the system
\begin{subequations}
\begin{align}
&\exp\left(\frac{\bm{\lambda}_{0i}^{\rm{opt}}}{\alpha\varepsilon}\right) \odot \left( \exp\left(-\frac{\bm{C}}{2\varepsilon}\right)\exp\left(\frac{\bm{\lambda}_{1i}^{\rm{opt}}}{\alpha\varepsilon}\right) \right) = \bm{\zeta}_{k},  \label{ZetakEquation}\\ 
&\bm{0} \in \partial_{\bm{\lambda}_{1i}^{\rm{opt}}}G_{i}^{*}\left(-\bm{\lambda}_{1i}^{\rm{opt}}\right) - \exp\left(\frac{\bm{\lambda}_{1i}^{\rm{opt}}}{\alpha\varepsilon}\right) \odot \left(\exp\left(-\frac{\bm{C}^{\top}}{2\varepsilon}\right) \exp\left(\frac{\bm{\lambda}_{0i}^{\rm{opt}}}{\alpha\varepsilon}\right)\right). \label{ZeroInSubdifferential}  
\end{align}
\label{lambda0lambda1equations}  
\end{subequations}
The proximal update $\bm{\mu}_{i}^{k+1}$ in (\ref{MuUpdateDiscreteRegularized}) is given by
\begin{align}
\bm{\mu}_{i}^{k+1} = \exp\left(\frac{\bm{\lambda}_{1i}^{\rm{opt}}}{\alpha\varepsilon}\right) \odot \left(\exp\left(-\frac{\bm{C}^{\top}}{2\varepsilon}\right) \exp\left(\frac{\bm{\lambda}_{0i}^{\rm{opt}}}{\alpha\varepsilon}\right)\right).
\label{MuUpdateGeneral}    
\end{align}
\end{proposition}
For a given $F_i$, in general, the pair $\left(\bm{\lambda}_{0i}^{\rm{opt}},\bm{\lambda}_{1i}^{\rm{opt}}\right)$ need to be computed numerically from (\ref{lambda0lambda1equations}); see various cases discussed in Appendix \ref{AppExampleMuUpdate}. In particular, Theorem \ref{Thm:WassProxOfLinear} of Appendix \ref{AppExampleMuUpdate}, deduces that when $F_i$ in (\ref{DefGi}) is a linear functional, then $\left(\bm{\lambda}_{0i}^{\rm{opt}},\bm{\lambda}_{1i}^{\rm{opt}}\right)$, and thus $\bm{\mu}_{i}^{k+1}$, can in fact be computed analytically. This result will find use in our experiments in Sec. \ref{sec:Experiments}.

We next consider numerically realizing the update (\ref{ZetaUpdateDiscreteRegularized}).

\subsection{The \texorpdfstring{$\bm{\zeta}$}{\zeta} Update}\label{subsec:ZetaUpdate}
The update (\ref{ZetaUpdateDiscreteRegularized}) concerns with computing the Sinkhorn regularized Wasserstein barycenter (see (\ref{DefSinkhornRegBary})) with an extra linear regularization. We have the following result (proof in Appendix \ref{AppProofThmDualOfSinkhornBaryWithLinReg}).
\begin{theorem}\label{Thm:DualOfSinkhornBaryWithLinReg}
Given $\alpha,\varepsilon > 0$, the squared Euclidean distance matrix $\bm{C}\in\mathbb{R}^{N\times N}$, and the probability vectors $\bm{\mu}_{i}^{k+1}\in\Delta^{N-1}$ for all $i\in[n]$, $k\in\mathbb{N}_{0}$, let $\bm{\Gamma}:=\exp\left(-\bm{C}/2\varepsilon\right)$. Let
\begin{align}
\left(\bm{u}_{1}^{\rm{opt}},\hdots,\bm{u}_{n}^{\rm{opt}}\right) = &\underset{\left(\bm{u}_{1},\hdots,\bm{u}_{n}\right)\in\mathbb{R}^{nN}}{\arg\min}\displaystyle\sum_{i=1}^{n}\big\langle\bm{\mu}_{i}^{k+1},\log\left(\bm{\Gamma}\exp\left(\bm{u}_{i}/\varepsilon\right)\right)\big\rangle\nonumber\\
&\quad{\rm{subject\;to}}\quad \displaystyle\sum_{i=1}^{n}\bm{u}_{i} = \frac{2}{\alpha}\bm{\nu}^{k}_{\rm{sum}}.     
\label{DualOfSinkhornBaryWithLinRegSimplified}    \end{align}
Then, the update $\bm{\zeta}^{k+1}$ in (\ref{ZetaUpdateDiscreteRegularized}) is given by
\begin{align}
\bm{\zeta}^{k+1} = \exp\left(\bm{u}_{i}^{\rm{opt}}/\varepsilon\right) \odot \left(\bm{\Gamma}\left(\bm{\mu}_{i}^{k+1} \oslash \left(\bm{\Gamma}\exp\left(\bm{u}_{i}^{\rm{opt}}/\varepsilon\right)\right)\right)\right) \:\in\:\Delta^{N-1}, \quad\text{for all}\; i\in[n].
\label{ZetaUpdateExplicit}    
\end{align}
\end{theorem}
We observe that (\ref{DualOfSinkhornBaryWithLinRegSimplified}) has a separable sum objective where each summand is a weighted log-sum-exp (thus convex). Denoting these summands as 
\begin{align}
f_{i}(\bm{u}_{i}) := \big\langle\bm{\mu}_{i}^{k+1},\log\left(\bm{\Gamma}\exp\left(\bm{u}_{i}/\varepsilon\right)\right)\big\rangle, \quad \bm{u}_{i}\in\mathbb{R}^{N}, \quad \text{for all}\; i\in[n],    
\label{deffi}    
\end{align}
we write (\ref{DualOfSinkhornBaryWithLinRegSimplified}) in the scaled ADMM form (\ref{BasicADMM}):
\begin{subequations}
\begin{align}
\bm{u}_{i}^{\ell+1} &= \prox^{\|\cdot\|_{2}}_{\frac{1}{\tau}f_{i}}\left(\bm{z}_{i}^{\ell} - \widetilde{\bm{\nu}}_{i}^{\ell}\right), \quad i\in[n],\label{OurxUpdate}\\
\bm{z}^{\ell+1} &= {\rm{proj}}_{\mathcal{C}}\left(\bm{u}^{\ell+1} + \widetilde{\bm{\nu}}^{\ell}\right), \label{OurzUpdate}\\
 \widetilde{\bm{\nu}}_{i}^{\ell+1} &= \widetilde{\bm{\nu}}_{i}^{\ell} + \left(\bm{u}^{\ell+1}_{i} - \bm{z}_{i}^{\ell+1}\right), \quad i\in[n],\label{OurdualUpdate}
\end{align}
\label{OurScaledADMM}
\end{subequations}
where $\ell\in\mathbb{N}_{0}$ is the ADMM iteration index while holding the index $k$ fixed, $\tau>0$, and $\bm{u}^{\ell}:= (\bm{u}_{1}^{\ell}, \hdots, \bm{u}_{n}^{\ell})\in\mathbb{R}^{nN}$, $\bm{z}^{\ell}:= (\bm{z}_{1}^{\ell}, \hdots, \bm{z}_{n}^{\ell})\in\mathbb{R}^{nN}$, $\widetilde{\bm{\nu}}^{\ell}:= (\widetilde{\bm{\nu}}_{1}^{\ell}, \hdots,\widetilde{\bm{\nu}}_{n}^{\ell})\in\mathbb{R}^{nN}$ for all $\ell\in\mathbb{N}_{0}$. The constraint set $\mathcal{C}$ in (\ref{OurzUpdate}) corresponds to the equality constraint in (\ref{DualOfSinkhornBaryWithLinRegSimplified}), i.e., 
\begin{align}
\mathcal{C} := \bigg\{(\bm{z}_{1}, \hdots, \bm{z}_{n})\in\mathbb{R}^{nN} \mid \bm{z}_{1} + \hdots + \bm{z}_{n} = \frac{2}{\alpha}\bm{\nu}^{k}_{\rm{sum}}\bigg\}.
\label{ConstrSet}    
\end{align}
To proceed further, we need the following Lemma (proof in Appendix \ref{AppProofOfLemma:ProjOnConstrSet}).
\begin{lemma}\label{lemma:ProjOnConstrSet}
For any $\bm{v} := (\bm{v}_{1}, \hdots, \bm{v}_{n})\in\mathbb{R}^{nN}$, where the subvectors $\bm{v}_{i}\in\mathbb{R}^{N}$ for all $i\in[n]$, let $\overline{\bm{v}} := \frac{1}{n}\sum_{i=1}^{n}\bm{v}_{i}\in\mathbb{R}^{N}$. Then the Euclidean projection of $\bm{v}$ onto $\mathcal{C}$ in (\ref{ConstrSet}) is
$${\rm{proj}}_{\mathcal{C}}\left(\bm{v}\right) = \left(\bm{v}_{1} - \overline{\bm{v}} + \frac{2}{n\alpha}\bm{\nu}^{k}_{\rm{sum}}, \hdots, \bm{v}_{n} - \overline{\bm{v}} + \frac{2}{n\alpha}\bm{\nu}^{k}_{\rm{sum}}\right)\in\mathbb{R}^{nN}.$$
\end{lemma}
Thanks to Lemma \ref{lemma:ProjOnConstrSet}, we can parallelize (\ref{OurzUpdate}) as
\begin{align}
\bm{z}_{i}^{\ell+1} = \left(\bm{u}_{i}^{\ell +1} - \frac{1}{n}\displaystyle\sum_{i=1}^{n}\bm{u}_{i}^{\ell + 1}\right) + \left(\widetilde{\bm{\nu}}_{i}^{\ell} - \frac{1}{n}\displaystyle\sum_{i=1}^{n}\widetilde{\bm{\nu}}_{i}^{\ell}\right) + \frac{2}{n\alpha}\bm{\nu}^{k}_{\rm{sum}}, \quad i\in[n].
\label{ParallelzUpdate}    
\end{align}
Therefore, (\ref{DualOfSinkhornBaryWithLinRegSimplified}) can be solved in a distributed manner:
\begin{subequations}
\begin{align}
\bm{u}_{i}^{\ell+1} &= \prox^{\|\cdot\|_{2}}_{\frac{1}{\tau}f_{i}}\left(\bm{z}_{i}^{\ell} - \widetilde{\bm{\nu}}_{i}^{\ell}\right), \quad i\in[n],\label{distributedxUpdate}\\
\bm{z}_{i}^{\ell+1} &= \left(\bm{u}_{i}^{\ell +1} - \frac{1}{n}\displaystyle\sum_{i=1}^{n}\bm{u}_{i}^{\ell + 1}\right) + \left(\widetilde{\bm{\nu}}_{i}^{\ell} - \frac{1}{n}\displaystyle\sum_{i=1}^{n}\widetilde{\bm{\nu}}_{i}^{\ell}\right) + \frac{2}{n\alpha}\bm{\nu}^{k}_{\rm{sum}}, \quad i\in[n], \label{distributedzUpdate}\\
 \widetilde{\bm{\nu}}_{i}^{\ell+1} &= \widetilde{\bm{\nu}}_{i}^{\ell} + \left(\bm{u}^{\ell+1}_{i} - \bm{z}_{i}^{\ell+1}\right), \quad i\in[n].\label{distributeddualUpdate}
\end{align}
\label{distributedADMM}
\end{subequations}
The proximal update (\ref{distributedxUpdate}) does not admit an analytical solution. To numerically compute (\ref{distributedxUpdate}), we take advantage of the structured Hessian (see Appendix \ref{AppGradHessInnerProxNewton}) of the proximal objective and implement the Newton's method with variable step size computed by backtracking line search. The recursions (\ref{distributedADMM}) comprise the inner layer ADMM in Fig. \ref{fig:GeneralBlockDiagram}.

\subsection{Summary and Convergence}\label{subsec:Algorithm}
Fig. \ref{fig:BlockDiagram} in Appendix \ref{AppAlgoDetails} provides a detailed schematic of the proposed algorithmic framework, i.e., an expanded version of Fig. \ref{fig:GeneralBlockDiagram}. A summary of the computational steps is also given in Appendix \ref{AppAlgoDetails}. In Appendix \ref{AppConvergenceInnerLayerADMM}, we provide a convergence guarantee for the ADMM (\ref{distributedADMM}). In Appendix \ref{AppGrouping}, we comment on different ways to implement the proposed algorithm depending on the number of ways to group the summand functionals in (\ref{AdditiveOptimizationMeasure}). 

%%%%%%%%%%%%%%%%%%%%%%%%%%%%%%%%%%%%%%%%%%%%%%
%%%%%%%%%%%%%%%%%%%%%%%%%%%%%%%%%%%%%%%%%%%%%%
\section{Experiments}\label{sec:Experiments}
We report two numerical experiments to illustrate the proposed framework. All simulations are performed on a MacBook Air with $1.1$ GHz Intel Core i5 CPU with $8$ GB RAM.

\textbf{Linear Fokker-Planck a.k.a. Kolmogorov’s forward PDE.} We consider computing the solution $\mu(\bm{\theta},t)$ for the IVP $\frac{\partial \mu}{\partial t}=\nabla \cdot(\mu \nabla V)+\beta^{-1} \Delta \mu,  \mu(\bm{\theta}, t=0)=\mu_{0}(\bm{\theta})$ (given) where $\bm{\theta}\equiv(\theta_1,\theta_2)\in\mathbb{R}^2$ with $V(\theta_{1},\theta_{2})=\frac{1}{4}\left(1+\theta_{1}^{4}\right)+\frac{1}{2}\left(\theta_{2}^{2}-\theta_{1}^{2}\right)$, $\beta>0$. The stationary measure $\mu_{\infty} \propto \exp (-\beta V) \differential\bm{\theta}$, which for our choice of $V$, is bimodal. 
%(see Fig. \ref{fig:mu_infty} for $\bm{\theta}\in[-2,2]^2$). 

% \begin{wrapfigure}{r}{0.14\textwidth}
% \begin{center}
%     \includegraphics[width=0.14\textwidth]{Figures/rho_inf}
%      \caption{$\mu_{\infty}$.}
%   \label{fig:mu_infty}
%   \end{center}
% \vspace*{-0.35in}  
% \end{wrapfigure}
For distributed computation, here $n=2$ and following Table \ref{Table:PDEs}, we choose $F_1(\bm{\mu}_1)=\left\langle \bm{V}_{k}, \bm{\mu}_1\right\rangle$, $F_2(\bm{\mu}_2)=\left\langle\beta^{-1} \log \bm{\mu}_2,\bm{\mu}_2\right\rangle$. The drift potential $\bm{V}_{k}(j):=V\left(\bm{\theta}_{k}^{j}\right)$ for sample index $j\in[N]$. Since $F_1$ is linear in $\bm{\mu}_1$, we use (\ref{WassProxOfLinear}) with $\bm{\Phi}_{1}(\bm{\mu}_1) = \langle \bm{V}_{k-1} + \bm{\nu}_{1}^{k}, \bm{\mu}_1\rangle$ to analytically compute the proximal updates $\bm{\mu}_1^{k+1}$, $k\in\mathbb{N}_0$. The simulation parameters are $\alpha=12$, $\tau=150$, $\beta=1$, and $\varepsilon=5 \times 10^{-2}$. To compute the proximal updates $\bm{\mu}_{2}^{k+1}$ via (\ref{MuUpdateGeneral}), we use the PROXRECUR algorithm from \citet[Sec. III-B.1]{caluya2019gradient} with algorithmic parameters $\delta=10^{-4}$, $L=20$. For doing so, we generate $N=1681$ uniform grid samples over $[-2,2]^2$, and use the initial distribution (five component mixture of Gaussians) $\bm{\mu}_{0} = \frac{1}{5}\sum_{i=1}^{5}\mathcal{N}\left(\bm{m}_{i}, \bm{\Sigma}\right)$ with $\bm{m}_{1}=(1,1)^{\top}$, $\bm{m}_{2}=(-1,-1)^{\top}$,$\bm{m}_{3}=(1,-1)^{\top}$, $\bm{m}_{4}=(-1,1)^{\top}$, $\bm{m}_{5}=(0,0)^{\top}$, $\bm{\Sigma}=0.1 \bm{I}_{2}$.

\noindent The resulting evolution of $\bm{\mu}_1$ and $\bm{\mu}_2$ are shown in Fig. \ref{fig:FPKPDE}. After $5000$ iterations of the outer layer ADMM (\ref{ADMMdiscreteRegularized}), both $\bm{\mu}_1$ and $\bm{\mu}_2$ tend to the known $\bm{\mu}_{\infty}$. %We solve (\ref{distributedxUpdate}) via the gradient descent with fixed step size $0.001$.
We performed only 3 iterations for the inner layer ADMM (\ref{distributedADMM}). The total simulation time was $99.89$ sec.
\begin{center}
\vspace*{-0.1in}
\begin{figure*}[htbp!]
    \centering
      \begin{subfigure}{0.97\linewidth}
        \includegraphics[width=0.96\linewidth]{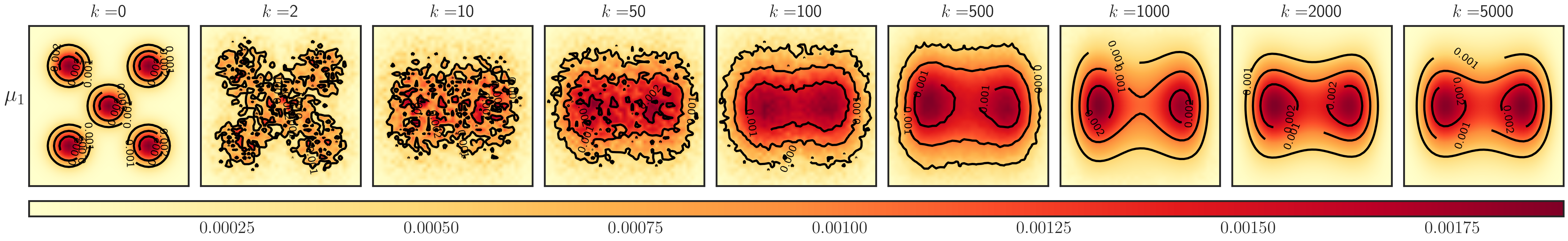}
          \caption{Contour plots of the transient solution of the joint measure $\mu_1$}
          \label{fig:mu_!}
      \end{subfigure}
        \hfill
      \begin{subfigure}{0.97\linewidth}
        \includegraphics[width=0.96\linewidth]{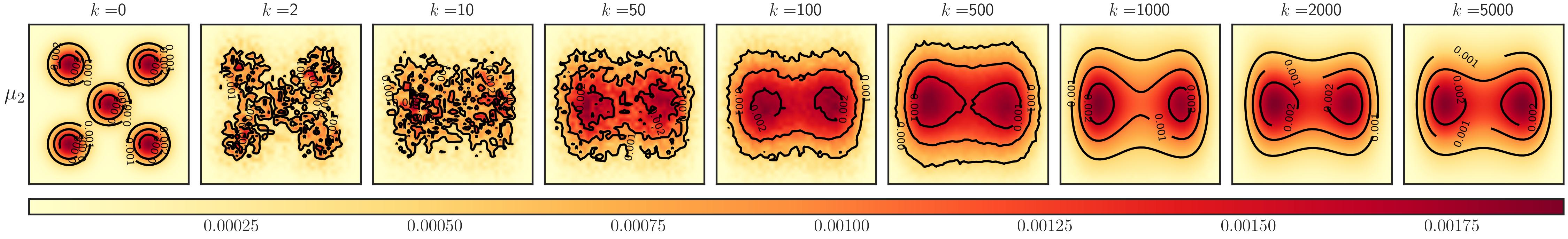}
          \caption{Contour plots of the transient solution of  the joint measure $\mu_2$}
          \label{fig:mu_2}
      \end{subfigure}
\caption{%
Distributed computation for solving the linear Fokker–Planck IVP over domain $[-2,2]^{2}$. Color denotes the value of the plotted variable; see colorbar (dark red = high, light yellow = low).}
\label{fig:FPKPDE}
\vspace*{-0.2in}
\end{figure*}    
\end{center}

\textbf{Aggregation-drift-diffusion nonlinear PDE.} We next consider solving a nonlinear PDE IVP $\frac{\partial \mu}{\partial t}=\nabla \cdot(\mu \nabla \left(U \oast \mu\right))+\nabla \cdot(\mu \nabla V) +\beta^{-1}\Delta \mu^{2}$ with $\bm{\theta}\in\mathbb{R}^{2}$, the same $\mu_0$ as in the previous example, $U(\bm{\theta})=\frac{1}{2}\|\bm{\theta}\|_{2}^{2}-\ln \|\bm{\theta}\|_2$, and $V(\bm{\theta})=-\frac{1}{4}\ln\|\bm{\theta}\|_2$. As $\beta^{-1} \downarrow 0$, the stationary solution $\mu_{\infty}$ is a uniform measure over annulus \citep[Sec. 4.3.2]{carrillo2022primal} with the inner and outer radii of $R_{i}=1/2$ and $R_{o}=\sqrt{5}/2$, respectively. To avoid evaluation of $U$ and $V$ at $\bm{\theta} = \bm{0}$, we set $U(\bm{0})$ and $V(\bm{0})$ to be equal to the respective average values of $U$ and $V$ on the cell of width $2h$ centered at $(0,0)$. In our simulation, $h=5\times 10^{-3}$.

Here, we have three spatial operators: interaction $\nabla \cdot(\mu \nabla \left(U \oast \mu\right))$, drift $\nabla \cdot(\mu \nabla V)$, and diffusion $\beta^{-1}\Delta \mu^{2}$. In Appendix \ref{AppAggregationDriftDiffusionNonlinearPDE}, we detail four different ways of splitting the operators and present quantitative results for each case. For the splitting $F_{1}(\bm{\mu}_1) = \langle\bm{V}_{k} + \beta^{-1}\log\bm{\mu}_1,\bm{\mu}_1\rangle$, $F_{2}(\bm{\mu}_2) = \langle\bm{U}_{k}\bm{\mu}_{2}^{k},\bm{\mu}_{2}\rangle$, the evolution of $\bm{\mu}_1$ and $\bm{\mu}_2$ are shown in Fig. \ref{fig:FPKPDE_exp2} which match with each other and with the annulus mentioned before. Appendix \ref{AppAggregationDriftDiffusionNonlinearPDE} provides more details on this numerical experiment.

\begin{figure*}[htbp!]
    \centering
      \begin{subfigure}{0.97\linewidth}
        \includegraphics[width=0.96\linewidth]{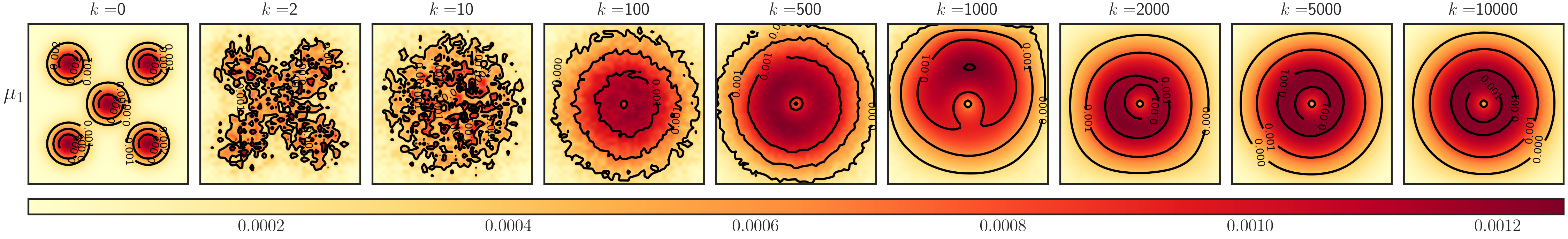}
          \caption{Contour plots of the transient solution of the joint measure $\mu_1$}
          \label{fig:mu1_exp2}
      \end{subfigure}
        \hfill
      \begin{subfigure}{0.97\linewidth}
        \includegraphics[width=0.96\linewidth]{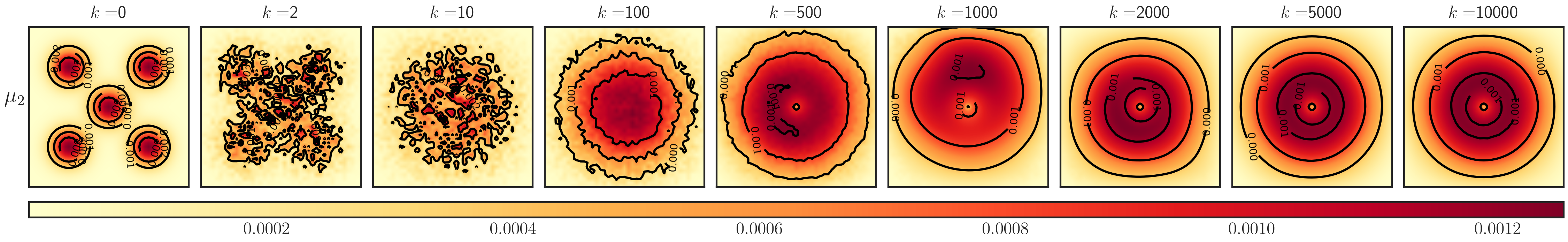}
          \caption{Contour plots of the transient solution of  the joint measure $\mu_2$}
          \label{fig:mu2_exp2}
      \end{subfigure}
\caption{%
Distributed computation for solving the nonlinear aggregation-drift-diffusion IVP over domain $[-2,2]^{2}$ for $F_{1}(\bm{\mu}_1) = \langle\bm{V}_{k} + \beta^{-1}\log\bm{\mu}_1,\bm{\mu}_1\rangle$, $F_{2}(\bm{\mu}_2) = \langle\bm{U}_{k}\bm{\mu}_{2}^{k},\bm{\mu}_{2}\rangle$. Color denotes the value of the plotted variable; see colorbar (dark red = high, light yellow = low).}
\label{fig:FPKPDE_exp2}
\vspace*{-0.2in}
\end{figure*}

%%%%%%%%%%%%%%%%%%%%%%%%%%%%%%%%%%%%%%%%%%%%%%
%%%%%%%%%%%%%%%%%%%%%%%%%%%%%%%%%%%%%%%%%%%%%%

\section{Conclusions}\label{sec:Conclusions}
We present a novel computational framework to solve measure-valued optimization problems with additive objective via distributed computation. Our findings provide new insights in generalizing the well-known finite dimensional Euclidean ADMM to its Wasserstein and Sinkhorn counterparts, and open up the possibility of designing measure-valued operator splitting algorithms. The proposed framework leverages existing proximal and Jordan-Kinderlehrer-Otto (JKO) schemes. Its feasibility is demonstrated via illustrative numerical experiments. While we provided convergence guarantee (Appendix \ref{AppConvergenceInnerLayerADMM}) for the proposed inner layer ADMM, an important undertaking not pursued here is the convergence guarantee for the overall scheme. This will be the topic of our future work.   

%%%%%%%%%%%%%%%%%%%%%%%%%%%%%%%%%%%%%%%%%%%%%%
%%%%%%%%%%%%%%%%%%%%%%%%%%%%%%%%%%%%%%%%%%%%%%

\newpage

\bibliography{references}

\begin{thebibliography}{58}
\providecommand{\natexlab}[1]{#1}
\providecommand{\url}[1]{\texttt{#1}}
\expandafter\ifx\csname urlstyle\endcsname\relax
  \providecommand{\doi}[1]{doi: #1}\else
  \providecommand{\doi}{doi: \begingroup \urlstyle{rm}\Url}\fi

\bibitem[Agueh \& Carlier(2011)Agueh and Carlier]{agueh2011barycenters}
Martial Agueh and Guillaume Carlier.
\newblock Barycenters in the wasserstein space.
\newblock \emph{SIAM Journal on Mathematical Analysis}, 43\penalty0
  (2):\penalty0 904--924, 2011.

\bibitem[Alvarez-Melis et~al.(2021)Alvarez-Melis, Schiff, and
  Mroueh]{alvarez2021optimizing}
David Alvarez-Melis, Yair Schiff, and Youssef Mroueh.
\newblock Optimizing functionals on the space of probabilities with input
  convex neural networks.
\newblock \emph{arXiv preprint arXiv:2106.00774}, 2021.

\bibitem[Ambrosio et~al.(2008)Ambrosio, Gigli, and
  Savar{\'e}]{ambrosio2008gradient}
Luigi Ambrosio, Nicola Gigli, and Giuseppe Savar{\'e}.
\newblock \emph{Gradient flows: in metric spaces and in the space of
  probability measures}.
\newblock Springer Science \& Business Media, 2008.

\bibitem[Arqu{\'e} et~al.(2022)Arqu{\'e}, Uribe, and
  Ocampo-Martinez]{arque2022approximate}
Ferran Arqu{\'e}, C{\'e}sar~A Uribe, and Carlos Ocampo-Martinez.
\newblock Approximate {W}asserstein attraction flows for dynamic mass transport
  over networks.
\newblock \emph{Automatica}, 143:\penalty0 110432, 2022.

\bibitem[Bauschke \& Kruk(2004)Bauschke and Kruk]{bauschke2004reflection}
HH~Bauschke and SG~Kruk.
\newblock Reflection-projection method for convex feasibility problems with an
  obtuse cone.
\newblock \emph{Journal of Optimization Theory and Applications}, 120\penalty0
  (3):\penalty0 503--531, 2004.

\bibitem[Bell(1938)]{bell1938iterated}
Eric~Temple Bell.
\newblock The iterated exponential integers.
\newblock \emph{Annals of Mathematics}, pp.\  539--557, 1938.

\bibitem[Benamou et~al.(2015)Benamou, Carlier, Cuturi, Nenna, and
  Peyr{\'e}]{benamou2015iterative}
Jean-David Benamou, Guillaume Carlier, Marco Cuturi, Luca Nenna, and Gabriel
  Peyr{\'e}.
\newblock Iterative {B}regman projections for regularized transportation
  problems.
\newblock \emph{SIAM Journal on Scientific Computing}, 37\penalty0
  (2):\penalty0 A1111--A1138, 2015.

\bibitem[Benamou et~al.(2016)Benamou, Carlier, and
  Laborde]{benamou2016augmented}
Jean-David Benamou, Guillaume Carlier, and Maxime Laborde.
\newblock An augmented {Lagrangian} approach to {W}asserstein gradient flows
  and applications.
\newblock \emph{ESAIM: Proceedings and surveys}, 54:\penalty0 1--17, 2016.

\bibitem[Bernton(2018)]{bernton2018langevin}
Espen Bernton.
\newblock Langevin monte carlo and {JKO} splitting.
\newblock In \emph{Conference on learning theory}, pp.\  1777--1798. PMLR,
  2018.

\bibitem[Bowles \& Agueh(2015)Bowles and Agueh]{bowles2015weak}
Malcolm Bowles and Martial Agueh.
\newblock Weak solutions to a fractional {Fokker--Planck} equation via
  splitting and {W}asserstein gradient flow.
\newblock \emph{Applied Mathematics Letters}, 42:\penalty0 30--35, 2015.

\bibitem[Boyd et~al.(2004)Boyd, Boyd, and Vandenberghe]{boyd2004convex}
Stephen Boyd, Stephen~P Boyd, and Lieven Vandenberghe.
\newblock \emph{Convex optimization}.
\newblock Cambridge university press, 2004.

\bibitem[Boyd et~al.(2011)Boyd, Parikh, and Chu]{boyd2011distributed}
Stephen Boyd, Neal Parikh, and Eric Chu.
\newblock \emph{Distributed optimization and statistical learning via the
  alternating direction method of multipliers}.
\newblock Now Publishers Inc, 2011.

\bibitem[Brenier(1991)]{brenier1991polar}
Yann Brenier.
\newblock Polar factorization and monotone rearrangement of vector-valued
  functions.
\newblock \emph{Communications on pure and applied mathematics}, 44\penalty0
  (4):\penalty0 375--417, 1991.

\bibitem[Bunne et~al.(2022)Bunne, Papaxanthos, Krause, and
  Cuturi]{bunne2022proximal}
Charlotte Bunne, Laetitia Papaxanthos, Andreas Krause, and Marco Cuturi.
\newblock Proximal optimal transport modeling of population dynamics.
\newblock In \emph{International Conference on Artificial Intelligence and
  Statistics}, pp.\  6511--6528. PMLR, 2022.

\bibitem[Butnariu \& Iusem(2000)Butnariu and Iusem]{butnariu2000totally}
Dan Butnariu and Alfredo~N Iusem.
\newblock \emph{Totally convex functions for fixed points computation and
  infinite dimensional optimization}, volume~40.
\newblock Springer Science \& Business Media, 2000.

\bibitem[Caluya \& Halder(2021)Caluya and Halder]{caluya2021wasserstein}
Kenneth Caluya and Abhishek Halder.
\newblock Wasserstein proximal algorithms for the {S}chr{\"o}dinger bridge
  problem: {D}ensity control with nonlinear drift.
\newblock \emph{IEEE Transactions on Automatic Control}, 2021.

\bibitem[Caluya \& Halder(2019)Caluya and Halder]{caluya2019gradient}
Kenneth~F Caluya and Abhishek Halder.
\newblock Gradient flow algorithms for density propagation in stochastic
  systems.
\newblock \emph{IEEE Transactions on Automatic Control}, 65\penalty0
  (10):\penalty0 3991--4004, 2019.

\bibitem[Carlier et~al.(2017)Carlier, Duval, Peyr{\'e}, and
  Schmitzer]{carlier2017convergence}
Guillaume Carlier, Vincent Duval, Gabriel Peyr{\'e}, and Bernhard Schmitzer.
\newblock Convergence of entropic schemes for optimal transport and gradient
  flows.
\newblock \emph{SIAM Journal on Mathematical Analysis}, 49\penalty0
  (2):\penalty0 1385--1418, 2017.

\bibitem[Carrillo et~al.(2022)Carrillo, Craig, Wang, and
  Wei]{carrillo2022primal}
Jos{\'e}~A Carrillo, Katy Craig, Li~Wang, and Chaozhen Wei.
\newblock Primal dual methods for {W}asserstein gradient flows.
\newblock \emph{Foundations of Computational Mathematics}, 22\penalty0
  (2):\penalty0 389--443, 2022.

\bibitem[Chizat \& Bach(2018)Chizat and Bach]{chizat2018global}
Lenaic Chizat and Francis Bach.
\newblock On the global convergence of gradient descent for over-parameterized
  models using optimal transport.
\newblock \emph{Advances in neural information processing systems}, 31, 2018.

\bibitem[Chu et~al.(2019)Chu, Blanchet, and Glynn]{chu2019probability}
Casey Chu, Jose Blanchet, and Peter Glynn.
\newblock Probability functional descent: {A} unifying perspective on {GAN}s,
  variational inference, and reinforcement learning.
\newblock In \emph{International Conference on Machine Learning}, pp.\
  1213--1222. PMLR, 2019.

\bibitem[Cuturi(2013)]{cuturi2013sinkhorn}
Marco Cuturi.
\newblock Sinkhorn distances: {L}ightspeed computation of optimal transport.
\newblock \emph{Advances in neural information processing systems},
  26:\penalty0 2292--2300, 2013.

\bibitem[Cuturi \& Peyr{\'e}(2016)Cuturi and Peyr{\'e}]{cuturi2016smoothed}
Marco Cuturi and Gabriel Peyr{\'e}.
\newblock A smoothed dual approach for variational {W}asserstein problems.
\newblock \emph{SIAM Journal on Imaging Sciences}, 9\penalty0 (1):\penalty0
  320--343, 2016.

\bibitem[Domingo-Enrich et~al.(2020)Domingo-Enrich, Jelassi, Mensch, Rotskoff,
  and Bruna]{domingo2020mean}
Carles Domingo-Enrich, S~Jelassi, A~Mensch, G~Rotskoff, and J~Bruna.
\newblock A mean-field analysis of two-player zero-sum games.
\newblock \emph{Advances in neural information processing systems}, 2020.

\bibitem[Dvurechenskii et~al.(2018)Dvurechenskii, Dvinskikh, Gasnikov, Uribe,
  and Nedich]{dvurechenskii2018decentralize}
Pavel Dvurechenskii, Darina Dvinskikh, Alexander Gasnikov, Cesar Uribe, and
  Angelia Nedich.
\newblock Decentralize and randomize: {F}aster algorithm for {W}asserstein
  barycenters.
\newblock \emph{Advances in Neural Information Processing Systems}, 31, 2018.

\bibitem[Fan et~al.(2022)Fan, Zhang, Taghvaei, and Chen]{fan2022variational}
Jiaojiao Fan, Qinsheng Zhang, Amirhossein Taghvaei, and Yongxin Chen.
\newblock Variational {W}asserstein gradient flow.
\newblock In \emph{International Conference on Machine Learning}, pp.\
  6185--6215. PMLR, 2022.

\bibitem[Frogner \& Poggio(2020)Frogner and Poggio]{frogner2020approximate}
Charlie Frogner and Tomaso Poggio.
\newblock Approximate inference with {W}asserstein gradient flows.
\newblock In \emph{International Conference on Artificial Intelligence and
  Statistics}, pp.\  2581--2590. PMLR, 2020.

\bibitem[Gabay \& Mercier(1976)Gabay and Mercier]{gabay1976dual}
Daniel Gabay and Bertrand Mercier.
\newblock A dual algorithm for the solution of nonlinear variational problems
  via finite element approximation.
\newblock \emph{Computers \& mathematics with applications}, 2\penalty0
  (1):\penalty0 17--40, 1976.

\bibitem[Gallou{\"e}t \& Monsaingeon(2017)Gallou{\"e}t and
  Monsaingeon]{gallouet2017jko}
Thomas~O Gallou{\"e}t and Leonard Monsaingeon.
\newblock A {JKO} splitting scheme for {Kantorovich--Fisher--Rao} gradient
  flows.
\newblock \emph{SIAM Journal on Mathematical Analysis}, 49\penalty0
  (2):\penalty0 1100--1130, 2017.

\bibitem[Genevay et~al.(2016)Genevay, Cuturi, Peyr{\'e}, and
  Bach]{genevay2016stochastic}
Aude Genevay, Marco Cuturi, Gabriel Peyr{\'e}, and Francis Bach.
\newblock Stochastic optimization for large-scale optimal transport.
\newblock In \emph{NIPS 2016-Thirtieth Annual Conference on Neural Information
  Processing System}, pp.\  3440--3448, 2016.

\bibitem[Glowinski \& Le~Tallec(1989)Glowinski and
  Le~Tallec]{glowinski1989augmented}
Roland Glowinski and Patrick Le~Tallec.
\newblock \emph{Augmented {L}agrangian and operator-splitting methods in
  nonlinear mechanics}.
\newblock SIAM, 1989.

\bibitem[Glowinski et~al.(2016)Glowinski, Pan, and Tai]{glowinski2016some}
Roland Glowinski, Tsorng-Whay Pan, and Xue-Cheng Tai.
\newblock Some facts about operator-splitting and alternating direction
  methods.
\newblock \emph{Splitting Methods in Communication, Imaging, Science, and
  Engineering}, pp.\  19--94, 2016.

\bibitem[Graham et~al.(1988)Graham, Knuth, and Patashnik]{graham1988concrete}
RL~Graham, DE~Knuth, and O~Patashnik.
\newblock Concrete mathematics, 1988.

\bibitem[Hong et~al.(2016)Hong, Luo, and Razaviyayn]{hong2016convergence}
Mingyi Hong, Zhi-Quan Luo, and Meisam Razaviyayn.
\newblock Convergence analysis of alternating direction method of multipliers
  for a family of nonconvex problems.
\newblock \emph{SIAM Journal on Optimization}, 26\penalty0 (1):\penalty0
  337--364, 2016.

\bibitem[Ito \& Kunisch(1990)Ito and Kunisch]{ito1990augmented}
Kazufumi Ito and Karl Kunisch.
\newblock The augmented {L}agrangian method for equality and inequality
  constraints in {H}ilbert spaces.
\newblock \emph{Mathematical programming}, 46\penalty0 (1-3):\penalty0
  341--360, 1990.

\bibitem[Jordan et~al.(1998)Jordan, Kinderlehrer, and
  Otto]{jordan1998variational}
Richard Jordan, David Kinderlehrer, and Felix Otto.
\newblock The variational formulation of the {Fokker--Planck} equation.
\newblock \emph{SIAM journal on mathematical analysis}, 29\penalty0
  (1):\penalty0 1--17, 1998.

\bibitem[Kanzow et~al.(2018)Kanzow, Steck, and Wachsmuth]{kanzow2018augmented}
Christian Kanzow, Daniel Steck, and Daniel Wachsmuth.
\newblock An augmented {L}agrangian method for optimization problems in
  {B}anach spaces.
\newblock \emph{SIAM Journal on Control and Optimization}, 56\penalty0
  (1):\penalty0 272--291, 2018.

\bibitem[Karlsson \& Ringh(2017)Karlsson and Ringh]{karlsson2017generalized}
Johan Karlsson and Axel Ringh.
\newblock Generalized {S}inkhorn iterations for regularizing inverse problems
  using optimal mass transport.
\newblock \emph{SIAM Journal on Imaging Sciences}, 10\penalty0 (4):\penalty0
  1935--1962, 2017.

\bibitem[Kent et~al.(2021)Kent, Li, Blanchet, and Glynn]{kent2021modified}
Carson Kent, Jiajin Li, Jose Blanchet, and Peter~W Glynn.
\newblock Modified {Frank Wolfe} in probability space.
\newblock \emph{Advances in Neural Information Processing Systems},
  34:\penalty0 14448--14462, 2021.

\bibitem[Laborde(2017)]{laborde201712}
Maxime Laborde.
\newblock On some nonlinear evolution systems which are perturbations of
  {W}asserstein gradient flows.
\newblock \emph{Topological Optimization and Optimal Transport: In the Applied
  Sciences}, 17:\penalty0 304, 2017.

\bibitem[Lemmens \& Nussbaum(2012)Lemmens and Nussbaum]{lemmens2012nonlinear}
Bas Lemmens and Roger Nussbaum.
\newblock \emph{Nonlinear {Perron-Frobenius} Theory}, volume 189.
\newblock Cambridge University Press, 2012.

\bibitem[Mei et~al.(2018)Mei, Montanari, and Nguyen]{mei2018mean}
Song Mei, Andrea Montanari, and Phan-Minh Nguyen.
\newblock A mean field view of the landscape of two-layer neural networks.
\newblock \emph{Proceedings of the National Academy of Sciences}, 115\penalty0
  (33):\penalty0 E7665--E7671, 2018.

\bibitem[Mokrov et~al.(2021)Mokrov, Korotin, Li, Genevay, Solomon, and
  Burnaev]{mokrov2021large}
Petr Mokrov, Alexander Korotin, Lingxiao Li, Aude Genevay, Justin~M Solomon,
  and Evgeny Burnaev.
\newblock Large-scale {W}asserstein gradient flows.
\newblock \emph{Advances in Neural Information Processing Systems},
  34:\penalty0 15243--15256, 2021.

\bibitem[Nesterov(2003)]{nesterov2003introductory}
Yurii Nesterov.
\newblock \emph{Introductory lectures on convex optimization: {A} basic
  course}, volume~87.
\newblock Springer Science and Business Media, 2003.

\bibitem[Nishihara et~al.(2015)Nishihara, Lessard, Recht, Packard, and
  Jordan]{nishihara2015general}
Robert Nishihara, Laurent Lessard, Ben Recht, Andrew Packard, and Michael
  Jordan.
\newblock A general analysis of the convergence of {ADMM}.
\newblock In \emph{International Conference on Machine Learning}, pp.\
  343--352. PMLR, 2015.

\bibitem[Parikh \& Boyd(2014)Parikh and Boyd]{parikh2014proximal}
Neal Parikh and Stephen Boyd.
\newblock Proximal algorithms.
\newblock \emph{Foundations and Trends in optimization}, 1\penalty0
  (3):\penalty0 127--239, 2014.

\bibitem[Peyr{\'e}(2015)]{peyre2015entropic}
Gabriel Peyr{\'e}.
\newblock Entropic approximation of {W}asserstein gradient flows.
\newblock \emph{SIAM Journal on Imaging Sciences}, 8\penalty0 (4):\penalty0
  2323--2351, 2015.

\bibitem[Salim et~al.(2020)Salim, Korba, and Luise]{salim2020wasserstein}
Adil Salim, Anna Korba, and Giulia Luise.
\newblock The {W}asserstein proximal gradient algorithm.
\newblock \emph{Advances in Neural Information Processing Systems},
  33:\penalty0 12356--12366, 2020.

\bibitem[Santambrogio(2017)]{santambrogio2017euclidean}
Filippo Santambrogio.
\newblock $\{$Euclidean, metric, and {W}asserstein$\}$ gradient flows: an
  overview.
\newblock \emph{Bulletin of Mathematical Sciences}, 7\penalty0 (1):\penalty0
  87--154, 2017.

\bibitem[Sirignano \& Spiliopoulos(2020)Sirignano and
  Spiliopoulos]{sirignano2020mean}
Justin Sirignano and Konstantinos Spiliopoulos.
\newblock Mean field analysis of neural networks: A central limit theorem.
\newblock \emph{Stochastic Processes and their Applications}, 130\penalty0
  (3):\penalty0 1820--1852, 2020.

\bibitem[Thompson(1963)]{thompson1963certain}
Anthony~C Thompson.
\newblock On certain contraction mappings in a partially ordered vector space.
\newblock \emph{Proceedings of the American Mathematical Society}, 14\penalty0
  (3):\penalty0 438--443, 1963.

\bibitem[Villani(2003)]{villani2003topics}
C{\'e}dric Villani.
\newblock \emph{Topics in optimal transportation}, volume~58.
\newblock American Mathematical Soc., 1st edition, 2003.

\bibitem[Villani(2009)]{villani2009optimal}
C{\'e}dric Villani.
\newblock \emph{Optimal transport: old and new}, volume 338.
\newblock Springer, 2009.

\bibitem[Wang \& Li(2022)Wang and Li]{wang2022accelerated}
Yifei Wang and Wuchen Li.
\newblock Accelerated information gradient flow.
\newblock \emph{Journal of Scientific Computing}, 90:\penalty0 1--47, 2022.

\bibitem[Wang et~al.(2019)Wang, Yin, and Zeng]{wang2019global}
Yu~Wang, Wotao Yin, and Jinshan Zeng.
\newblock Global convergence of {ADMM} in nonconvex nonsmooth optimization.
\newblock \emph{Journal of Scientific Computing}, 78\penalty0 (1):\penalty0
  29--63, 2019.

\bibitem[Wibisono(2018)]{wibisono2018sampling}
Andre Wibisono.
\newblock Sampling as optimization in the space of measures: The {L}angevin
  dynamics as a composite optimization problem.
\newblock In \emph{Conference on Learning Theory}, pp.\  2093--3027. PMLR,
  2018.

\bibitem[Yang et~al.(2021)Yang, Li, Sun, and Toh]{yang2021fast}
Lei Yang, Jia Li, Defeng Sun, and Kim-Chuan Toh.
\newblock A fast globally linearly convergent algorithm for the computation of
  {W}asserstein barycenters.
\newblock \emph{The Journal of Machine Learning Research}, 22\penalty0
  (1):\penalty0 984--1020, 2021.

\bibitem[Zhang et~al.(2018)Zhang, Chen, Li, and Carin]{zhang2018policy}
Ruiyi Zhang, Changyou Chen, Chunyuan Li, and Lawrence Carin.
\newblock Policy optimization as {W}asserstein gradient flows.
\newblock In \emph{International Conference on Machine Learning}, pp.\
  5737--5746. PMLR, 2018.

\end{thebibliography}
\bibliographystyle{iclr2023_conference}

%%%%%%%%%%%%%%%%%%%%%%%%%%%%%%%%%%%%%%%%%%%%%%%%%%%%%%%%%%%%%%%%%%%%%%%%%%%%%%%%%%%%%%%%%%%%%%

\newpage

\appendix

\section{Examples of $F_i$ and Wasserstein Gradient Flows}\label{AppExampleFi}
In this Section, we provide specific examples of $F_i,\Phi_i$ in (\ref{FiPhii}), and the associated Wasserstein gradient flows (WGFs) (\ref{WassGradFlow}). 

We denote the base space as $\mathcal{X}\subseteq\mathbb{R}^{d}$ and its element as $\bm{\theta}\in\mathcal{X}$. For fixed $i\in[n]$, important examples of $F_{i}$ include $\int_{\mathcal{X}} V(\bm{\theta}) \differential\mu_{i}(\bm{\theta})$ (potential energy for some suitable advection potential $V$), $\beta^{-1}\int_{\mathcal{X}} \log\mu_{i}(\bm{\theta}) \differential\mu_{i}(\bm{\theta})$ (logarithmic internal energy with the ``inverse temperature'' parameter $\beta>0$), $\int_{{\mathcal{X}}\times\mathcal{X}} U(\bm{\theta},\bm{\sigma}) \differential\mu_{i}(\bm{\theta}) \differential\mu_{i}(\bm{\sigma})$ (interaction energy for some symmetric positive definite interaction potential $U$), and $(\beta^{-1}/(m-1))\int_{\mathcal{X}}\mu_{i}^{m-1}(\bm{\theta})\differential\mu_i(\bm{\theta})$ (power law internal energy). 

In Table \ref{Table:PDEs}, we summarize how the WGF (\ref{WassGradFlow}) specializes in such cases. In particular, the PDEs in the second column of Table \ref{Table:PDEs} are well known: the Liouville advection PDE (first row), the Fokker-Planck a.k.a. Kolmogorov's forward advection-diffusion PDE (second row), the advection-aggregation a.k.a. propagation of chaos PDE (third row), and the porus medium a.k.a. advection-nonlinear power law diffusion PDE (fourth row).

\begin{table}[h]
\centering
 \begin{tabular}{| l | l |} 
 \hline
 $\Phi_{i}(\cdot)=F_{i}(\cdot) + \int \nu_{i}^{k}\differential(\cdot)$ & WGF (\ref{WassGradFlow})\\  
 \hline\hline
 $\int_{\mathcal{X}} \left(V(\bm{\theta}) + \nu_{i}^{k}(\bm{\theta})\right)\differential\mu_{i}(\bm{\theta})$\! & $\dfrac{\partial\widetilde{\mu}_{i}}{\partial t} = \nabla\cdot\left(\widetilde{\mu}_{i}\left(\nabla V + \nabla\nu_{i}^{k}\right)\right)$\!\\[1.5ex]
 \hline
$\int_{\mathcal{X}} \left(\nu_{i}^{k}(\bm{\theta})+\beta^{-1}\log\mu_{i}(\bm{\theta})\right) \differential\mu_{i}(\bm{\theta})$\! & $\dfrac{\partial\widetilde{\mu}_{i}}{\partial t} = \nabla\cdot\left(\widetilde{\mu}_{i}\nabla\nu_{i}^{k}\right) + \beta^{-1}\Delta\widetilde{\mu}_{i}$\!\\[1.5ex]
\hline
$\int_{\mathcal{X}}\nu_{i}^{k}(\bm{\theta})\differential\mu_{i}(\bm{\theta}) + \int_{\mathcal{X}\times\mathcal{X}} U(\bm{\theta},\bm{\sigma}) \differential\mu_{i}(\bm{\theta}) \differential\mu_{i}(\bm{\sigma})$\! & $\dfrac{\partial\widetilde{\mu}_{i}}{\partial t} = \nabla\cdot\left(\widetilde{\mu}_{i}\left(\nabla\nu_{i}^{k} + \nabla\left(U \oast \widetilde{\mu}_{i}\right)\right)\right)$\! \\[1.5ex]
\hline
$\int_{\mathcal{X}} \left( \nu_{i}^{k}(\bm{\theta}) +\frac{\beta^{-1}}{m-1} \mu_{i}^{m-1}\right)\differential\mu_{i}(\bm{\theta}), m>1$\!& $\dfrac{\partial\widetilde{\mu}_{i}}{\partial t} = \nabla\cdot\left(\widetilde{\mu}_{i}\nabla\nu_{i}^{k}\right) + \beta^{-1}\Delta\widetilde{\mu}_{i}^{m}$\!\\[1.5ex]
\hline
 \end{tabular}
\vspace*{0.1in} 
\caption{{\small{Specific instances of the WGF (\ref{WassGradFlow}) for different choices of $F_{i}$, and hence $\Phi_{i}$. The Euclidean gradient operator $\nabla$ is w.r.t. $\bm{\theta}\in\mathcal{X}$. The operator $\oast$ can be seen as a generalized convolution, given by $(U\oast \widetilde{\mu}_{i})(\bm{\theta}):=\int_{\mathcal{X}} U(\bm{\theta},\bm{\sigma}) \differential\widetilde{\mu}_{i}(\bm{\sigma})$ where $U(\bm{\theta},\bm{\sigma})$ is symmetric and positive definite for all $(\bm{\theta},\bm{\sigma})\in\mathcal{X}\times\mathcal{X}\subseteq\mathbb{R}^{d}\times\mathbb{R}^{d}$.}}}
\vspace*{-0.1in}
\label{Table:PDEs}
\end{table}

We emphasize here that different from standard WGF literature, the functionals $\Phi_i(\cdot)$ listed in Table \ref{Table:PDEs} are a sum of two functionals: a ``physical" free energy functional $F_i(\cdot)$ (e.g., advection, dffusion, interaction), and an ``algorithmic" linear functional $\int \nu_i^{k}\differential(\cdot)$ that specifically arises from our consensus constraint. The latter is an algorithmic construct and has no physical meaning.

%%%%%%%%%%%%%%%%%%%%%%%%%%%%%%%%%%%%%%%%%%%%%%%%%%%%%%%%%%%%%%%%%%%%%%%%%%%%%%%%%%%%%%%%%%%%%%

\section{Examples of \texorpdfstring{$\bm{\mu}$}{\mu} Updates}\label{AppExampleMuUpdate}
In this Section, we exemplify the usage of Proposition \ref{prop:ProximalUpdateGeneral} for several functionals of practical interest.

\textbf{Example: $F_{i}(\bm{\mu}_{i}) = \beta^{-1}\langle\log\bm{\mu}_{i},\bm{\mu}_{i}\rangle$, $\beta>0$.}

As pointed out in Table \ref{Table:PDEs} second row, this specific choice of $F_{i}$ correspond to WGF with advection and linear diffusion. In this case, Proposition \ref{prop:ProximalUpdateGeneral} reduces exactly to \citep[Theorem 1]{caluya2019gradient} allowing further simplification of (\ref{ZeroInSubdifferential}). In particular, the system (\ref{lambda0lambda1equations}) can be solved via certain cone-preserving block coordinate iteration proposed in \citep[Sec. III.B,C]{caluya2019gradient} that is \emph{provably contractive w.r.t. the Thompson metric} \citep{thompson1963certain} \citep[Ch. 2.1]{lemmens2012nonlinear}. Consequently, the block coordinate iteration is guaranteed to converge to a unique pair $\left(\bm{\lambda}_{0i}^{\text{opt}},\bm{\lambda}_{1i}^{\text{opt}}\right)$ with linear rate of convergence. This makes the proximal update (\ref{MuUpdateGeneral}) semi-analytical in the sense the pair $\left(\bm{\lambda}_{0i}^{\text{opt}},\bm{\lambda}_{1i}^{\text{opt}}\right)$ can be numerically computed by performing the contractive block coordinate iteration while ``freezing'' the index $k\in\mathbb{N}_{0}$. With the converged pair $\left(\bm{\lambda}_{0i}^{\text{opt}},\bm{\lambda}_{1i}^{\text{opt}}\right)$, the evaluation (\ref{MuUpdateGeneral}) is analytical for each $k\in\mathbb{N}_{0}$.  

\textbf{Example: $F_{i}(\bm{\mu}_{i}) = \langle \bm{V}, \bm{\mu}_i\rangle$.}

When $F_{i}$ and hence $G_{i}$ in (\ref{DefGi}), is linear in $\bm{\mu}_{i}$, the proximal update $\bm{\mu}_{i}^{k+1}$ can be computed analytically, obviating the zero order hold sub-iterations mentioned in the previous example. We summarize this novel result in the following Theorem \ref{Thm:WassProxOfLinear}. Notice in particular that the case of advection PDE shown in the first row of Table \ref{Table:PDEs} can be treated via Theorem \ref{Thm:WassProxOfLinear} with $\Phi_{i}(\bm{\mu}_i) = \langle \bm{V} + \bm{\nu}_{i}^{k}, \bm{\mu}_i\rangle$, for given $i\in[n]$. In this discrete version, $\bm{V}\in\mathbb{R}^{N}$ is the advection potential evaluated at the $N$ sample locations in $\mathbb{R}^{d}$. 
\begin{theorem}\label{Thm:WassProxOfLinear}
Given $\bm{a}\in\mathbb{R}^{N}\setminus\{\bm{0}\}$, let $\Phi(\bm{\mu}) := \langle\bm{a},\bm{\mu}\rangle$ for $\bm{\mu}\in\Delta^{N-1}$. Let $\bm{C}\in\mathbb{R}^{N\times N}$ be the squared Euclidean distance matrix, and for $\varepsilon>0$, let $\bm{\Gamma}:=\exp\left(-\bm{C}/2\varepsilon\right)$. For any $\bm{\zeta}\in\Delta^{N-1}$, $\alpha > 0$, the proximal operator
\begin{align}
\prox^{W_{\varepsilon}}_{\frac{1}{\alpha}\Phi}\left(\bm{\zeta}\right) = \exp\left(-\dfrac{1}{\alpha\varepsilon}\bm{a}\right) \odot \left(\bm{\Gamma}^{\top}\left(\bm{\zeta}\oslash\left(\bm{\Gamma}\exp\left(-\dfrac{1}{\alpha\varepsilon}\bm{a}\right)\right)\right)\right).
\label{WassProxOfLinear}    
\end{align}
\end{theorem}
\begin{proof}
We start from (\ref{LagrangeDualProblem}) by dropping the indices $i$ and $k$, and set $G(\bm{\mu})=\langle\bm{a},\bm{\mu}\rangle$, where $\bm{a}\in\mathbb{R}^{N}\setminus\{\bm{0}\}$.

For notational ease, let $\bm{y}:= \exp\left(\frac{\bm{\lambda}_{0}}{\alpha\varepsilon}\right)\in\mathbb{R}^{N}_{>0}, \bm{z} := \exp\left(\frac{\bm{\lambda}_{1}}{\alpha\varepsilon}\right)\in\mathbb{R}^{N}_{>0}$. Since $G$ is linear, its Legendre-Fenchel conjugate is an indicator function:
\begin{align*}
G^{*}(-\bm{\lambda}_{1}) = \begin{cases}
0 & \text{if}\quad\bm{\lambda}_{1} = -\bm{a},\\
+\infty & \text{otherwise}.
\end{cases}    
\end{align*}
Therefore, (\ref{LagrangeDualProblem}) yields
\begin{subequations}
\begin{align}
\bm{\lambda}_{0}^{\rm{opt}} &= \underset{\bm{\lambda}_{0}\in\mathbb{R}^{N}}{\arg\max}\bigg\{\langle\bm{\lambda}_{0},\bm{\zeta}\rangle - \alpha\varepsilon\langle\bm{y},\bm{\Gamma z}\rangle\bigg\},\label{lambda0linear}\\
\bm{\lambda}_{1}^{\rm{opt}} &= -\bm{a}.\label{lambda1linear}    
\end{align}
\label{lambda0lambda1ForLinear}
\end{subequations}
From (\ref{lambda1linear}), 
\begin{align}
\bm{z}^{\rm{opt}} = \exp\left(-\frac{1}{\alpha\varepsilon}\bm{a}\right).
\label{zoptlinear}    
\end{align}
Setting the gradient of the objective in (\ref{lambda0linear}) to zero, determines $\bm{\lambda}_{0}^{\rm{opt}}$, or equivalently $\bm{y}^{\rm{opt}}$ as
\begin{align}
\bm{y}^{\rm{opt}} = \bm{\zeta} \oslash \left(\bm{\Gamma}\bm{z}^{\rm{opt}}\right).
\label{yoptlinear}    
\end{align}
From (\ref{MuUpdateGeneral}), the proximal update is
\begin{align}
\prox^{W_{\varepsilon}}_{\frac{1}{\alpha}\Phi}\left(\bm{\zeta}\right) &= \bm{z}^{\rm{opt}} \odot \left(\bm{\Gamma}^{\top}\bm{y}^{\rm{opt}}\right)\nonumber\\
&\stackrel{(\ref{yoptlinear})}{=} \bm{z}^{\rm{opt}} \odot \left(\bm{\Gamma}^{\top}\left(\bm{\zeta} \oslash \left(\bm{\Gamma}\bm{z}^{\rm{opt}}\right)\right)\right).\label{ProxLinearFnOfMu}
\end{align}
Substituting (\ref{zoptlinear}) in (\ref{ProxLinearFnOfMu}), we arrive at (\ref{WassProxOfLinear}).
\end{proof}

\textbf{Example: $F_{i}(\bm{\mu}_{i}) = \langle \bm{U}\bm{\mu}_i, \bm{\mu}_i\rangle$.}

The case of advection-aggregation PDE shown in the third row of Table \ref{Table:PDEs} leads to $F_{i}(\bm{\mu}_{i}) = \langle \bm{U}\bm{\mu}_i, \bm{\mu}_i\rangle$, and thus 
$\Phi_{i}(\bm{\mu}_i) = \langle \bm{U}\bm{\mu}_i + \bm{\nu}_{i}^{k}, \bm{\mu}_i\rangle$ for given $i\in[n]$. In this discrete version, $\bm{U}\in\mathbb{R}^{N\times N}$ is the interaction potential. Following \citet[Sec. 4]{benamou2016augmented}, we approximate $\Phi_{i}(\bm{\mu}_i) \approx \widehat{\Phi}_{i}(\bm{\mu}_i,\bm{\mu}_i^k) := \langle \bm{U}\bm{\mu}_i^{k} + \bm{\nu}_{i}^{k}, \bm{\mu}_i\rangle$, resulting in semi-implicit variant of the proximal update (\ref{MuUpdateDiscreteRegularized}) given by 
$$\bm{\mu}_{i}^{k+1} =  \prox^{W_{\varepsilon}}_{\frac{1}{\alpha}\widehat{\Phi}_{i}(\bm{\mu}_i,\bm{\mu}_i^k)}\left(\bm{\zeta}^{k}\right) = \prox^{W_{\varepsilon}}_{\frac{1}{\alpha}\left(\langle \bm{U}\bm{\mu}_i^{k} + \bm{\nu}_{i}^{k}, \bm{\mu}_i\rangle\right)}\left(\bm{\zeta}^{k}\right), \quad k\in\mathbb{N}_{0}.$$
Convergence and consistency guarantees for such semi-implicit scheme are available, see e.g., \citep[Sec. 12.3]{laborde201712}. Such semi-imlicit schemes allow us to apply Theorem \ref{Thm:WassProxOfLinear} by setting $\Phi_{i}(\bm{\mu}_i) \equiv \widehat{\Phi}_{i}(\bm{\mu}_i,\bm{\mu}_i^{k})$, for given $i\in[n]$.

\textbf{Example: $F_{i}(\bm{\mu}_{i}) = \langle \frac{\beta^{-1}}{m-1}\bm{1},\bm{\mu}_i^{m}\rangle = \frac{\beta^{-1}}{m-1}\|\bm{\mu}_i\|_{m}^{m}$, $m>1$.}

The case of porous medium a.k.a. advection-nonlinear power law diffusion PDE shown in the fourth row of Table \ref{Table:PDEs} corresponds to $F_{i}(\bm{\mu}_{i}) = \langle \frac{\beta^{-1}}{m-1}\bm{1},\bm{\mu}_i^{m}\rangle$ (vector exponent $m$ is elementwise), and thus $\Phi_{i}(\bm{\mu}_i) = \langle \frac{\beta^{-1}}{m-1}\bm{\mu}_i^{m-1} + \bm{\nu}_{i}^{k}, \bm{\mu}_i\rangle$ for given $i\in[n]$. In this case, the proximal update becomes amenable via the following result.

\begin{theorem}\label{theorem:porous medium Eq}
Given $\bm{\nu}^{k}\in\mathbb{R}^{N}$, $\beta>0, m>1$, let $\Phi(\bm{\mu}):=\langle \frac{\beta^{-1}}{m-1}\bm{\mu}^{m-1} + \bm{\nu}^{k}, \bm{\mu}\rangle$ for $\bm{\mu}\in\Delta^{N-1}$.

Let $\bm{C}\in\mathbb{R}^{N\times N}$ be the squared Euclidean distance matrix, and for $\varepsilon>0$, let $\bm{\Gamma}:=\exp\left(-\bm{C}/2\varepsilon\right)$. For any $\bm{\zeta}\in\Delta^{N-1}$, $\alpha > 0$, let $(\bm{y}^{\mathrm{opt}}$, $\bm{z}^{\mathrm{opt}})\in\mathbb{R}^{N}_{>0}\times\mathbb{R}^{N}_{>0}$ be the solution of
\begin{subequations}
\begin{align}
\bm{y}\odot \left(\bm{\Gamma}^{\top} \bm{z}\right) &=  \bm{\zeta}, \label{33a}\\
     \bm{z} \odot\left(\bm{\Gamma}^{\top} \bm{y}\right) &= (\beta)^{\frac{1}{m-1}}\left(\frac{m-1}{m}\right)^{\frac{m}{m-1}}\left(-\frac{m}{m-1}\right)\nonumber\\
    &\quad\left(\bm{1}^{\top}\left(-\alpha\varepsilon \ln (\bm{z})-\bm{\nu}^{k}\right)^{m}\right)^{\frac{2-m}{m-1}}  \left(-\alpha\varepsilon \ln (\bm{z})-\bm{\nu}^{k}\right)^{m-1}.\label{33b}
\end{align}
\label{yzUpdate_PorousEq}
\end{subequations}
Then
\begin{align}
\prox^{W_{\varepsilon}}_{\frac{1}{\alpha}\Phi}\left(\bm{\zeta}\right)=\bm{z}^{\mathrm{opt}} \odot\left(\bm{\Gamma}^{\top} \bm{y}^{\mathrm{opt}}\right).
\label{ProxUpdatePorousMedium}    
\end{align}
\end{theorem}
\begin{proof}
The Legendre-Fenchel conjugate of $\Phi(\bm{\mu})$ is
\begin{align}
    \Phi^{*}(\bm{\lambda})=\sup _{\bm{\mu}}\left\{\bm{\lambda}^{\top} \bm{\mu}-(\bm{\nu}^{k})^{\top} \bm{\mu}-\frac{\beta^{-1}}{m-1} \bm{1}^\top\bm{\mu}_{i}^m \right\}. \label{35_m}
\end{align}
From (\ref{35_m}), direct computation gives
\begin{align}
    \Phi^{*}(\bm{\lambda})=\left(\beta\left(\frac{m-1}{m}\bm{1}^{\top}(\bm{\lambda}-\bm{\nu}^{k})\right)^{m}\right)^{\frac{1}{m-1}}. \label{37_m}
\end{align}
Let $\bm{y}:= \exp\left(\frac{\bm{\lambda}_{0}}{\alpha\varepsilon}\right)\in\mathbb{R}^{N}_{>0}, \bm{z} := \exp\left(\frac{\bm{\lambda}_{1}}{\alpha\varepsilon}\right)\in\mathbb{R}^{N}_{>0}$, and drop the subscripts $i$ in (\ref{LagrangeDualProblem}) for notational ease. Fixing $\bm{\lambda}_{1}$, and taking the gradient of the objective in (\ref{LagrangeDualProblem}) w.r.t. $\bm{\lambda}_{0}$ gives (\ref{33a}).

On the other hand, fixing $\bm{\lambda}_{0}$, and taking the gradient of the objective in (\ref{LagrangeDualProblem}) w.r.t. $\bm{\lambda}_{1}$ gives
\begin{equation}
    \nabla_{\bm{\lambda}_{1}} \Phi^{*}\left(-\bm{\lambda}_{1}\right)=\bm{z} \odot\left(\bm{\Gamma}^{\top} \bm{y}\right).\label{38_m}
\end{equation}
Using (\ref{37_m}) into the left hand side of (\ref{38_m}) results in (\ref{33b}). Finally, (\ref{MuUpdateGeneral}) yields the proximal update (\ref{ProxUpdatePorousMedium}). 
\end{proof}
\begin{remark}
That the unique pair $(\bm{y}^{\mathrm{opt}}$, $\bm{z}^{\mathrm{opt}})\in\mathbb{R}^{N}_{>0}\times\mathbb{R}^{N}_{>0}$ can be found from cone-preserving contractive fixed point recursion, follows from nonlinear Perron-Frobenius theory as in \citep[Sec. III-C]{caluya2019gradient}. In Sec. \ref{sec:Experiments}, we provide numerical results for advection-nonlinear power law diffusion with $m=2$.
\end{remark}

%%%%%%%%%%%%%%%%%%%%%%%%%%%%%%%%%%%%%%%%%%%%%%%%%%%%%%%%%%%%%%%%%%%%%%%%%%%%%%%%%%%%%%%%%%%%%%

\section{Proof of Theorem \ref{Thm:DualOfSinkhornBaryWithLinReg}}\label{AppProofThmDualOfSinkhornBaryWithLinReg}
We make use of the following Proposition from \citet[Sec. 4.1]{cuturi2016smoothed}, rephrased in our notation.
\begin{proposition}\label{prop:SinkhornRegularizedWassBaryWithExtraRegularization}\citep[Proposition 1]{cuturi2016smoothed}
Let $$W_{\varepsilon,\bm{\mu}_{i}}^{2}(\bm{\zeta}) := \underset{\bm{M}_{i}\in\Pi_{N}\left(\bm{\mu}_{i},\bm{\zeta}\right)}{\min}\bigg\langle\frac{1}{2}\bm{C} + \varepsilon\log\bm{M}_{i},\bm{M}_{i}\bigg\rangle, \quad \varepsilon > 0,$$
for given $\bm{\mu}_{i}\in\Delta^{N-1}$ for all $i\in[n]$, and for a given squared Euclidean distance matrix $\bm{C}\in\mathbb{R}^{N\times N}$. Let the superscript $^{*}$ denote the Legendre-Fenchel conjugate. Given weights $w_1,\hdots,w_{n}>0$, linear operator $\mathcal{A}$, and a convex real-valued function $J$, consider the variational problem
\begin{align}
\bm{\zeta}^{\rm{opt}} = \underset{\bm{\zeta}\in\Delta^{N-1}}{\arg\min}\displaystyle\sum_{i=1}^{n}w_{i}W_{\varepsilon,\bm{\mu}_{i}}^{2}(\bm{\zeta}) + J\left(\mathcal{A}\bm{\zeta}\right).
\label{SinkhornBaryWithGenReg}    
\end{align}
The dual problem of (\ref{SinkhornBaryWithGenReg}) is given by
\begin{align}
\left(\bm{u}_{1}^{\rm{opt}},\hdots,\bm{u}_{n}^{\rm{opt}},\bm{v}^{\rm{opt}}\right) = &\underset{\left(\bm{u}_{1},\hdots,\bm{u}_{n},\bm{v}\right)\in\mathbb{R}^{(n+1)N}}{\arg\min}\displaystyle\sum_{i=1}^{n}w_{i}\left(W_{\varepsilon,\bm{\mu}_{i}}^{2}\right)^{*}\left(\bm{u}_{i}\right) + J^{*}\left(\bm{v}\right)\nonumber\\
&\qquad{\rm{subject\;to}}\quad \mathcal{A}^{*}\bm{v} + \displaystyle\sum_{i=1}^{n}w_{i}\bm{u}_{i} = \bm{0},
\label{DualOfSinkhornBaryWithGenReg}    
\end{align}
and the primal-dual relation giving the minimizer in (\ref{SinkhornBaryWithGenReg}) is
\begin{align}
\bm{\zeta}^{\rm{opt}} = \nabla_{\bm{u}_{i}}\left(W_{\varepsilon,\bm{\mu}_{i}}^{2}\right)^{*}\left(\bm{u}_{i}^{\rm{opt}}\right)\:\in\:\Delta^{N-1}, \quad\text{for all}\; i\in[n].
\label{PrimalDualSinkhornBaryWithGenReg}    
\end{align}
\end{proposition}
We recast (\ref{ZetaUpdateDiscreteRegularized}) in the form (\ref{SinkhornBaryWithGenReg}) by setting the probability vectors $\bm{\mu}_{i}\equiv \bm{\mu}_{i}^{k+1}$, the weights $w_{1}=w_{2}=\hdots=w_{n}=1$, the operator $\mathcal{A}$ as identity, and the function $J(\cdot) \equiv \langle-\frac{2}{\alpha}\bm{\nu}^{k}_{\text{sum}},\cdot\rangle$. Since $J$ is linear, we have
\begin{align}
J^{*}(\bm{v}) = \begin{cases}
0 & \text{if}\quad\bm{v} = -\frac{2}{\alpha}\bm{\nu}^{k}_{\text{sum}},\\
+\infty & \text{otherwise}.
\end{cases}
\label{LegendreFenchelofJ}    
\end{align}
Also, $\mathcal{A}$ being the identity operator, we get $\mathcal{A}^{*}\bm{v}=\bm{v}$. Therefore, the dual problem (\ref{DualOfSinkhornBaryWithGenReg}) corresponding to (\ref{ZetaUpdateDiscreteRegularized}) becomes
\begin{align}
\left(\bm{u}_{1}^{\rm{opt}},\hdots,\bm{u}_{n}^{\rm{opt}}\right) = &\underset{\left(\bm{u}_{1},\hdots,\bm{u}_{n}\right)\in\mathbb{R}^{nN}}{\arg\min}\displaystyle\sum_{i=1}^{n}\left(W_{\varepsilon,\bm{\mu}_{i}^{k+1}}^{2}\right)^{*}\left(\bm{u}_{i}\right)\nonumber\\
&\quad{\rm{subject\;to}}\quad \displaystyle\sum_{i=1}^{n}\bm{u}_{i} = \frac{2}{\alpha}\bm{\nu}^{k}_{\text{sum}}.    
\label{DualOfSinkhornBaryWithLinReg}    
\end{align}
Consequently, the update (\ref{ZetaUpdateDiscreteRegularized}) can be performed by first solving the problem (\ref{DualOfSinkhornBaryWithLinReg}), and then evaluating the gradient of the Legendre-Fenchel conjugate (\ref{PrimalDualSinkhornBaryWithGenReg}) at the minimizer of (\ref{DualOfSinkhornBaryWithLinReg}).

It is known \citep[Theorem 2.4]{cuturi2016smoothed} that for given $\varepsilon>0$ and $\bm{\mu}\in\Delta^{N-1}$, the Legendre-Fenchel conjugate $\left(W_{\varepsilon,\bm{\mu}}^{2}\right)^{*}\left(\bm{u}\right)$ is $C^{\infty}(\mathbb{R}^{N})$ w.r.t. $\bm{u}\in\mathbb{R}^{N}$, and the gradient $\nabla_{\bm{u}}\left(W_{\varepsilon,\bm{\mu}}^{2}\right)^{*}\left(\bm{u}\right)$ is $1/\varepsilon$ Lipschitz. Furthermore, \citet[Theorem 2.4]{cuturi2016smoothed} gives the explicit formula
\begin{subequations}
\begin{align}
\left(W_{\varepsilon,\bm{\mu}}^{2}\right)^{*}\left(\bm{u}\right) &= -\varepsilon\big\langle\bm{\mu},\log\left(\bm{\mu}\oslash\left(\bm{\Gamma}\exp\left(\bm{u}/\varepsilon\right)\right)\right)\big\rangle, \label{ExplicitFormulaLegFenSinkWassSquared}\\
\nabla_{\bm{u}}\left(W_{\varepsilon,\bm{\mu}}^{2}\right)^{*}\left(\bm{u}\right) &= \exp\left(\bm{u}/\varepsilon\right) \odot \left(\bm{\Gamma}\left(\bm{\mu} \oslash \left(\bm{\Gamma}\exp\left(\bm{u}/\varepsilon\right)\right)\right)\right) \:\in\:\Delta^{N-1}. \label{ExplicitFormulaGradOfLegFenSinkWassSquared}
\end{align}
\label{ExplicitFormulaLegFenSinkWassSquaredAndGrad}
\end{subequations}
Using (\ref{ExplicitFormulaLegFenSinkWassSquared}) in the objective of (\ref{DualOfSinkhornBaryWithLinReg}) followed by algebraic simplification yields (\ref{DualOfSinkhornBaryWithLinRegSimplified}). Using (\ref{ExplicitFormulaGradOfLegFenSinkWassSquared}) in (\ref{PrimalDualSinkhornBaryWithGenReg}), we obtain (\ref{ZetaUpdateExplicit}). \qedsymbol

%%%%%%%%%%%%%%%%%%%%%%%%%%%%%%%%%%%%%%%%%%%%%%%%%%%%%%%%%%%%%%%%%%%%%%%%%%%%%%%%%%%%%%%%%%%%%%%%%%%%

\section{Proof of Lemma \ref{lemma:ProjOnConstrSet}}\label{AppProofOfLemma:ProjOnConstrSet}
We re-write the constraint set $\mathcal{C}$ as
\begin{align}
\mathcal{C} = \bigg\{\bm{z}\in\mathbb{R}^{nN} \mid \bm{A}\bm{z} = \frac{2}{\alpha}\bm{\nu}^{k}_{\text{sum}}\bigg\},
\label{RewriteC}    
\end{align}
where $\bm{z}=(\bm{z}_1,\hdots,\bm{z}_{n})$, $\bm{z}_{i}\in\mathbb{R}^{N}$ for all $i\in[n]$, $\bm{A} := \left[\bm{I}_{N}, \hdots, \bm{I}_{N}\right]\in\mathbb{R}^{N\times nN}$, and $\bm{I}_{N}$ is the $N\times N$ identity matrix.

Following \citet[Sec. 4]{bauschke2004reflection}, we have
\begin{align}
{\rm{proj}}_{\mathcal{C}}\left(\bm{v}\right) = \bm{v} - \bm{A}^{\dagger}\left(\bm{Av} - \frac{2}{\alpha}\bm{\nu}^{k}_{\text{sum}}\right)
\label{ProjPseudoinv}    
\end{align}
where the superscript $^{\dagger}$ denotes the Moore-Penrose pseudoinverse. For our $\bm{A}\in\mathbb{R}^{N\times nN}$, (\ref{ProjPseudoinv}) simplifies to
\begin{align*}
{\rm{proj}}_{\mathcal{C}}\left(\bm{v}\right) &= \bm{v} - \bm{A}^{\top}\left(\bm{A}\bm{A}^{\top}\right)^{-1}\left(\bm{Av} - \frac{2}{\alpha}\bm{\nu}^{k}_{\text{sum}}\right)\\
&= \bm{v} - \begin{bmatrix}
\frac{1}{n}\bm{I}_{N}\\
\vdots\\
\frac{1}{n}\bm{I}_{N}
\end{bmatrix}
\left(\sum_{i=1}^{n}\bm{v}_{i} - \frac{2}{\alpha}\bm{\nu}^{k}_{\text{sum}}\right)\\
&= \bm{v} - \begin{bmatrix}
\overline{\bm{v}} - \frac{2}{n\alpha}\bm{\nu}^{k}_{\text{sum}}\\
\vdots\\
\overline{\bm{v}} - \frac{2}{n\alpha}\bm{\nu}^{k}_{\text{sum}}
\end{bmatrix},
\end{align*}
thus completing the proof. \qedsymbol

%%%%%%%%%%%%%%%%%%%%%%%%%%%%%%%%%%%%%%%%%%%%%%%%%%%%%%%%%%%%%%%%%%%%%%%%%%%%%%%%%%%%%%%%%%%%%%%%%%%%

\begin{figure}[tpb]
\centering
\includegraphics[width=0.9\linewidth]{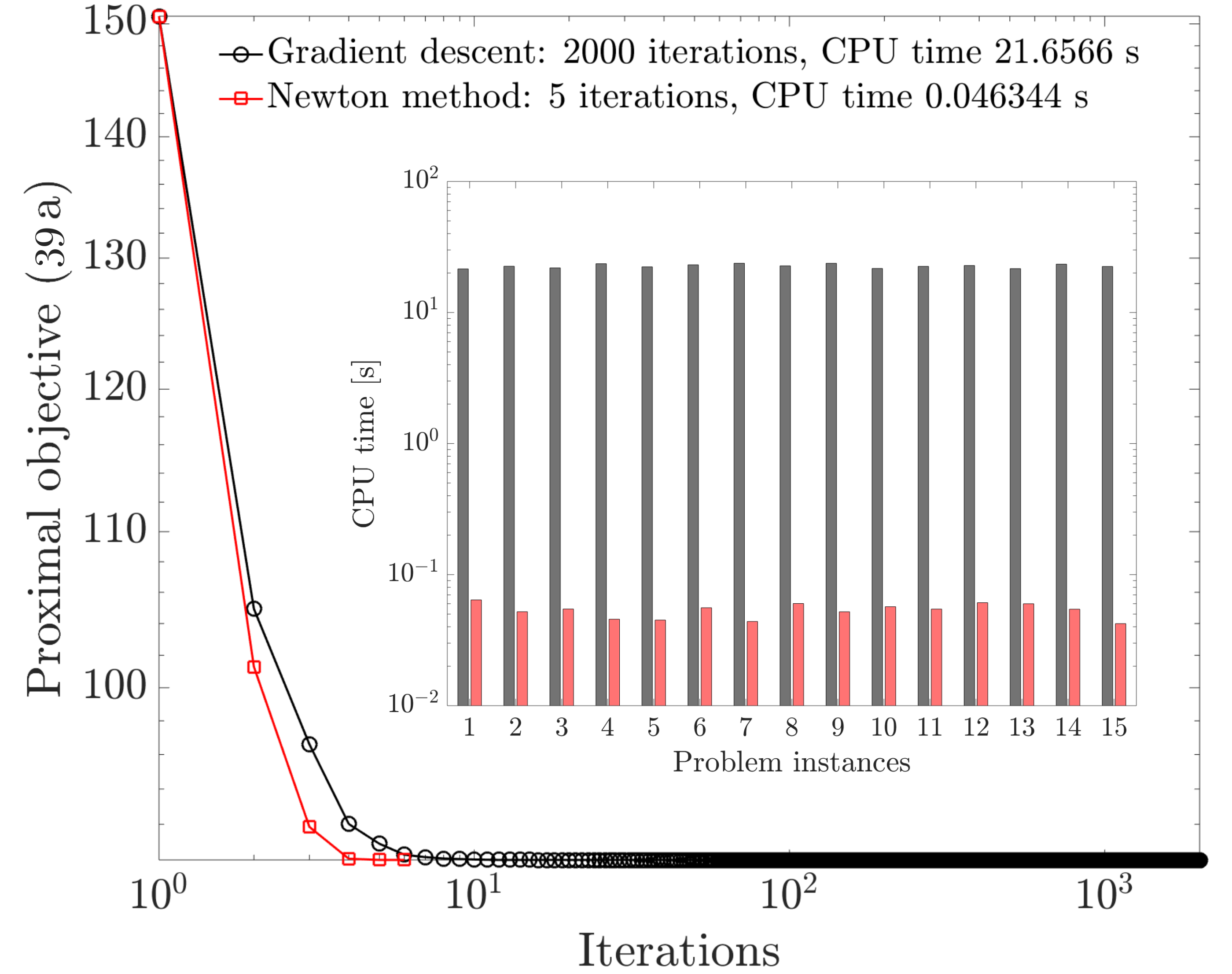}
\caption{\emph{Main plot:} A typical instance of the proximal optimization problem (\ref{distributedxUpdate}) with $N=441$, $\varepsilon=\tau=0.1$, random initial guess, randomly generated input data (i.e., proximal argument in $\mathbb{R}^{N}$), random parameter $\bm{\mu}\in\Delta^{N-1}$, and the Euclidean distance matrix $\bm{C}\in\mathbb{R}^{N\times N}$ for uniform grid over $[-1,1]^{2}$ with spatial discretization length $0.1$ in both directions. The problem instance was solved via the gradient descent and the Newton's method with the same numerical tolerance $10^{-4}$. The stopping criterion for the gradient descent was the norm of the gradient being less than or equal to the numerical tolerance. For the Newton's method, we used the standard stopping criterion \citep[p. 487]{boyd2004convex}: one half of the squared Newton decrement being less than or equal to the numerical tolerance. Both algorithms used variable step size via backtracking line search (see Appendix E) with parameters $\alpha_{0}=0.3, \beta_{0}=0.7$. For gradient descent, the proximal objective after the last iteration was equal to $90.018062265357955$; the same for Newton's method was equal to $90.018072956312977$. \emph{Inset plot:} The CPU time comparisons for 15 instances of (\ref{distributedxUpdate}) with randomly chosen initial guess, proximal argument and parameter $\bm{\mu}\in\Delta^{N-1}$ while keeping all other parameters fixed and same as before across all problem instances. The longer (resp. shorter) bars are for the gradient descent (resp. Newton's method). For all 15 problem instances, the gradient descent took 2000--2002 iterations while the Newton's method required 5--6 iterations. So the convergence trend shown in the main plot is typical. As explained in Sec. \ref{subsec:ZetaUpdate}, solving the proximal problem (\ref{distributedxUpdate}) arises as a sub-problem for the inner layer ADMM (\ref{distributedADMM}).
}
\label{fig:uProxUnitTest}
\end{figure}

\section{Gradient and Hessian of (\ref{deffi}), and Solving (\ref{distributedxUpdate})}\label{AppGradHessInnerProxNewton}

\textbf{Gradient of (\ref{deffi}).} To reduce clutter, let us drop the indices $i\in[n]$ and $k\in\mathbb{N}_{0}$ for the time being, and focus on computing the gradient and Hessian of
$$f(\bm{u}):=\big\langle\bm{\mu},\log\left(\bm{\Gamma}\exp\left(\bm{u}/\varepsilon\right)\right)\big\rangle$$
w.r.t. $\bm{u}\in\mathbb{R}^{N}$ for given $\bm{\mu}=(\mu_1, \hdots,\mu_{N})\in\Delta^{N-1}$. Notice that $f$ is twice continuously differentiable but is not everywhere strictly convex; e.g., $f$ is affine along any line $\bm{u}=u_{0}\bm{1}$ where $u_{0}$ is some nonzero real and $\bm{1}$ denotes the $N\times 1$ vector of ones. 

Denote the $j$th row of the matrix $\bm{\Gamma}\in\mathbb{R}^{N\times N}$ as $\bm{\gamma}_{j}$, and write
\begin{align}
f(\bm{u}) = \displaystyle\sum_{j=1}^{N}\mu_{j}\log\langle\bm{\gamma}_{j},\exp\left(\bm{u}/\varepsilon\right)\rangle.  \label{fAsSum}    
\end{align}
Using the chain rule in (\ref{fAsSum}), we have
\begin{align}
\nabla_{\bm{u}}f = \dfrac{1}{\varepsilon}\displaystyle\sum_{j=1}^{N} \mu_{j} \dfrac{\bm{\gamma}_{j}\odot \exp\left(\bm{u}/\varepsilon\right)}{\langle \bm{\gamma}_{j}, \exp\left(\bm{u}/\varepsilon\right)\rangle} = \dfrac{1}{\varepsilon} \left(\bm{\Gamma}^{\top}\bm{\mu}\right)\odot\exp\left(\bm{u}/\varepsilon\right)\oslash \left(\bm{\Gamma}\exp\left(\bm{u}/\varepsilon\right)\right).
\label{gradf}    
\end{align} 
Bringing back the indices $i\in[n]$ and $k\in\mathbb{N}_{0}$ as in (\ref{deffi}), and letting $\bm{e}_{i} := \exp\left(\bm{u}_{i}/\varepsilon\right)$, the expression (\ref{gradf}) gives 
\begin{align}
\nabla_{\bm{u}_{i}}f_{i} 
%= \dfrac{1}{\varepsilon} \left(\bm{\Gamma}^{\top}\bm{\mu}_{i}^{k+1}\right)\odot\exp\left(\bm{u}_{i}/\varepsilon\right)\oslash \left(\bm{\Gamma}\exp\left(\bm{u}_{i}/\varepsilon\right)\right) 
= \dfrac{1}{\varepsilon} \left(\bm{\Gamma}^{\top}\bm{\mu}_{i}^{k+1}\right)\odot\bm{e}_{i}\oslash \left(\bm{\Gamma}\bm{e}_{i}\right).
\label{gradfi}     
\end{align}

\textbf{Hessian of (\ref{deffi}).} Proceeding from (\ref{gradfi}), we get the Hessian
\begin{align}
\nabla_{\bm{u}_{i}}^{2} f_{i} = \dfrac{1}{\varepsilon^2}\left[\diag\left(\left(\bm{\Gamma}^{\top}\bm{\mu}_{i}^{k+1}\right)\odot\bm{e}_{i}\oslash \left(\bm{\Gamma}\bm{e}_{i}\right)\right) - \diag\left(\left(\bm{\Gamma}^{\top}\bm{\mu}_{i}^{k+1}\right)\oslash\left(\bm{\Gamma}\bm{e}_{i}\right)^{2}\right)\bm{\Gamma}\odot\left(\bm{e}_{i}\bm{e}_{i}^{\top}\right)\right]
\label{hessf}    
\end{align}
where $\left(\bm{\Gamma}\bm{e}_{i}\right)^{2}$ denotes the elementwise square of the vector $\bm{\Gamma}\bm{e}_{i}$. 

Because the matrix $\bm{C}$ is symmetric, $\bm{\Gamma}$ is symmetric too, and we can drop the transpose from (\ref{hessf}). Furthermore, since $\bm{\Gamma}\odot\left(\bm{e}_{i}\bm{e}_{i}^{\top}\right) = \diag\left(\bm{e}_{i}\right)\bm{\Gamma}\diag\left(\bm{e}_{i}\right)$, we can rewrite (\ref{hessf}) as
\begin{align}
\nabla_{\bm{u}_{i}}^{2} f_{i} = \dfrac{1}{\varepsilon^2}\diag\left(\left(\bm{\Gamma}\bm{\mu}_{i}^{k+1}\right)\odot\bm{e}_{i}\oslash \left(\bm{\Gamma}\bm{e}_{i}\right)\right)\left[\bm{I}_{N} - \diag\left(\bm{1}\oslash\left(\bm{\Gamma}\bm{e}_{i}\right)\right)\bm{\Gamma}\diag\left(\bm{e}_{i}\right)\right].
\label{hessfSimplified}     
\end{align}
Notice that the matrix $\diag\left(\bm{1}\oslash\left(\bm{\Gamma}\bm{e}_{i}\right)\right)\bm{\Gamma}\diag\left(\bm{e}_{i}\right)$ is elementwise positive and row stochastic, and therefore, by linear Perron-Frobenius theorem, the matrix in square braces in (\ref{hessfSimplified}) has zero as a simple eigenvalue. Thus, the Hessian (\ref{hessfSimplified}) is positive semidefinite. The Hessian of the proximal objective in (\ref{distributedxUpdate}) is $\bm{I}_{N} + \frac{1}{\tau}\nabla_{\bm{u}_{i}}^{2} f_{i}$ where $\tau>0$, and is, therefore, strictly positive definite.

\textbf{Solving (\ref{distributedxUpdate}) via Newton's Method.} The structured Hessian of the proximal objective in (\ref{distributedxUpdate}) mentioned above, makes the per iteration complexity for solving (\ref{distributedxUpdate}) via Newton's method to be $\mathcal{O}(N^{2})$ flops instead of $\mathcal{O}(N^{3})$ flops--the latter would be the case for Cholesky factorization-based solution of the associated linear system. Fig. \ref{fig:uProxUnitTest} shows that the typical convergence for the Newton's method occurs in approx. 5 iterations, much faster than gradient descent (see Fig. \ref{fig:uProxUnitTest} caption for details).

\textbf{Backtracking line search.} For unconstrained minimization of an objective $f_{0}$ via recursive algorithms such as gradient descent or Newton's method, at each iteration, we compute the corresponding descent direction $\Delta x$ at $x\in\text{domain}(f_{0})$. Then we apply the recursive update rule $x \leftarrow x + t\Delta x$ where $t$ is a variable step size at that iteration. A standard method of computing the step size is the backtracking line search \citep[p. 464]{boyd2004convex}. Given parameters $\alpha_{0}\in(0,0.5)$, $\beta_{0}\in(0,1)$, the backtracking line search  starts with an initial step size $t=1$, and while $f_{0}(x + t\Delta x) > f_{0}(x) + \alpha_{0}t\langle \nabla f, \Delta x\rangle$, sets $t \leftarrow \beta_{0}t$. The resulting value of $t$ is used as the step size at that iteration. Both the gradient descent and Newton's method implementations as reported in Fig. \ref{fig:uProxUnitTest}, use backtracking line search with parameter values detailed in Fig. \ref{fig:uProxUnitTest} caption.

%%%%%%%%%%%%%%%%%%%%%%%%%%%%%%%%%%%%%%%%%%%%%%%%%%%%%%%%%%%%%%%%%%%%%%%%%%%%%%%%%%%%%%%%%%%%%%%%%%%%

\section{Summary of the Overall Algorithm}\label{AppAlgoDetails}
\begin{figure}[t]
\centering
\includegraphics[width=1\linewidth]{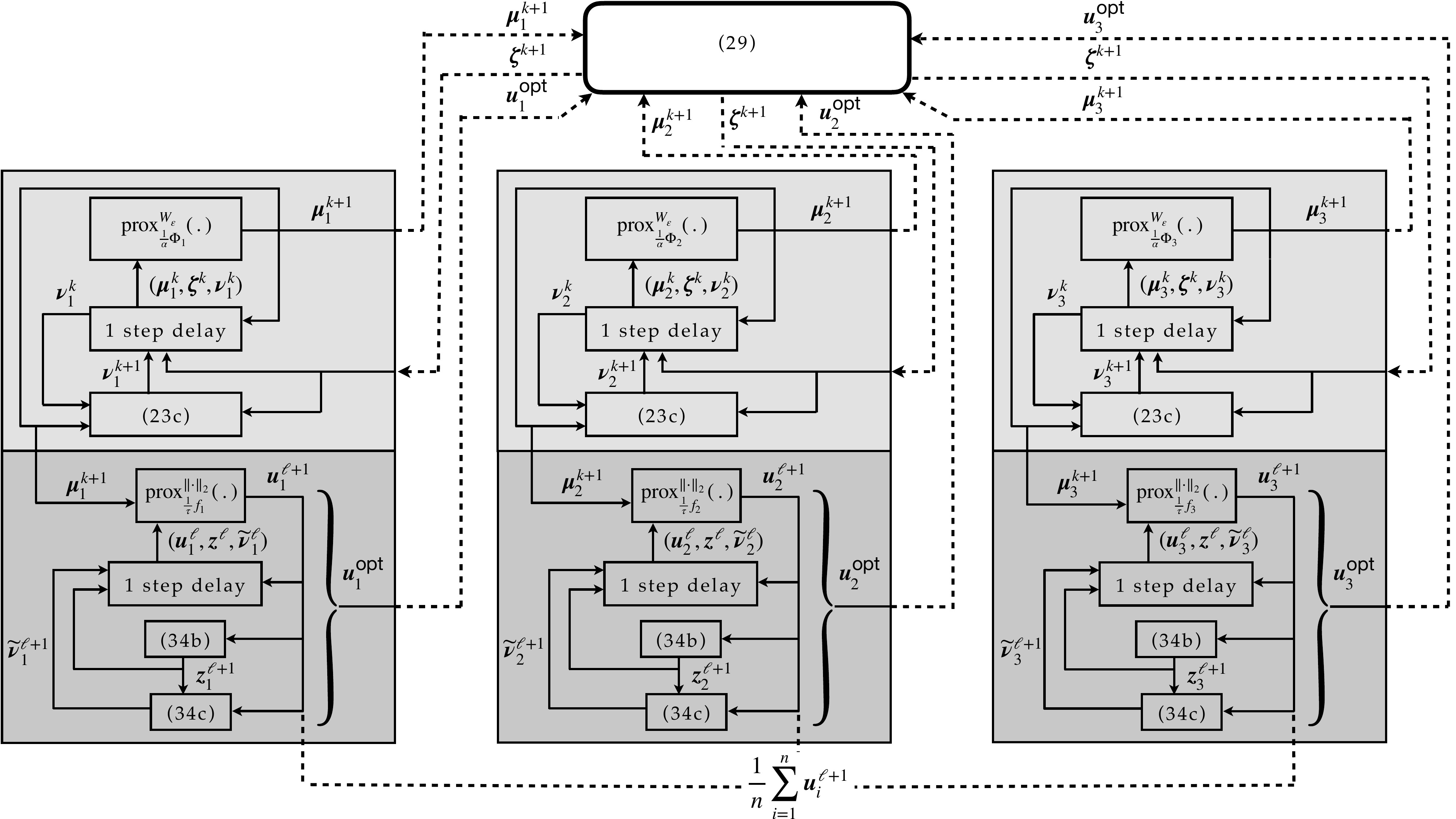}
\caption{Detailed schematic of the proposed computational framework. As in Fig. \ref{fig:GeneralBlockDiagram}, the lighter and darker shades correspond to the ``upstairs" and ``downstairs" computation in the distributed processors, respectively, which in turn, correspond to the outer and inner layer ADMM, respectively.}
\label{fig:BlockDiagram}
\end{figure}
In Fig. \ref{fig:BlockDiagram}, we detail the computational framework proposed in Sec. \ref{subsec:MuUpdate} and Sec. \ref{subsec:ZetaUpdate}. We view Fig. \ref{fig:BlockDiagram} as an expanded version of the high level schematic given in Fig. \ref{fig:GeneralBlockDiagram}, i.e., Fig. \ref{fig:BlockDiagram} depicts the low level details omitted in Fig. \ref{fig:GeneralBlockDiagram}. 

Note that in the inner layer ADMM, to update $\bm{z}_i$ in (\ref{distributedzUpdate}), we need $\frac{1}{n}\displaystyle\sum_{i=1}^{n}\bm{u}_{i}^{\ell + 1}$ from the other distributed processors and the pipeline below the diagram in Fig. \ref{fig:BlockDiagram} gathers these data from all distributed processors and feeds to (\ref{distributedzUpdate}).

In summary, the computational steps are as follows.

\textbf{Step 0.} Split the objective $F$ as (\ref{AdditiveOptimizationnplusoneObj}) and relabel the argument of the functionals $F_{i}$ as $\mu_{i}$ $\forall i\in[n]$.

\textbf{Step 1.} Initialize $\bm{\mu}_i^0$, $\bm{\zeta}^0$ everywhere positive, and $\bm{\nu}_i^0$ arbitrary $\forall i\in[n]$.
    
\textbf{Step 2.} Perform distributed ``upstairs'' updates (\ref{MuUpdateDiscreteRegularized}) for $\bm{\mu}_{i}^{k+1}$ via Prop. \ref{prop:ProximalUpdateGeneral} (outer layer ADMM). 

\textbf{Step 3.} Perform distributed ``downstairs'' updates $\bm{u}_i^{\text{opt}}$ from the inner layer ADMM (\ref{distributedADMM}).
    
\textbf{Step 4.} Perform centralized update for $\bm{\zeta}^{k+1}$ using (\ref{ZetaUpdateExplicit}) (outer layer ADMM).

\textbf{Step 5.} Perform distributed ``upstairs'' updates for $\bm{\nu}_i^{k+1}$ using (\ref{NuUpdateDiscreteRegularized}) (outer layer ADMM).

The above steps are repeated until a user-specified maximum number of outer layer iterations are done, or the maximum of the pairwise Wasserstein distances fall below a prescribed tolerance.

%%%%%%%%%%%%%%%%%%%%%%%%%%%%%%%%%%%%%%%%%%%%%%%%%%%%%%%%%%%%%%%%%%%%%%%%%%%%%%

\section{Convergence Guarantee for the Inner Layer ADMM}\label{AppConvergenceInnerLayerADMM}
In the following, we present sufficient conditions that guarantee the convergence
of inner layer ADMM (\ref{distributedADMM}). To this end, we need two preparatory lemmas. 
\begin{lemma}\label{LemmaHessianToLipschitzGrad}
\citep[p. 58, Thm. 2.1.6]{nesterov2003introductory}
A $C^2$ convex function $f$ with ${\rm{domain}}(f)=\mathbb{R}^{N}$, has Lipschitz continuous gradient w.r.t. $\|\cdot\|_2$ with Lipschitz
constant $L>0$ if $\bm{v}^{\top}\nabla^{2}f(\bm{u})\bm{v} \leq L\|\bm{v}\|_{2}^2$ for all $\bm{u},\bm{v}\in\mathbb{R}^{N}$.    
\end{lemma}

\begin{lemma}\label{Lemma_Lipschitz_continuous_gradient}
The $C^2$ convex function $f$ given by (\ref{fAsSum}) with ${\rm{domain}}(f)=\mathbb{R}^{N}$, has Lipschitz continuous gradient w.r.t. $\|\cdot\|_2$ with Lipschitz constant $L=\dfrac{1}{\varepsilon^2}\big\|\bm{\Gamma}\bm{\mu}^{k+1}\big\|_{\infty}$.
\end{lemma}
\begin{proof}
Let $\bm{e}:=\exp(\bm{u}/\varepsilon)$. From (\ref{hessf}), for all $\bm{u},\bm{v}\in\mathbb{R}^{N}$, we have
\begin{align}
\bm{v}^{\top}\nabla^{2}f(\bm{u})\bm{v} &= \dfrac{1}{\varepsilon^2}\bm{v}^{\top}\left[\diag\left(\left(\bm{\Gamma}^{\top}\bm{\mu}^{k+1}\right)\odot\bm{e}\oslash \left(\bm{\Gamma}\bm{e}\right)\right) - \diag\left(\left(\bm{\Gamma}^{\top}\bm{\mu}^{k+1}\right)\oslash\left(\bm{\Gamma}\bm{e}\right)^{2}\right)\bm{\Gamma}\odot\left(\bm{e}\bm{e}^{\top}\right)\right]\bm{v}\nonumber\\
&\leq \dfrac{1}{\varepsilon^2}\bm{v}^{\top}\diag\left(\left(\bm{\Gamma}^{\top}\bm{\mu}^{k+1}\right)\odot\bm{e}\oslash \left(\bm{\Gamma}\bm{e}\right)\right)\bm{v},
\label{LipInequalityFirst}
\end{align}
since the quadratic term followed by the minus sign is nonnegative. Hence (\ref{LipInequalityFirst}) yields
\begin{align} \bm{v}^{\top}\nabla^{2}f(\bm{u})\bm{v} \leq \dfrac{1}{\varepsilon^2} \big\|\left(\bm{\Gamma}^{\top}\bm{\mu}^{k+1}\right)\odot\bm{e}\oslash \left(\bm{\Gamma}\bm{e}\right)\big\|_{\infty} \|\bm{v}\|_{2}^{2} \leq \dfrac{1}{\varepsilon^2}\big\|\bm{\Gamma}^{\top}\bm{\mu}^{k+1}\big\|_{\infty} \big\|\bm{e}\oslash \left(\bm{\Gamma}\bm{e}\right)\big\|_{\infty}\|\bm{v}\|_{2}^{2}.    
\label{LipInequalitySecond}
\end{align}
Recall that $\bm{\Gamma}:=\exp\left(-\bm{C}/2\varepsilon\right)$ where $\bm{C}\in\mathbb{R}^{N\times N}$ is a squared Euclidean distance matrix. So the entries of the symmetric matrix $\bm{C}$ are in $[0,\infty)$ and thus, the entries of the symmetric matrix $\bm{\Gamma}$ are in $(0,1]$ with all diagonal entries being equal to $1$. Therefore, $\big\|\bm{e}\oslash \left(\bm{\Gamma}\bm{e}\right)\big\|_{\infty}\leq1$, and (\ref{LipInequalitySecond}) gives
$$\bm{v}^{\top}\nabla^{2}f(\bm{u})\bm{v} \leq \dfrac{1}{\varepsilon^2}\big\|\bm{\Gamma}\bm{\mu}^{k+1}\big\|_{\infty} \|\bm{v}\|_{2}^{2} \qquad \forall\bm{u},\bm{v}\in\mathbb{R}^{N},$$
where we dropped the transpose due to the symmetry of $\bm{\Gamma}$. Invoking Lemma \ref{LemmaHessianToLipschitzGrad}, we conclude the proof.
\end{proof}

\begin{theorem}
Let $\bm{C}$, $\bm{\Gamma}$, and $\bm{\mu}_{i}^{k+1}\in\Delta^{N-1}$ for all $i\in[n]$, $k\in\mathbb{N}_{0}$, as in Theorem \ref{Thm:DualOfSinkhornBaryWithLinReg}. 
If
\begin{align}
    \tau > \frac{\sqrt{2}}{\varepsilon^2}\big\|\bm{\Gamma}\bm{\mu}_{n}^{k+1}\big\|_{\infty},
\end{align}
then the sequence $\left(\bm{u}_{1}^{\ell},\hdots,\bm{u}_{n}^{\ell}\right)$ generated by the inner layer ADMM given in (\ref{distributedADMM}) converge to the optimal solutions of problem (\ref{DualOfSinkhornBaryWithLinRegSimplified}), i.e., $\left(\bm{u}_{1}^{\ell},\hdots,\bm{u}_{n}^{\ell}\right) \xrightarrow[]{\ell \nearrow \infty}\left(\bm{u}_{1}^{\rm{opt}},\hdots,\bm{u}_{n}^{\rm{opt}}\right)$.
\end{theorem}

\begin{proof}
We start our proof by presenting a sufficient condition for convergence of certain generic multi-block ADMM, and show that the inner layer ADMM given in (\ref{distributedADMM}) satisfies these conditions.

To this end, we start with the following convex minimization problem:
\begin{align}
    &\underset{\left(\bm{u}_{1},\hdots,\bm{u}_{n}\right)\in\mathbb{R}^{nN}}{\min} g_1\left(\bm{u}_1\right)+g_2\left(\bm{u}_2\right)+\cdots+g_n\left(\bm{u}_n\right) \nonumber\\
    &\!\!\text{subject to}\, \bm{A}_1 \bm{u}_1+\bm{A}_2 \bm{u}_2+\cdots+\bm{A}_n \bm{u}_n=\bm{b},\nonumber\\
    &\qquad\qquad\bm{u}_i \in \mathcal{U}_i \quad \text{for all}\; i\in[n],
\label{generalADMM}
\end{align}
where $\bm{A}_i \in \mathbb{R}^{N \times N}, \bm{b} \in \mathbb{R}^N$, the sets $\mathcal{U}_i \subseteq \mathbb{R}^{N}$ are closed convex, and $g_i: \mathcal{U}_i \rightarrow \mathbb{R}$ are closed convex functions for all $i\in[n]$. The augmented Lagrangian for (\ref{generalADMM}) is
\begin{align}
\mathcal{L}_\tau\left(\bm{u}_1, \ldots, \bm{u}_n ; \bm{\lambda}\right):= g_1\left(\bm{u}_1\right) + \hdots + g_n\left(\bm{u}_n\right)+\left\langle\bm{\lambda}, \sum_{i=1}^n \bm{A}_i \bm{u}_i-\bm{b}\right\rangle+\frac{\tau}{2}\left\|\sum_{i=1}^n \bm{A}_i \bm{u}_i-\bm{b}\right\|^2,
\end{align}
where $\bm{\lambda}\in\mathbb{R}^{N}$ is the Lagrange multiplier, and $\tau>0$ is a penalty parameter. For (\ref{generalADMM}), consider the multi-block ADMM recursions:
\begin{align}
    \left\{\begin{aligned}
\bm{u}_1^{\ell+1} &=\operatorname{argmin}_{\bm{u}_1 \in \mathcal{U}_1} \mathcal{L}_\tau\left(\bm{u}_1, \bm{u}_2^\ell, \ldots, \bm{u}_n^\ell ; \bm{\lambda}^\ell\right), \\
\bm{u}_2^{\ell+1} &=\operatorname{argmin}_{\bm{u}_2 \in \mathcal{U}_2} \mathcal{L}_\tau\left(\bm{u}_1^{\ell+1}, \bm{u}_2, \bm{u}_3^\ell, \ldots, \bm{u}_n^\ell ; \bm{\lambda}^\ell\right), \\
&\;\;\vdots \\
\bm{u}_n^{\ell+1} &=\operatorname{argmin}_{\bm{u}_n \in \mathcal{U}_n} \mathcal{L}_\tau\left(\bm{u}_1^{\ell+1}, \bm{u}_2^{\ell+1}, \ldots, \bm{u}_{n-1}^{\ell+1}, \bm{u}_n ; \bm{\lambda}^\ell\right), \\
\bm{\lambda}^{\ell+1} &=\bm{\lambda}^\ell+\tau\left(\sum_{i=1}^n \bm{A}_i \bm{u}_i^{\ell+1}-\bm{b}\right).
\end{aligned}\right.
\label{multiADMM}
\end{align}

For (\ref{generalADMM}), when the following conditions \citep[Corollary 3.5]{hong2016convergence}: 
\begin{itemize}
    \item[c1.] the matrices $\bm{A}_i$ have full column rank for all $i \in [n-1]$, and $\bm{A}_n=\bm{I}_{N}$,
    \item[c2.] the sets $\mathcal{U}_i$ are closed convex for all $i \in [n]$,
    \item[c3.] the mappings $g_i$ are lower bounded for all $i\in [n]$,
    \item[c4.] $\tau > \sqrt{2}L$ where $L$ is Lipschitz constant (w.r.t. $\|\cdot\|_2$) for $\nabla_{\bm{u}_n} g_{n}$,
\end{itemize}
are satisfied, then as the recursion index $\ell\nearrow\infty$, the solution of the multi-block ADMM (\ref{multiADMM}) converges to the optimal solutions of (\ref{generalADMM}).
Notice that the recursions (\ref{distributedADMM}) associated with the problem (\ref{DualOfSinkhornBaryWithLinRegSimplified}), are indeed an instance of the generic recursions (\ref{multiADMM}) associated with (\ref{generalADMM}). In particular, $$g_i(\bm{u}_i) \equiv \big\langle\bm{\mu}_{i}^{k+1},\log\left(\bm{\Gamma}\exp\left(\bm{u}_{i}/\varepsilon\right)\right)\big\rangle,$$ where the probability vectors $\bm{\mu}_{i}^{k+1}\in\Delta^{N-1}$ for all $i\in[n]$, $k\in\mathbb{N}_{0}$. Thus motivated, we check the conditions c1-c4.

Specifically, condition c1 for (\ref{DualOfSinkhornBaryWithLinRegSimplified}) is satisfied because $\bm{A}_i=\bm{I}_{N}$ for all $i \in [n]$. Condition c2 for (\ref{DualOfSinkhornBaryWithLinRegSimplified}) holds since $\mathcal{U}_{i}=\mathbb{R}^{N}$ for all $i \in [n]$, which are closed as well as affine (hence convex) sets.

For condition c3, we need to verify that the mappings $\bm{u}_i\mapsto g_i(\bm{u}_i)= \big\langle\bm{\mu}_{i}^{k+1},\log\left(\bm{\Gamma}\exp\left(\bm{u}_{i}/\varepsilon\right)\right)\big\rangle$
% \mu_{j}\log\langle\bm{\gamma}_{j},\exp\left(\bm{u}_i/\varepsilon\right)\rangle$, which are the summands in \eqref{fAsSum}, 
are uniformly lower bounded.
The lower bound for $\bm{u}_i\mapsto g_i(\bm{u}_i)$ can be found as the following unconstrained minimum 
\begin{align}
g_i^{\rm{opt}} := \min_{\bm{u}_i\in \mathbb{R}^N} \big\langle\bm{\mu}_{i}^{k+1},\log\left(\bm{\Gamma}\exp\left(\bm{u}_{i}/\varepsilon\right)\right)\big\rangle,
\label{unconstrainedMinimizationForLowerBound}    
\end{align}
which is the minimum of a convex combination of log-sum-exp composed with an affine map.

By choosing matrix $\bm{A}$ as an invertible matrix and introducing two new variables, $\Tilde{\bm{u}}\in \mathbb{R}^{N}$ and $\bm{y}\in \mathbb{R}^{N}$, we reformulate problem (\ref{unconstrainedMinimizationForLowerBound}) as:
\begin{align}
&\min_{\bm{y}\in \mathbb{R}^N} f_0(\bm{y}):=\mu_{j}\log\sum_{i=1}^{N} \exp(\bm{y}_{i}) \nonumber\\
&\!\!\text{subject to}~~~ \bm{u}/\varepsilon=\bm{A}\Tilde{\bm{u}},\nonumber\\
&\quad \quad \quad ~~~~~\bm{A}\Tilde{\bm{u}}+\log\bm{\gamma}_{j}=\bm{y},
\label{constrainedProblem}    
\end{align}
where $\bm{y}_{i}$ is $i$th element of vector $\bm{y}$. The Lagrangian of the reformulated problem is
\begin{align}
L(\bm{u},\Tilde{\bm{u}},\bm{y},\bm{\kappa},\bm{\eta})=\mu_{j}\log\sum_{i=1}^{N} \exp(\bm{y}_{i})+\bm{\eta}^{\top}\left( \bm{A}\Tilde{\bm{u}}+\log\bm{\gamma}_{j}-\bm{y} \right)+\bm{\kappa}^{\top}\left( \bm{u}/\varepsilon-\bm{A}\Tilde{\bm{u}}\right),
\label{Lagr}
\end{align}
where $\bm{\kappa}$ and $\bm{\eta}$ are the Lagrangian multipliers, and the corresponding Lagrange dual function is defined as 
\begin{align}
h(\bm{\eta},\bm{\kappa})=\inf_{\bm{u},\Tilde{\bm{u}}, \bm{y}}\bigg\{\mu_{j}\log\sum_{i=1}^{N} \exp(\bm{y}_{i})+\bm{\eta}^{\top}\left( \bm{A}\Tilde{\bm{u}}+\log\bm{\gamma}_{j}-\bm{y} \right)+\bm{\kappa}^{\top}\left( \bm{u}/\varepsilon-\bm{A}\Tilde{\bm{u}}\right)\bigg\}.
\label{dualFunction}
\end{align}
% To find the dual function, we minimize $L$ over $\bm{u}$, $\Tilde{\bm{u}}$, and $\bm{y}$. 
Minimizing over $\bm{u}$ results in $h(\bm{\eta},\bm{\kappa})=-\infty$ unless $\bm{\kappa=0}$.  Substituting $\bm{\kappa=0}$ in (\ref{dualFunction}), we get
\begin{align}
  h(\bm{\eta})=\inf_{\Tilde{\bm{u}}, \bm{y}} \left(\mu_{j}\log\sum_{i=1}^{N} \exp(\bm{y}_{i})+\bm{\eta}^{\top}\left( \bm{A}\Tilde{\bm{u}}+\log\bm{\gamma}_{j}-\bm{y} \right)\right).
  \label{dualFunc2}
\end{align}
Minimizing over $\tilde{\bm{u}}$ results in $h(\bm{\eta})=-\infty$ unless $\bm{A}^{\top}\bm{\eta=0}$.
So, 
\begin{align}     h(\bm{\eta})=\bm{\eta}^{\top}\log\bm{\gamma}_{j} +\inf_{\bm{y}} \left( \mu_{j}\log\sum_{i=1}^{N} \exp(\bm{y}_{i})- \bm{\eta}^{\top}\bm{y} \right)=\bm{\eta}^{\top}\log\bm{\gamma}_{j}+f_{0}^{*}(\bm{\eta}),
\end{align}
where the conjugate of $f_{0}$ is 
\begin{align}
    f_{0}^{*}=\left\{\begin{aligned} &\bigg\langle\bm{\eta},\log  \frac{\bm{\eta}}{\mu_j} \bigg\rangle \quad \quad \bm{\eta}\succeq 0, ~ \bm{1}^{\top} \bm{\eta}=1,\\
    & ~\infty \quad \quad  \quad \quad \quad  \quad \quad  \text{otherwise}.\end{aligned}\right.
\end{align}
Therefore, the dual problem of (\ref{constrainedProblem}) is
\begin{align}
&\max_{\bm{\eta}\in \mathbb{R}^N} \: \: \bm{\eta}^{\top}\log\bm{\gamma}_{j}- \bigg\langle\bm{\eta},\log  \frac{\bm{\eta}}{\mu_j} \bigg\rangle \nonumber\\
&\!\!\text{subject to}~~~  \bm{\eta}\succeq 0,\nonumber\\
&\quad \quad \quad ~~~~~\bm{1}^{\top}  \bm{\eta}=1,\nonumber\\
& \quad \quad \quad ~~~~~\bm{A}^{\top}\bm{\eta}=0.
\label{dualProblem}    
\end{align}
The solution to the above entropy maximization problem provides a lower bound for the mappings $\bm{u}\mapsto \mu_{j}\log\langle\bm{\gamma}_{j},\exp\left(\bm{u}/\varepsilon\right)\rangle$ in $\mathbb{R}^N$, thus helping satisfy condition c3.

From Lemma \ref{Lemma_Lipschitz_continuous_gradient},  $\nabla_{\bm{u}_n} g_{n}$ has the Lipschitz constant $L=\dfrac{1}{\varepsilon^2}\big\|\bm{\Gamma}^{\top}\bm{\mu}^{k+1}\big\|_{\infty}$. So, by choosing 
\begin{align*}
    \tau > \frac{\sqrt{2}}{\varepsilon^2}\big\|\bm{\Gamma}^{\top}\bm{\mu}^{k+1}\big\|_{\infty}
\end{align*} we satisfy condition c4. This completes the proof.
\end{proof}
 
%%%%%%%%%%%%%%%%%%%%%%%%%%%%%%%%%%%%%%%%%%%%%%%%%%%%%%%%%%%%%%%%%%%%%%%%%%%%%%%%%%%%%%%%%%%%%%%%%%%%

\section{Details of the Aggregation-Drift-Diffusion Nonlinear PDE Case Study}\label{AppAggregationDriftDiffusionNonlinearPDE}
We choose four different ways of splitting the spatial operators of this nonlinear PDE and present the simulation results for each case of splitting. These choices lead to differently split free energy functionals in our proposed two-layer ADMM, and it is natural to investigate comparative numerical performance due to such variability.

In the first case, we group $\nabla \cdot(\mu \nabla \left(U \oast \mu\right))$ and $\nabla \cdot(\mu \nabla V)$ together as the first term, and $\beta^{-1}\Delta \mu^{2}$ as the second term:
\begin{align*}
    \frac{\partial\mu}{\partial t}= \underbrace{\nabla \cdot(\mu \nabla V) +\beta^{-1}\Delta \mu^{2}}_{i=1}+\underbrace{\nabla \cdot(\mu \nabla U \oast \mu)}_{i=2}.
\end{align*}
In the second case, we group $\nabla \cdot(\mu \nabla \left(U \oast \mu\right))$ and $\beta^{-1}\Delta \mu^{2}$ together as the first term, and $\nabla \cdot(\mu \nabla V) $ as the second term:
\begin{align*}
    \frac{\partial\mu}{\partial t}=\underbrace{\nabla \cdot(\mu \nabla U \oast \mu)+\beta^{-1}\Delta \mu^{2}}_{i=1}+ \underbrace{\nabla \cdot(\mu \nabla V)}_{i=2}.
\end{align*}
In the third case, we group $\nabla \cdot(\mu \nabla U \oast \mu)$ and $\nabla \cdot(\mu \nabla V)$ together as the first term, and $\beta^{-1}\Delta \mu^{2}$ as the second term:
\begin{align*}
    \frac{\partial\mu}{\partial t}=\underbrace{\nabla \cdot(\mu \nabla V) +\nabla \cdot(\mu \nabla U \oast \mu)}_{i=1}+ \underbrace{\beta^{-1}\Delta \mu^{2}}_{i=2}.
\end{align*}
Finally, in the fourth case, we consider $\nabla \cdot(\mu \nabla U \oast \mu)$ as the first term, $\nabla \cdot(\mu \nabla V)$ as the second term, and $\beta^{-1}\Delta \mu^{2}$ as the third term:
\begin{align*}
    \frac{\partial\mu}{\partial t}=\underbrace{\nabla \cdot(\mu \nabla V)}_{i=1}+\underbrace{\nabla \cdot(\mu \nabla U \oast \mu)}_{i=2}+ \underbrace{\beta^{-1}\Delta \mu^{2}}_{i=3}.
\end{align*}

The corresponding $F_i$'s and the pairwise Wasserstein distances between the solutions $\bm{\mu}_{i}^{k}$ and $\bm{\mu}_{j}^{k}$, $i\neq j$, for each case of splitting are given in Table \ref{Table: exp2}. The reported Wasserstein distances are computed by solving the respective Kantorovich linear programs. Table \ref{Table: centralizedexp2} shows a comparison between how long it took for the centralized and proposed Wasserstein ADMM methods to run using the same simulation setup. It also displays the accuracy results by plotting the Wasserstein distances between the centralized and Wasserstein ADMM iterations, based on the known stationary measure. These results provide two clear findings: 
Firstly, the proposed ADMM updates are faster (much faster when using three-way splitting) than the corresponding updates in the centralized approach.
Secondly, as the iterations continue, the proposed algorithm outperforms the centralized method in terms of accuracy, as seen in the improvement of Wasserstein distance to the known stationary solution. In Table \ref{table: different alpha}, we show how the final objective value changes for this case study based on different ADMM barrier parameter values ($\alpha$).
We maintained a constant inner ADMM iteration number of 3 throughout this analysis. We also performed simulations varying the inner ADMM iteration number while keeping $\alpha=12$ fixed. The resulting fluctuations in the final objective value are detailed in Table \ref{table: different inner layer iteration}.

\begin{table}[t!]
\centering
\begin{tabular}{| c | l | c |} 
 \hline
 Case & Functionals & Wasserstein distances \\ [0.5ex] 
 \hline\hline
\#1  & $\begin{array}{lll}
      F_1(\bm{\mu})=\left\langle \bm{V}_{k} +\beta^{-1}  \bm{\mu} ,\bm{\mu}\right\rangle, \\ F_{2}(\bm{\mu}) = \left\langle \bm{U}_{k}\bm{\mu}^{k}, \bm{\mu}\right\rangle\\
      ~\\
      \text{average runtime = 294.06 s}
   \end{array}$ 
  & \raisebox{-0.6\totalheight}{\includegraphics[width=0.525\textwidth]{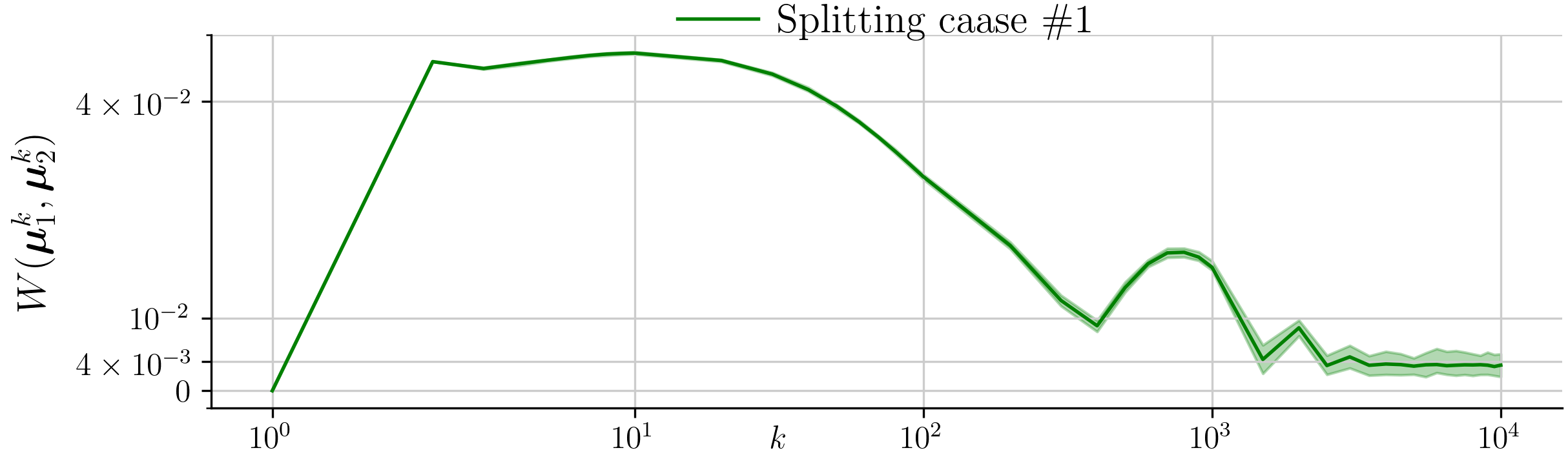}} \\
  \hline
\#2  & $\begin{array}{lll}
      F_{1}(\bm{\mu}) = \langle \bm{U}_{k}\bm{\mu}^{k}+\beta^{-1}  \bm{\mu}, \bm{\mu}\rangle,\\
    F_2(\bm{\mu})=\left\langle \bm{V}_{k} ,\bm{\mu}\right\rangle\\
      ~\\
      \text{average runtime = 285.32 s}
 \end{array}$  & \raisebox{-0.6\totalheight}{\includegraphics[width=0.525\textwidth]{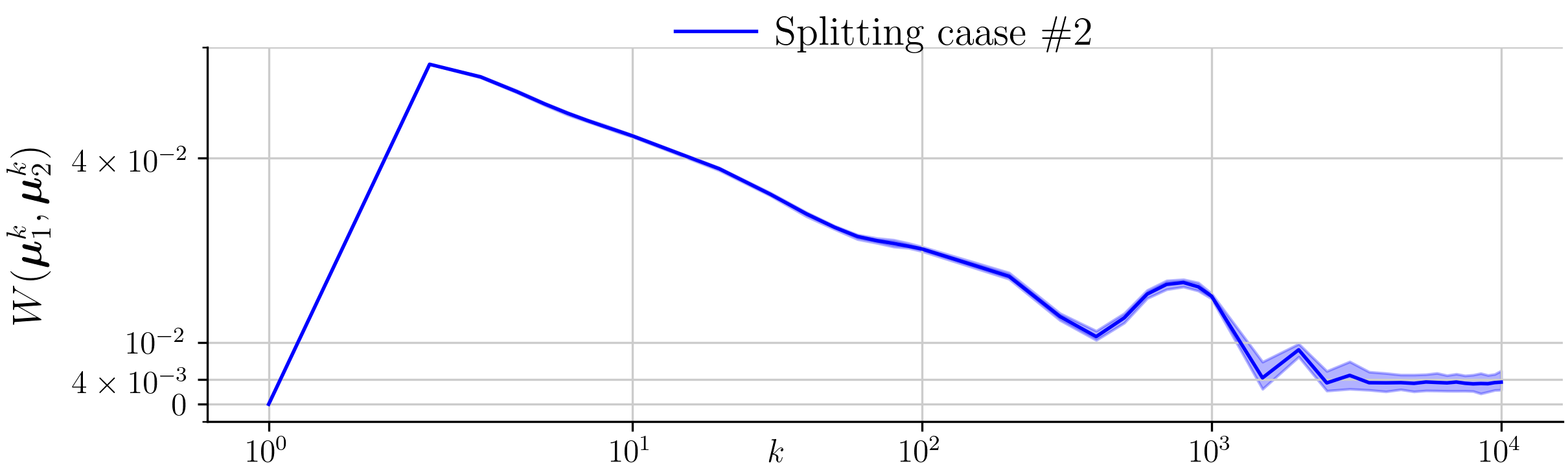}} \\
  \hline
\#3  & $\begin{array}{lll}
       F_{1}(\bm{\mu}) = \langle \bm{U}_{k}\bm{\mu}^{k}+\bm{V}_{k}, \bm{\mu}\rangle,\\
       F_2(\bm{\mu})=\left\langle \beta^{-1}  \bm{\mu} ,\bm{\mu}\right\rangle\\
      ~\\
      \text{average runtime = 289.87 s}
   \end{array}$ & \raisebox{-0.6\totalheight}{\includegraphics[width=0.525\textwidth]{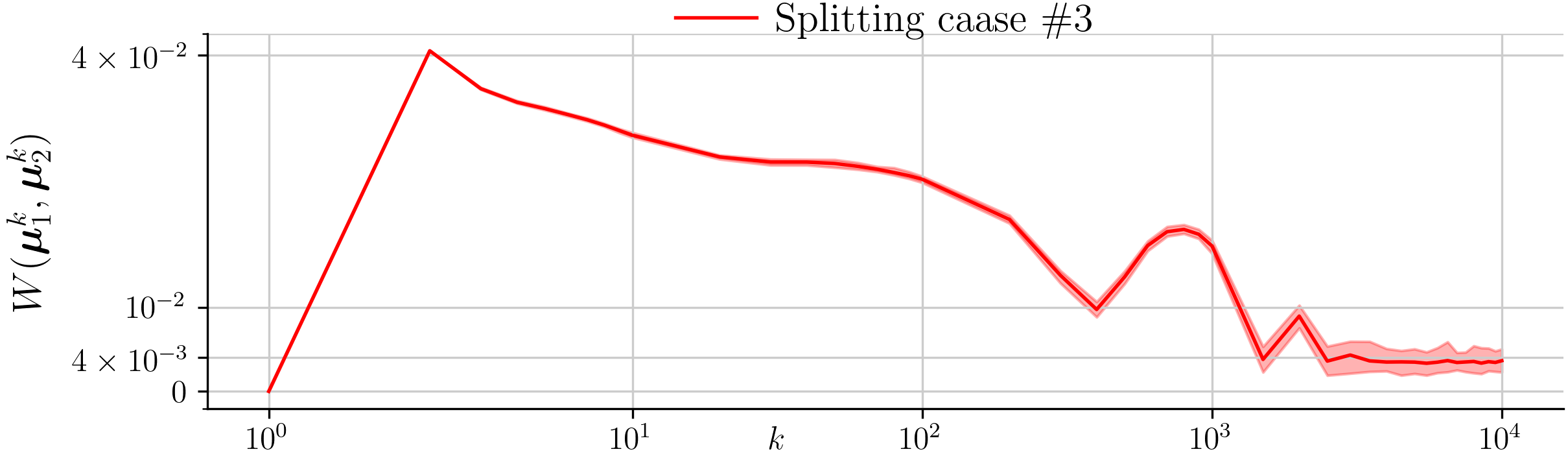}} \\
  \hline
  \#4 & $\begin{array}{lll}
        F_{1}(\bm{\mu}) =\left\langle\bm{V}_{k}, \bm{\mu}\right\rangle,\\ F_{2}(\bm{\mu})=\left\langle \bm{U}_{k}\bm{\mu}^{k} \right\rangle ,\\
        F_3(\bm{\mu})=\left\langle \beta^{-1}  \bm{\mu} ,\bm{\mu}\right\rangle\\
      ~\\
      \text{average runtime = 108.99 s}
      \end{array}$ &
  \raisebox{-0.6\totalheight}{\includegraphics[width=0.525\textwidth]{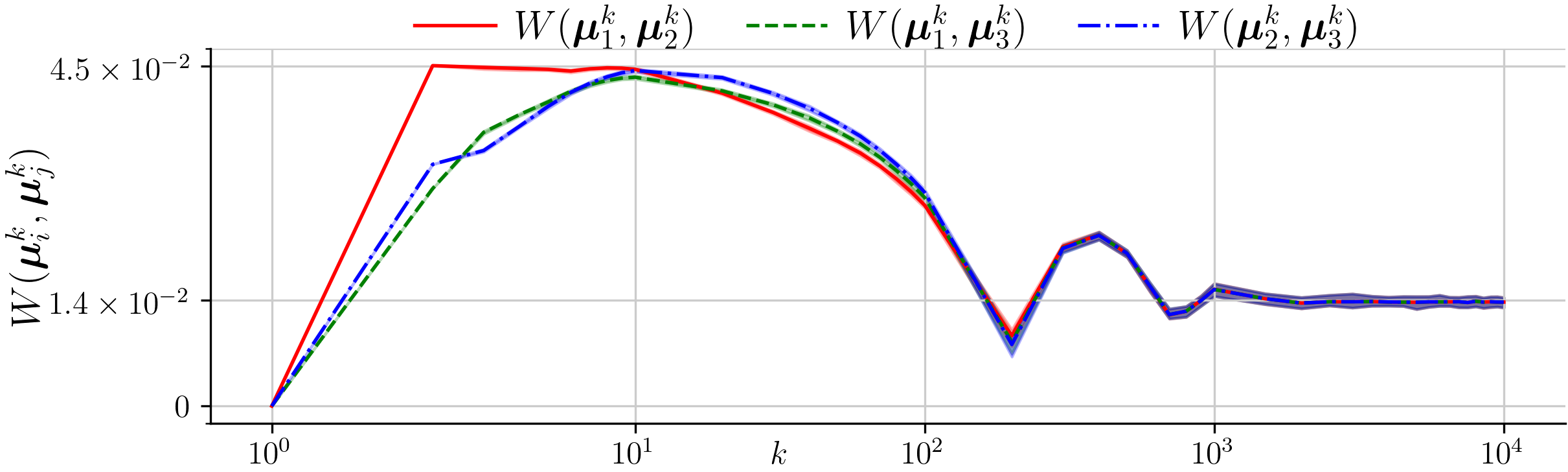}}\\
  \hline
\end{tabular}
\caption{For the aggregation-drift-diffusion nonlinear PDE, the choice of functionals $F_i$, $i\in\{1,2,3\}$ for each case of splitting and the pairwise Wasserstein distances between the solutions $\bm{\mu}_{i}^{k}$ and $\bm{\mu}_{j}^{k}$, $i,j\in\{1,2,3\}$, $i\neq j$, for 100 executions of the code with the same initial samples. In the functional column, the drift potential vector $\bm{V}_{k} \in \mathbb{R}^{N}$ and the symmetric matrix $\bm{U}_{k}\in\mathbb{R}^{N\times N}$ are respectively given by
$
\bm{V}_{k}(i):=V\left(\bm{\theta}_{k}^{i}\right), i\in[N]$ and $\bm{U}_{k}(i, j):=U\left(\bm{\theta}_{k}^{i}-\bm{\theta}_{k}^{j}\right),~ i, j\in[N]$.  We executed the code for each case of splitting 100 times and plot the averaged Wasserstein distance for each splitting case. The figures in the first three rows show the averaged Wasserstein distance of the solution of each term after 10000 iterations for the cases that we split the nonlinear PDE to two terms, and the shadow shows the variation range for each case of splitting. In the last row, each curve shows the averaged Wasserstein distance of the solution of each term after 10000 iterations for the cases that we split the nonlinear PDE to three terms. The shadow shows the variation range of each Wasserstein distances of $\bm{\mu}_1$, $\bm{\mu}_2$, and $\bm{\mu}_3$.
Because we start from the same initial distribution for $\bm{\mu}_{i}$, $i=\{1,2,3\}$, $W(\bm{\mu}_{i}^{k},\bm{\mu}_{j}^{k})$ , $i,j\in\{1,2,3\}$, $i\neq j$ at $k=0$ is zero. In the reported average runtimes, the average is taken over the 100 executions of the same code with the same initial samples.}
\vspace*{-0.1in}
\label{Table: exp2}
\end{table}

\begin{table}
	\hspace*{-.2cm} % Adjust this value as needed
% \begin{adjustwidth}{-.2cm}{}
 % \nointerlineskip\leavevmode
\begin{tabular}{| c | c | c |} 
 \hline
 Case & Functionals & Wasserstein distances \\ [0.53ex] 
 \hline\hline
\#1  & $\begin{array}{lll}
      F_1(\bm{\mu})=\left\langle \bm{V}_{k} +\beta^{-1}  \bm{\mu} ,\bm{\mu}\right\rangle, \\ F_{2}(\bm{\mu}) = \left\langle \bm{U}_{k}\bm{\mu}^{k}, \bm{\mu}\right\rangle
   \end{array}$ 
  & \raisebox{-0.6\totalheight}{\includegraphics[width=0.55\textwidth]{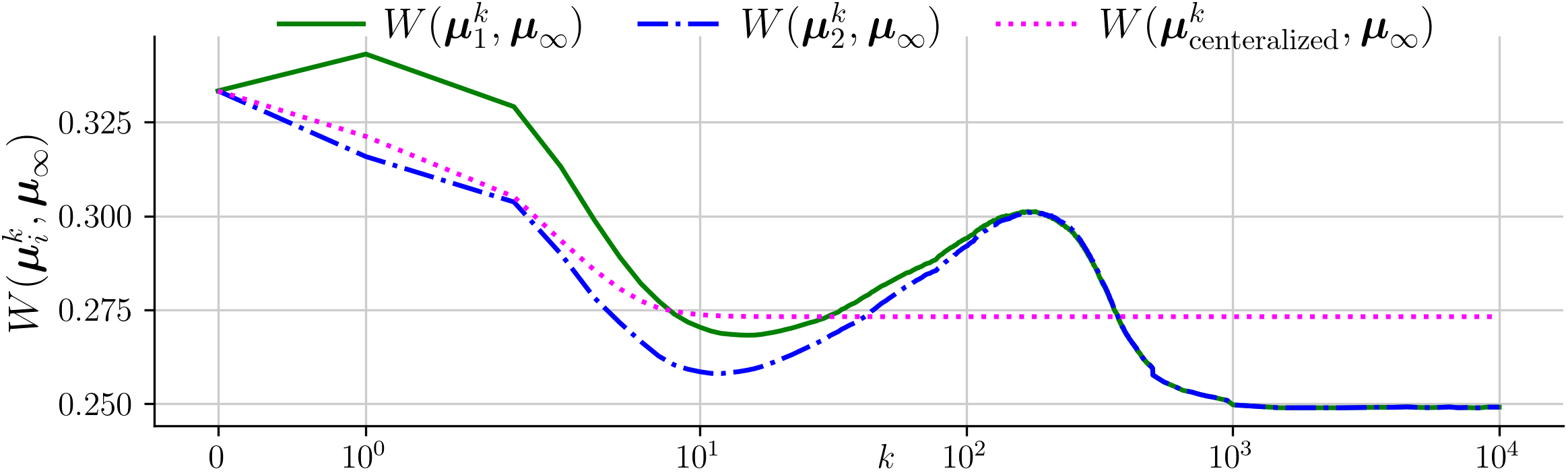}} \\
  \hline
\#2  & $\begin{array}{lll}
      F_{1}(\bm{\mu}) = \langle \bm{U}_{k}\bm{\mu}^{k}+\beta^{-1}  \bm{\mu}, \bm{\mu}\rangle,\\
    F_2(\bm{\mu})=\left\langle \bm{V}_{k} ,\bm{\mu}\right\rangle
 \end{array}$  & \raisebox{-0.6\totalheight}{\includegraphics[width=0.55\textwidth]{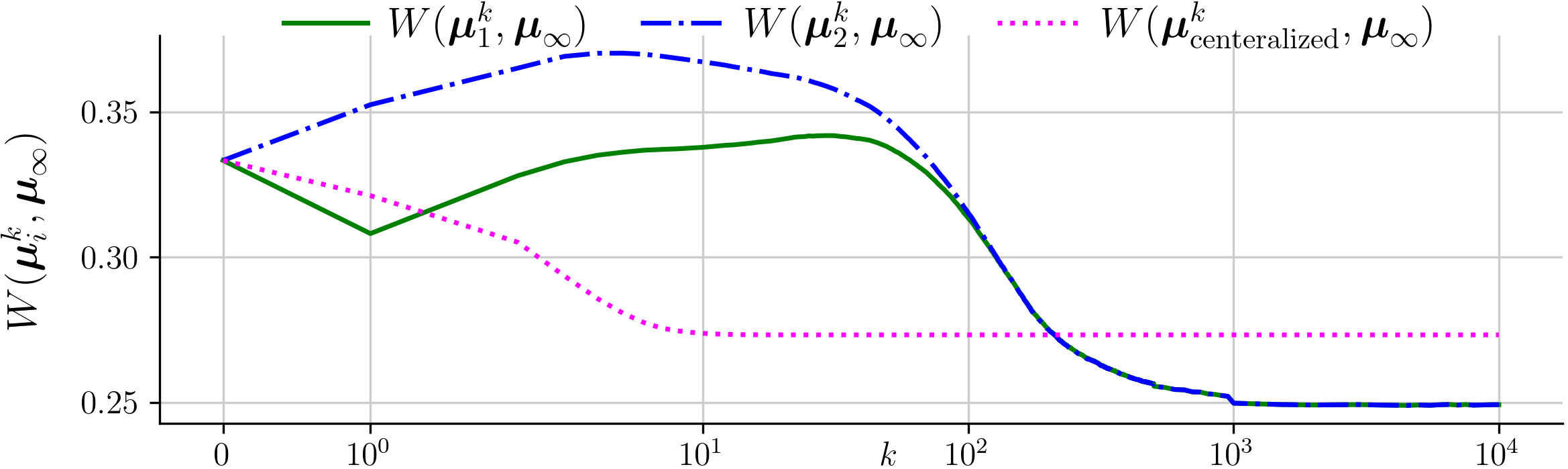}} \\
  \hline
\#3  & $\begin{array}{lll}
       F_{1}(\bm{\mu}) = \langle \bm{U}_{k}\bm{\mu}^{k}+\bm{V}_{k}, \bm{\mu}\rangle,\\
       F_2(\bm{\mu})=\left\langle \beta^{-1}  \bm{\mu} ,\bm{\mu}\right\rangle
   \end{array}$ & \raisebox{-0.6\totalheight}{\includegraphics[width=0.55\textwidth]{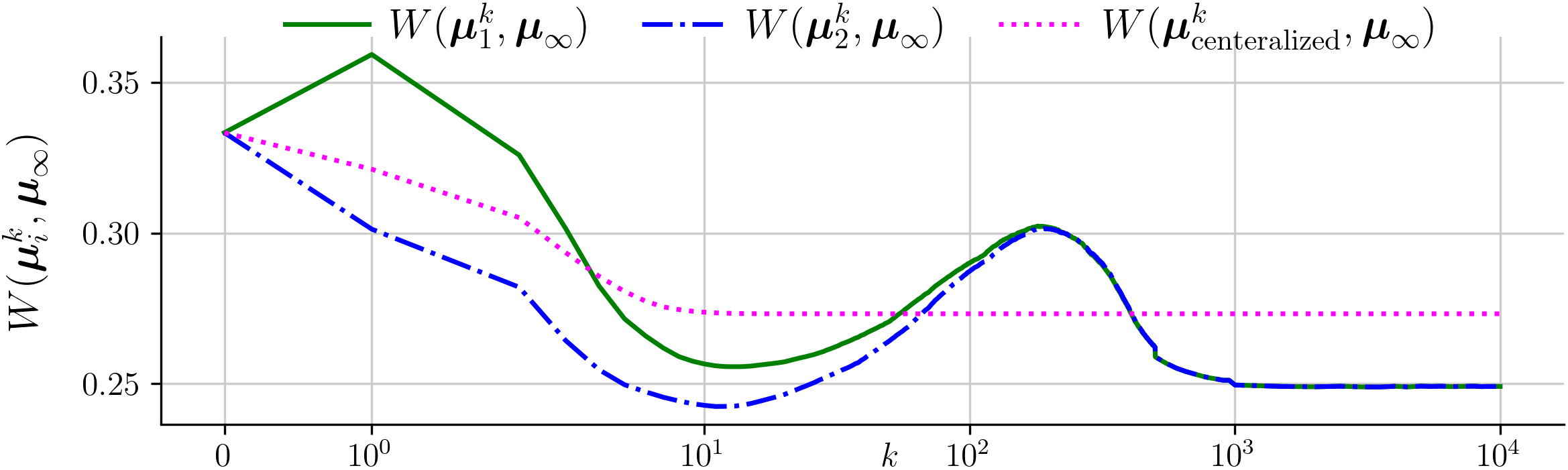}} \\
  \hline
  \#4 & $\begin{array}{lll}
        F_{1}(\bm{\mu}) =\left\langle\bm{V}_{k}, \bm{\mu}\right\rangle,\\ F_{2}(\bm{\mu})=\left\langle \bm{U}_{k}\bm{\mu}^{k},\bm{\mu} \right\rangle ,\\
        F_3(\bm{\mu})=\left\langle \beta^{-1}  \bm{\mu} ,\bm{\mu}\right\rangle
      \end{array}$ &
  \raisebox{-0.6\totalheight}{\includegraphics[width=0.55\textwidth]{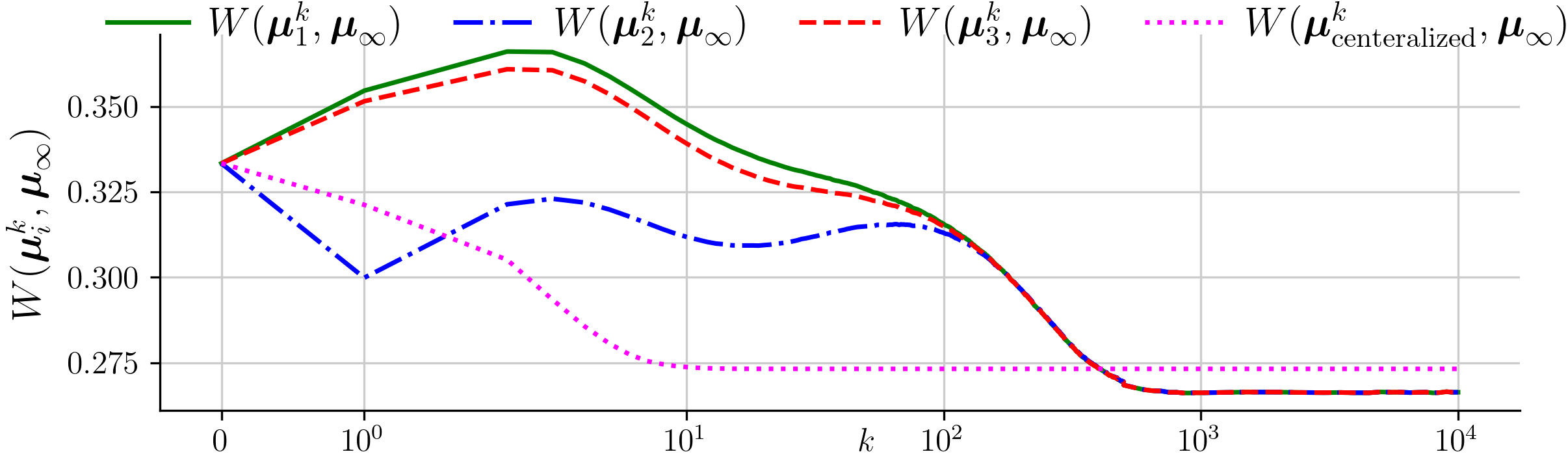}}\\
  \hline
\end{tabular}
%\end{adjustwidth}
\vspace*{0.05in}
\caption{For the aggregation-drift-diffusion nonlinear PDE case study in Sec. V, comparison of the Wasserstein distances to the known stationary solution $\bm{\mu}_{\infty}$, from the iterates of the centralized ($\bm{\mu}_{\text{centralized}}^{k}$), and from the iterates of the proposed Wasserstein ADMM algorithm $\bm{\mu}_{i}^{k}$, $i\in{1,2,3}$. The known $\bm{\mu}_{\infty}$ here is a uniform measure over an annulus with the inner radius $R_{i}=1/2$ and the outer radius $R_{o}=\sqrt{5}/2$ \citep[Sec. 4.3.2]{carrillo2022primal}. All Wasserstein distances are computed by solving the corresponding Kantorovich LPs as in Supp. Material Sec. H, Table 2. All simulations are done with the same set up as in Sec. V and Supp. Material H, Table 2, i.e., with the same uniform grid over $[-2,2]^2$ with 1681 samples, $\beta^{-1}=0.0520$ as in \citep[Sec. 4.3.2]{carrillo2022primal}, and the same $\mu_0, U, V$ and other parameters reported in Sec. V. For centralized computation, we used the proximal recursion algorithm in \citep{caluya2019gradient}. The figures in the last column show that after 10000 iterations, the Wasserstein distances between $\mu_i$ and $\mu_\infty$ in all cases are smaller than the corresponding Wasserstein distance between the centralized solution and $\mu_\infty$. The average runtime (averaged over 100 executions of the same code as in Table 2, Supp. Material) from the proposed Wasserstein ADMM algorithm in all cases remain below $300$ sec, and especially it is recorded at $108.99$ sec in case \#4, significantly below the total runtime of the centralized variant ($310.21$ sec).}
\vspace*{-0.2in}
\label{Table: centralizedexp2}
\end{table}

{\footnotesize{\begin{center}
\begin{table}[htpb]
\hspace*{-2.8cm} % Adjust this value as needed
%\begin{adjustwidth}{-2.8cm}{}
\begin{tabular}{ |c|c|c|c|c|c|c|c|c|c|c|c| } 
\hline 
$\alpha$ & $10$&$10.5$&$11$ &$11.5$&$12$& $12.5$&$13$&$13.5$ &$14$&$14.5$&$15$\\
\hline\hline
 $F^{10000}$, case \#1& $10.8945$&$10.9153$&$10.9058$ &$10.9224$&$10.8978$& $10.9064$&$10.8922$&$10.9203$ &$10.9124$&$10.9203$&$10.9139$\\  
 \hline
 $F^{10000}$, case \#2& $11.0544$&$11.0586$&$11.0624$ &$11.0598$&$11.0618$& $11.0578$&$11.0694$&$11.0692$ &$11.0591$&$11.0570$&$11.0561$\\  
\hline
 $F^{10000}$, case \#3& $11.0282$&$11.0344$&$11.0296$ &$11.0325$&$11.0275$& $11.0312$&$11.0338$&$11.0301$ &$11.0395$&$11.0351$&$11.0305$\\  
 \hline
 $F^{10000}$, case \#4& $16.5034$&$16.5051$&$16.5087$ &$16.5012$&$16.5106$& $16.5080$&$16.5049$&$16.5029$ &$16.5030$&$16.5018$&$16.5057$\\  
 \hline
\end{tabular}
%\end{adjustwidth}
\caption{Value of the objective $F^{10000}:=\langle \bm{V}_{k} + \bm{U}_{k}\bm{\mu}^{k} + \beta^{-1}\bm{\mu},\bm{\mu}\rangle\vert_{k=10000}$ at the final consensus iterate $\bm{\mu}\equiv\bm{\mu}^{10000}$ for cases in Fig. 1 w.r.t. different values of ADMM barrier parameter $\alpha\in[10,15]$.}
\label{table: different alpha}
\end{table}
\end{center}}}

{\footnotesize{\begin{center}
\begin{table}[htpb]
	\hspace*{-1.5cm} % Adjust this value as needed
%\begin{adjustwidth}{-1.5cm}{}
\begin{tabular}{ |c|c|c|c|c|c|c|c|c| } 
\hline 
 Inner layer ADMM iter. \# & $3$&$4$&$5$ &$6$&$7$& $8$&$9$&$10$\\
\hline\hline
 $F^{10000}$, case \#1& $10.9263$&$10.8981$&$10.9165$ &$10.8997$&$10.9124$& $10.9157$&$10.8813$&$10.9009$\\  
 \hline
 $F^{10000}$, case \#2& $11.0638$&$11.0546$&$11.0643$ &$11.0625$&$11.0632$& $11.0583$&$11.0701$&$11.0678$ \\  
\hline
 $F^{10000}$, case \#3& $11.0368$&$11.0457$&$11.0374$ &$11.0381$&$11.0363$& $11.0359$&$11.0318$&$11.0322$ \\  
 \hline
 $F^{10000}$, case \#4& $16.5072$&$16.5023$&$16.5046$ &$16.5001$&$16.5123$& $16.5039$&$16.5045$&$16.5034$ \\  
 \hline
\end{tabular}
%\end{adjustwidth}
\caption{Value of the objective $F^{10000}:=\langle \bm{V}_{k} + \bm{U}_{k}\bm{\mu}^{k} + \beta^{-1}\bm{\mu},\bm{\mu}\rangle\vert_{k=10000}$ at the final consensus iterate $\bm{\mu}\equiv\bm{\mu}^{10000}$ for cases in Fig. 1 w.r.t. different number for the Inner layer ADMM iteration.}
\label{table: different inner layer iteration}
\end{table}
\end{center}}}

%%%%%%%%%%%%%%%%%%%%%%%%%%%%%%%%%%%%%%%%%%%%%%%%%%%%%%%%%%%%%%%%%%%%%%%%%%%%%%%%%%%%%%%%%%%%%%%%%%%%

\section{Grouping of Summand Functionals}\label{AppGrouping}
In (\ref{AdditiveOptimizationMeasure}), $F = F_1 + \hdots + F_n$, $n>1$, where the summand functionals $F_i$, $i\in[n]$, are necessarily \emph{distinct}. Suppose that we have $n$ \emph{indistinguishable} computing elements available for distributed computation. We can use any subset of them to implement our proposed algorithm depending on how we group the $n$ distinct summand functionals. Clearly, the grouping $\{\{F_1,\hdots,F_n\},\{0\},\hdots,\{0\}\}$ corresponds to centralized computation. Then the number of ways to implement our distributed algorithm over $n$ computing elements is
\begin{align}
B_{n} - 1, \quad n=2,3,\hdots, \quad \text{where $B_n$ denotes the $n$th Bell number \citep{bell1938iterated}}.
\label{NumberOfWaysnDistinctSummandsnIdenticalComputers}
\end{align}
The minus one in (\ref{NumberOfWaysnDistinctSummandsnIdenticalComputers}) discounts the centralized computation. The first few Bell numbers are $B_2 = 2, B_3 = 5, B_4 = 15, B_5=52, B_6 = 203, \hdots$. 

For our first experiment in Sec. \ref{sec:Experiments}, $n=2$ and there is $B_{2} - 1 = 1$ way to implement the proposed algorithm. For our second experiment in Sec. \ref{sec:Experiments}, $n=3$ and there are $B_{3} - 1 = 4$ ways to implement the proposed algorithm as detailed in Appendix \ref{AppAggregationDriftDiffusionNonlinearPDE}.

More generally, if we have $n$ distinct summand functionals with $r \leq n$ indistinguishable computing elements available, then the number of ways to implement our distributed algorithm is
\begin{align}
\displaystyle\sum_{k=1}^{r}\!\left\{{n\atop k}\!\right\}\!, \text{where $\left\{{n\atop k}\right\}$ denote the Stirling numbers of second kind \citep[p. 244]{graham1988concrete}.} 
\label{NumberOfWaysnDistinctSummandsrIdenticalComputers}
\end{align}
For $r=n$, (\ref{NumberOfWaysnDistinctSummandsrIdenticalComputers}) reduces to (\ref{NumberOfWaysnDistinctSummandsnIdenticalComputers}).

\end{document}